\documentclass[reqno,11pt]{amsart}

\usepackage[top=2.0cm,bottom=2.0cm,left=3cm,right=3cm]{geometry}
\usepackage{amsthm,amsmath,amssymb,dsfont}
\usepackage{mathrsfs,amsfonts,functan,extarrows,mathtools}
\usepackage[colorlinks]{hyperref}

\usepackage{marginnote}
\usepackage{xcolor}
\usepackage{stmaryrd}
\usepackage{esint}
\usepackage{graphicx}
\usepackage{bm}

\newtheorem{theorem}{Theorem}[section]
\newtheorem{proposition}{Proposition}[section]

\newtheorem{corollary}{Corollary}[section]
\newtheorem{remark}{Remark}[section]
\newtheorem{lemma}{Lemma}[section]

\numberwithin{equation}{section}
\allowdisplaybreaks

\def\d{\mathrm{d}}
\def\no{\nonumber}
\def\R{\mathbb{R}}

\def\eps{\varepsilon}

\def\n{\mathbf{n}}

\newcounter{wronumber}

\begin{document}
	\title[Allen-Cahn-Navier-Stokes equations]
	{On the initial-boundary value problem of two-phase incompressible flows with variable density in smooth bounded domain}
	
	\author[Ning Jiang]{Ning Jiang}
	\address[Ning Jiang]{\newline School of Mathematics and Statistics, Wuhan University, Wuhan, 430072, P. R. China}
	\email{njiang@whu.edu.cn}

	\author[Yi-Long Luo]{Yi-Long Luo}
	\address[Yi-Long Luo]
	{\newline School of Mathematics, South China University of Technology, Guangzhou, 510641, P. R. China}
	\email{luoylmath@scut.edu.cn}
	
	\author[Di Ma]{Di Ma}
	\address[Di Ma]
	{\newline School of Mathematics and Statistics, Wuhan University, Wuhan, 430072, P. R. China}
	\email{dima\_math@whu.edu.cn}
	
	\thanks{ \today}
	
	\maketitle
	
	\begin{abstract}
        In this work, we study the so-called Allen-Cahn-Navier-Stokes equations, a diffuse-interface model for two-phase incompressible flows with different densities. We first prove the local-in-time existence and uniqueness of classical solutions with finite initial energy over the smooth bounded domain $\Omega$. The key point is to transform the boundary values of the higher order spatial derivatives to that of the higher order time derivatives by employing the well-known Agmon-Douglis-Nireberg theory in \cite{ADN-2}. We then prove global existence near the equilibrium $(0, \pm 1)$ and justify the time exponetial decay $e^{- c_\# t}$ of the global solution. The majority is that the derivative $f' (\phi)$ of the physical relevant energy density $f(\phi)$ will generate an additional damping effect under the perturbation $\phi = \varphi \pm 1$.\\
		
		
	\end{abstract}
	
	
	


	\section{Introduction}
	
	\subsection{Two-phase incompressible flows}
 
 In the past two decades, the phase field approach was employed by many researchers in various fluid models \cite{Abels-CMP2009, Abels-M3AS2012, DLL-JMFM2013, Liu-Shen-PD2003, Liu-Shen-Yang-2015-JSC, Shen-Yang-SISC2010, Wu-Xu-CMS2013},  and have carried out extensive analytical and numerical studies on the two-phase flows. For phase field model of incompressible binary fluids with identical densities, or with small density ratios where Boussinesq approximation can be applied in practice, we refer the readers to \cite{Abels-ARMA2009, Boyer-AA1999, GG-AIP2010, GG-DCDS2010, GPV-M3AS1996, Wu-Xu-CMS2013, Xu-Zhao-Liu-SIMA2010} and references therein for detailed derivations and mathematical analysis. However, in most cases, the density differences of the components are not negligible, whence studies on the incompressible two-phase fluids with non-matched densities become even more interesting and challenging.
 
 Recently, in \cite{Jiang-Li-Liu-2017-DCDS}, by employing the {\em energetic variational approach}, a new two-phase incompressible flows with variable density was derived. In fact, before \cite{Jiang-Li-Liu-2017-DCDS},  analytical and numerical results for two-phase incompressible flow models that are valid for large density ratios between different species are quite limited \cite{Abels-AIP2013, Liu-Shen-Yang-2015-JSC, Shen-Yang-SISC2010}. Most studies of the phase-field models for binary fluids have been restricted to the cases with the same density or with small density differences. In the latter case, Boussinesq approximation can be used, where the variable density is replaced by a background constant density value and external gravitational force is added to model the effect of density force \cite{Liu-Shen-PD2003, Wu-Xu-CMS2013}.  In the model derived in \cite{Jiang-Li-Liu-2017-DCDS},  the density ratio between two phases can be quite large and hence the Boussinesq approximation is no more physically valid.

 In this paper, we study the hydrodynamics of a diffuse-interface model describing a mixture of two immiscible incompressible fluids in the smooth bounded domain $\Omega \subseteq \R^3$ with different densities $\rho_1$ and $\rho_2$. A phase field $\phi (t,x)$ is introduced to characterize the two fluids such that
	\begin{center}
		$\phi(t,x)=
		\begin{cases}
			\quad 1 \,,  & \textrm{fluid} \ \ 1 \ \textrm{with density } \rho_1 \,, \\
			\quad -1 \,, &\textrm{fluid} \ \ 2 \ \textrm{with density } \rho_2\,,\\
		\end{cases}$
	\end{center}
	with a thin, smooth transition region. While the two fluids are mixed, $\phi (t,x)$ will ranged be in $(-1,1)$. More precisely, we study the following Allen-Cahn-Navier-Stokes (briefly, ACNS) system:
	\begin{equation}\label{equ-1}
		\left\{
		\begin{array}{l}
			\rho(\phi)( u_t +u\cdot\nabla u)+\nabla p =\nabla\cdot(\mu\nabla u-\lambda \nabla \phi \otimes \nabla \phi) \,,  \\[2mm]
			\qquad \qquad \nabla\cdot u =0 \,,   \\[1mm]
			(\dot{\phi} = )\, \phi_t +u\cdot\nabla \phi =\gamma(\lambda\Delta\phi-\lambda f'(\phi)-\rho'(\phi)\frac{\mid u\mid^{2}}{2}) \,.
		\end{array}
		\right.
	\end{equation}
	Here, $\phi : \mathbb{R}^+ \times  \Omega \rightarrow \mathbb{R} $ is the phase field function that labels different species, $u : \R^+ \times \Omega \rightarrow \R^3$ denotes the velocity of the fluid, and $p : \R^+ \times \Omega \rightarrow \R $ stands for the pressure. $\rho(\cdot)$ is the average density which is a given function of $\phi$. $\mu > 0$ is the viscosity, $\lambda > 0$ represents the competition between the kinetic energy and the free energy, and $\gamma > 0$ comes from microscopic internal damping during the mixing of two immiscible incompressible fluids. $\dot{\phi}$ is the material derivative of $\phi$ with respect to the velocity $u$. The physical relevant energy density functional $ f $ that represent the two phases of the mixture usually has a double-well structure. Without loss of generality, in current paper, we assume that
	\begin{equation}\label{equ-2}
		f(\phi) = \tfrac{1}{4\varepsilon^2}(\phi^2 - 1)^2
	\end{equation}
	with some small parameter $\varepsilon > 0$. For the average density function $\rho (\cdot)$, it is generally assumed that (see \cite{Jiang-Li-Liu-2017-DCDS})
	\begin{equation}\label{Average-rho-1}
		\begin{aligned}
			\rho (\cdot ) \in C^1 (\R)\,, \ \rho (-1) = \rho_1\,, \ \rho(1) = \rho_2 \,, \ \rho (s) \in [ \rho_1 , \rho_2 ] \ \textrm{for } -1 \leq s \leq 1
		\end{aligned}
	\end{equation}
	with $\rho_1 < \rho_2$ being two positive constants and the following exterior convexity
	\begin{equation}\label{Average-rho-2}
		\begin{aligned}
			s \rho' (s) \geq 0 \,,  \ \textrm{for } |s| > 1 \,.
		\end{aligned}
	\end{equation}
	In current paper, we assume that the $\rho (\cdot) \in C^2 (\R)$ is the parabolic average of $\rho_1$ and $\rho_2$ satisfying \eqref{Average-rho-1}-\eqref{Average-rho-2} as follows (see \cite{Liu-Shen-Yang-2015-JSC}):
	\begin{equation}\label{rho-def}
		\begin{aligned}
			\rho (\phi ) = \tfrac{1}{4} \rho_1 (\phi - 1)^2 + \tfrac{1}{4} \rho_2 ( \phi + 1 )^2 \,.
		\end{aligned}
	\end{equation}

    Let $\n = \n (x) \,, x \in \partial \Omega$ be the outnormal vector of the boudnary $\partial \Omega$. We now impose the boundary conditions:
    \begin{equation}\label{BC-ACNS}
    	u|_{\partial\Omega} = 0 \,, \quad \tfrac{\partial}{\partial \n } \phi|_{\partial \Omega}=0,
    \end{equation}
    and the initial data:
    \begin{equation}\label{equ-3}
    	u(0,x)= u^{in}(x) \in \Omega, \quad  \phi(0,x)= \phi^{in}(x) \in \R
    \end{equation}
    with the compatibility condition $\nabla \cdot u^{in}=0$.

    As shown in \cite{Jiang-Li-Liu-2017-DCDS}, once $-1 \leq \phi^{in} (x) \leq 1$ initially in $\Omega$, the Maxmal Principle of the heat equation implies that $-1 \leq \phi(t,x) \leq 1$ for $t \geq 0$ and $x \in \Omega$. We therefore know that the average density $\rho (\phi)$ admits the lower and upper bounds, namely,
	\begin{equation}\label{Low-Bnd-rho-phi}
		\begin{aligned}
			\rho (\phi) = \tfrac{\rho_1 + \rho_2}{4} \Big( \phi + \tfrac{\rho_2 - \rho_1}{\rho_2 + \rho_1} \Big)^2 + \tfrac{\rho_1 \rho_2}{\rho_1 + \rho_2} \in [\tfrac{\rho_1 \rho_2}{\rho_1 + \rho_2}, \tfrac{2\rho_1 \rho_2 + \rho_2^2}{2(\rho_1 + \rho_2)}]
		\end{aligned}
	\end{equation}
	for all $\phi \in [-1, 1]$.

	The system \eqref{equ-1} was derived from employing the {\em energetic variational approach} by Jiang-Li-Liu \cite{Jiang-Li-Liu-2017-DCDS}, in which they considered the total energy
	\begin{equation*}
		\begin{aligned}
			E^{total} : = \int_{\Omega} \Big( \tfrac{1}{2} \rho (\phi) |u|^2 + \lambda ( \tfrac{1}{2} |\nabla \phi|^2 + f (\phi) ) \Big) \d x \,.
		\end{aligned}
	\end{equation*}
	It consists of the first part of the macroscopic kinetic energy and the second part of the Helmholtz free energy. They also took the dissipation of the energy as
	\begin{equation*}
		\begin{aligned}
			\triangle^{dissipative} : = \int_\Omega \Big( \mu |\nabla u|^2 + \tfrac{1}{\gamma} |\dot{\phi}|^2 \Big) \d x \,,
		\end{aligned}
	\end{equation*}
	where the first part accounts for the macroscopic dissipation due to viscosity and the second part comes from microscopic internal damping during the mixing. Finally, the so-called energetic variational approach implies the system \eqref{equ-1}. The details can be seen in \cite{Jiang-Li-Liu-2017-DCDS}.

\subsection{Notations and main results}

In the sequel, we consider the smooth bounded domain $\Omega$ in $\R^3$. We first denote by $L^p \, (1 \leq p \leq \infty)$ by the standard Lebesgue space with norm
\begin{equation*}
	\begin{aligned}
		\| f \|_{L^p} = \Big( \int_{\Omega} |f|^p d x \Big)^\frac{1}{p} \, (1 \leq p < \infty) \,, \quad \| f \|_{L^\infty} = \underset{x \in \Omega}{\mathrm{ess \ sup }} \, |f (x)| \,.
	\end{aligned}
\end{equation*}
For simplicity, we denote by $\| f \| : = \| f \|_{L^2}$. Let $W^{m,p} $ be the standard Sobolev space with norm $\| f \|_{W^{m,p}}^2 = \sum_{|\alpha| \leq m} \| \partial^\alpha f \|^2_{L^p}$.  Here $\alpha = ( \alpha_1, \alpha_2, \alpha_3 ) \in \mathbb{N}^s$ with $|\alpha| = \alpha_1 + \alpha_2 + \alpha_3$ and
\begin{equation*}
	\begin{aligned}
		\partial^m f = \tfrac{\partial^{|\alpha|} f }{\partial x_1^{\alpha_1} \partial x_x^{\alpha_2} \partial x_3^{\alpha_3} } \,.
	\end{aligned}
\end{equation*}
As usual, $W^{m,2} $ simply denotes by $H^m $ with norm $\| f \|_m : = \| f \|_{H^m} $. In particular, $\| f \|_0 = \| f \|$. Moreover, we introduce a weighted $L^2$ space by
\begin{equation*}
	\begin{aligned}
		\| f \|_{ L_{\rho(\phi)}^2 } = \Big( \int_\Omega \rho(\phi) | f |^2 d x \Big)^\frac{1}{2} \,.
	\end{aligned}
\end{equation*}
Remark that the corresponding vector-valued Lebesgue and Sobolev spaces will still be expressed by $L^p$, $W^{m,p}$, $H^m$ and $L_{\rho(\phi)}^2$, etc.

We also employ  $A \lesssim B$ to denote by $A \leq C B$ for some harmless constant $C > 0$. Moreover, $A \thicksim B$ means that $C_1 B \leq A \leq C_2 B$ for two harmless constants $C_1, C_2 > 0$.

We then introduce the following energy functional $\mathbb{E}_j (t)$ and dissipation functional $\mathbb{D}_j (t)$ for $j \geq 0$,
\begin{equation}\label{EjDj-0}
	\begin{aligned}
		\mathbb{E}_j (t)  = &  \| \partial_t^{j} u \|_{L^2_{\rho(\phi)}}^2 + \| \nabla \partial_t^{j} u \|^2 + \| \partial_t^{j} \phi \|^2_2 \,, \\
		\mathbb{D}_j(t)  = & \| \nabla \partial_t^{j} u \|^2_1 + \| \partial_t^j u_t \|_{L^2_{\rho(\phi)}}^2 + \| \nabla \partial_t^{j} \phi \|^2_2 +  \| \partial_t^j \phi_t \|^2_1 + \| \partial_t^j p \|^2_1 \,.
	\end{aligned}
\end{equation}

\begin{theorem}[Local well-posedness]\label{th1}
	Let integer $\Lambda \geq 2$ and $\Omega \subseteq \R^3$ be a smooth bounded domain. Assume that the initial data satisfy $-1 \leq \phi^{in} (x) \leq 1$ in $\Omega$ and
	$$ E^{in}_\Lambda := \| u^{in} \|_{{2 \Lambda + 1}}^2 + \| \phi^{in} \|_{{2 \Lambda + 2}}^2 < \infty \,. $$
	Then there exists a $T>0$, depending only on $E^{in}_\Lambda$, $\Omega$, $\Lambda$ and the all coefficients, such that the ACNS system \eqref{equ-1}-\eqref{equ-2} with boudnary conditions \eqref{BC-ACNS} admits a unique solution $(u, p, \phi) (t,x)$ satisfying
\begin{equation*}
	\begin{aligned}
		& \partial_t^\ell u \in L^\infty (0, T; H^{\Lambda - \ell + 1}) \cap L^2 (0, T; H^2) \,, \partial_t^\ell u_t \in L^2 (0, T; L^2_{\rho (\phi)}) \,, \\
		& \partial_t^\ell p \in L^\infty (0, T; H^{\Lambda - \ell}) \,, \ \partial_t^\ell \phi \in L^\infty (0, T; H^{\Lambda - \ell + 2}) \,, \partial_t^\ell \phi_t \in L^2 (0, T; H^1)
	\end{aligned}
\end{equation*}
for $0 \leq \ell \leq \Lambda$. Moreover, the following energy inequality
\begin{equation}
	\begin{split}
		\sum_{j=0}^\Lambda \mathbb{E}_j (t) + \sum_{j=0}^\Lambda \int_0^t \mathbb{D}_j (t') d t' \leq C
	\end{split}
\end{equation}
and
\begin{equation}\label{Energy-Bnd1}
	\begin{aligned}
		\sum_{\ell + s \leq \Lambda} \Big( \| \partial_t^\ell u \|^2_{s+1} + \| \partial_t^\ell p \|^2_s + \| \partial_t^\ell \phi \|^2_{s+2} \Big) (t) \leq C
	\end{aligned}
\end{equation}
hold for any $t \in [0, T]$ and some constant $C>0$, depending only on $E^{in}_\Lambda$, $T$, $\Omega$ and all the coefficients.
\end{theorem}

\begin{remark}
	The equations \eqref{equ-1} and the energy bound \eqref{Energy-Bnd1} indicate that $\partial_t^\ell u \thicksim \Delta^\ell u$ and $\partial_t^\ell \varphi \thicksim \Delta^\ell \varphi$, which implies that when $\ell = \Lambda$, there is a constant $C' > 0$ such that
	\begin{equation}
		\begin{aligned}
			\sup_{t \in [0, T]} \big( \| u \|^2_{2 \Lambda + 1} + \| \varphi \|^2_{2 \Lambda + 2} \big) (t) \leq C' \,.
		\end{aligned}
	\end{equation}
\end{remark}

The next theorem is to prove the global well-posedness of \eqref{equ-1}-\eqref{equ-3} near the equilibrium $(0, \pm 1)$. More precisely, let $\phi (t,x) = \pm 1 + \varphi (t,x)$. We then know that $(u, p, \varphi)$ satisfies
\begin{equation}\label{g-equ-1}
	\left\{
	\begin{array}{l}
		\varrho (\varphi) ( u_t + u \cdot \nabla u) + \nabla p
		= \nabla \cdot ( \mu \nabla u - \lambda \nabla \varphi \otimes \nabla \varphi ) \,,\\[2mm]
		\qquad \qquad \nabla \cdot u =0 \,,\\
		\varphi_t + u \cdot \nabla \varphi + \tfrac{2 \gamma \lambda}{\eps^2} \varphi = \gamma ( \lambda \Delta \varphi - \lambda h ( \varphi ) - \varrho' ( \varphi ) \frac{|u|^2}{2})
	\end{array}
	\right.
\end{equation}
with the boundary conditions
\begin{equation}\label{g-equ-3}
	u|_{\partial\Omega}=0, \quad \tfrac{\partial}{\partial \n} \varphi|_{\partial\Omega}=0,
\end{equation}
and initial data
\begin{equation}\label{g-equ-2}
	\begin{aligned}
		u (0,x) = u^{in} (x) \in \R^3 \,, \quad \varphi (0,x) = \varphi^{in} (x) : = \phi^{in} (x) \pm 1 \in \R \,,
	\end{aligned}
\end{equation}
which satisfies the compatibility condition $\nabla \cdot u^{in}=0$. Here $h (\varphi) = \tfrac{1}{\eps^2} (\varphi^3 \pm 3 \varphi^2)$ and
\begin{equation}
	\begin{aligned}
		\varrho (\varphi) : = \rho ( \varphi \pm 1 ) =
		\left\{
		\begin{array}{l}
			\tfrac{\rho_1}{4} \varphi^2 + \tfrac{\rho_2}{4} ( \varphi + 2 )^2 \quad \textrm{for  } \varphi + 1 \,, \\[2mm]
			\tfrac{\rho_1}{4} ( \varphi - 2 )^2 + \tfrac{\rho_2}{4} \varphi^2 \quad \textrm{for  } \varphi - 1 \,.
		\end{array}
		\right.
	\end{aligned}
\end{equation}
We now introuce the global energy functionals and dissipations as follows: for $j \geq 0$,
\begin{equation}\label{GED-j-order}
	\begin{split}
		\mathds{E}_j(t) = & \| \partial_t^{j} u \|_{L^2_{\varrho(\varphi)}}^2 + \| \nabla \partial_t^{j} u \|^2 + \| \partial_t^{j} \varphi \|^2_2 \,, \\
		\mathds{D}_j (t) = & \| \nabla \partial_t^{j} u \|^2_1 + \| \partial_t^j u_t \|_{L^2_{\varrho(\varphi)}}^2 + \| \partial_t^{j} \varphi \|^2_3 + \| \partial_t^j \varphi_t \|^2_1 + \| \partial_t^j p \|^2_1 \,.
	\end{split}
\end{equation}

\begin{theorem} [Global stability near $(0, \pm 1)$]\label{th2}
		Let integer $\Lambda \geq 2$. Assume that $u^{in} \in H^{2 \Lambda + 1}$ and $\varphi^{in} \in H^{2 \Lambda + 2}$. Then there is a small positive constant $\upsilon_0$, depending only on $\Lambda$, $\Omega$ and the all coefficients, such that if the initial energy
	\begin{equation*}
		\begin{split}
			\mathcal{E}^{in}_\Lambda : = \| u^{in} \|_{2 \Lambda + 1}^2 + \| \nabla \varphi^{in} \|_{2 \Lambda + 2}^2 \leq \upsilon_0 \,,
		\end{split}
	\end{equation*}
	then \eqref{g-equ-1}-\eqref{g-equ-3} with initial data \eqref{g-equ-2} admits a unique global in time solution $(u, p, \varphi)$ satisfying
	\begin{equation*}
		\begin{aligned}
			& \partial_t^\ell u \in L^\infty (0, \infty; H^{\Lambda - \ell + 1}) \cap L^2 (0, \infty; H^2) \,, \partial_t^\ell u_t \in L^2 (0, \infty; L^2_{\rho (\phi)}) \,, \\
			& \partial_t^\ell p \in L^\infty (0, \infty; H^{\Lambda - \ell}) \,, \ \partial_t^\ell \phi \in L^\infty (0, \infty; H^{\Lambda - \ell + 2}) \,, \partial_t^\ell \phi_t \in L^2 (0, \infty; H^1)
		\end{aligned}
	\end{equation*}
	for $0 \leq \ell \leq \Lambda$.
	Furthermore, there hold
	\begin{equation}
		\begin{aligned}
			\sum_{0 \leq k \leq \Lambda} \mathds{E}_k (t) \leq c_0 \mathcal{E}_\Lambda^{in} e^{- c_\# t} \,, \quad \sum_{0 \leq k \leq \Lambda} \int_{0}^t \mathds{D}_k (t') \d t' \leq c_0 \mathcal{E}_\Lambda^{in} \,,
		\end{aligned}
	\end{equation}
    and
    \begin{equation}\label{Energy-Bnd-2}
    	\sum_{\ell + s \leq \Lambda} \Big( \| \partial_t^\ell u \|^2_{s+1} + \| \partial_t^\ell p \|^2_s + \| \partial_t^\ell \varphi \|^2_{s+2} \Big) (t) \leq c_1 \mathcal{E}_\Lambda^{in} e^{- c_\# t}
    \end{equation}
    for all $t \geq 0$ and for some positive constants $c_0, c_1, c_\# > 0$, depending only on $\Lambda$, $\Omega$ and the all coefficients.
\end{theorem}	

\begin{remark}
	Together with the equations \eqref{g-equ-1}, the energy bound \eqref{Energy-Bnd-2} guarantees that $\partial_t^\ell u \thicksim \Delta^\ell u$ and $\partial_t^\ell \varphi \thicksim \Delta^\ell \varphi$. It thereby infers that for $\ell = \Lambda$,
	\begin{equation}
		\begin{aligned}
			\big( \| u \|^2_{2 \Lambda + 1} + \| \varphi \|^2_{2 \Lambda + 2} \big) (t) \leq c_2 \mathcal{E}_\Lambda^{in} e^{- c_\# t}
		\end{aligned}
	\end{equation}
	for all $t \geq 0$ and for some positive constants $c_0, c_2, c_\# > 0$.
\end{remark}

\subsection{Main ideas and sketch of the proofs}	

The first goal of this paper is to prove the local well-posedness of the initial-boundary value problem \eqref{equ-1}-\eqref{equ-3} over the smooth bounded domain $\Omega$ in the functions spaces that the spatial variables with regularity $H^s$ for large index $s$, i.e., Theorem \ref{th1}. In many known literatures that considered the well-posedness of various models, only the period domain or whole space was considered in the spaces with spatial regularity $H^s$ for large $s$. Oppositely, if the bounded domains were focused, the aims were to prove the existence of weak solutions or strong solutions (in $H^2$ space, for example), in which cases only the information of boundary values was required, rather than that of higher order spartial derivatives of the boundary values.

In the smooth bounded domain $\Omega$, if one investigates the well-posedness of the ACNS model in the functions spaces with $H^s$-regularity of the spatial variables, the {\em main difficulties} come from dealing with the boundary values of the higher order spatial derivatives. Generally speaking, although the boundary value of a function is finite, the higher order derivatives may be infinite. For example, the function $f (x) = \sqrt{x}$ for $x \geq 0$ is continuous up to the boundary $x = 0$ but $f' (0) = \infty$. Note that it is impossible to control the boundary values of higher order spatial derivatives only employing the usual $H^s$-theory over the period domain or whole space. For instance,
\begin{equation*}
	\begin{aligned}
		- \int_\Omega \Delta \partial^m u \cdot \partial^m u d x = - \int_{\partial \Omega} \tfrac{\partial}{\partial \n} \partial^m u \cdot \partial^m u d S + \int_\Omega |\nabla \partial^m u|^2 d x \,,
	\end{aligned}
\end{equation*}
the boundary integral $\int_{\partial \Omega} \tfrac{\partial}{\partial \n} \partial^m u \cdot \partial^m u d S$ cannot be controlled by the ``good" quantity $\int_\Omega |\nabla \partial^m u|^2 d x$ combining with the Trace Theorem.

In order to overcome the difficulty, we employ the Agmon-Douglis-Nirenberg (briefly, ADN) theory associated with the general elliptic system in \cite{ADN-2}, which was reviewed in Section \ref{Sec:ADN} below. The {\em main ideas} are as follows. The ACNS system \eqref{equ-1} with boundary conditions \eqref{BC-ACNS} can be rewritten as the abstract forms
\begin{equation*}
	\left\{
	    \begin{aligned}
	    	- \mu \Delta u + \nabla p = \mathcal{U} (u_t, u, \phi) \,, \quad \nabla \cdot u = 0 \,, \\
	    	\Delta \phi = \Theta (\phi_t, u, \phi) \,, \\
	    	u |_{\partial \Omega} = 0 \,, \quad \tfrac{\partial \phi}{\partial \n} |_{\partial \Omega} = 0 \,.
	    \end{aligned}
	\right.
\end{equation*}
By the ADN theory, $\| (u, \phi) \|_{s+2}$ can be bounded by $\| (u, \phi)_t \|_s$, namely, the second order spatial derivatives of $(u, \phi)$ can be transformed to the first order time derivatives of $(u, \phi)$. We therefore mutate the higher order spatial derivatives problem to the higher order time derivatives problem.

The {\em key point} to dominate the higher order time derivatives is that the boundary condition \eqref{BC-ACNS}, i.e., $u |_{\partial \Omega} = 0$ and $\tfrac{\partial \phi}{\partial \n} |_{\partial \Omega} = 0$ can imply the boundary conditions \eqref{BC-HighTimeD} of the higher order time derivatives, i.e., $\partial_t^k u |_{\partial \Omega} = 0$ and $\tfrac{\partial }{\partial \n} \partial_t^k \phi |_{\partial \Omega} = 0$ for any $k \geq 0$. The sketch of the proofs as follows.
\begin{enumerate}
	\item In Section \ref{Sec:ADN} below, we first review the general ADN theory, and reduce the special forms of estimates in Lemma \ref{Lmm-Stokes} and Lemma \ref{Lmm-Laplace} required in this paper.
	
	\item In Subsection \ref{Subsec:3_2}, we derive the closed $H^2$-estimates for the ACNS system \eqref{equ-1}. Here the boundary conditions $u |_{\partial \Omega} = 0$ and $\tfrac{\partial \phi}{\partial \n} |_{\partial \Omega} = 0$ in \eqref{BC-ACNS} are required.
	
	\item In Subsection \ref{Subsec:3_3}, we derive the closed $H^2$-estimates for the ACNS system \eqref{equ-1} after acting the higher order time derivatives operator $\partial_t^k$ for $k \geq 1$. Here the boundary conditions $\partial_t^k u |_{\partial \Omega} = 0$ and $\tfrac{\partial}{\partial \n} \partial_t^k \phi |_{\partial \Omega} = 0$ in \eqref{BC-HighTimeD} below are required.
	
	\item Based on closed $H^2$-bounds of the various orders of the time derivatives in Step (2) and (3), we apply the ADN theory in Lemma \ref{Lmm-Stokes} and Lemma \ref{Lmm-Laplace} to dominate the higher order time-spatial mixed derivatives, see Lemma \ref{Lmm-HSDE}. Here the core is to control the quantities $\bm{\Phi}_{\ell, s}$ and $\mathbf{U}_{\ell, s}$ as in Corollary \ref{Coro-UPhi-bnd}.
\end{enumerate}

The second goal of this paper is to investigate the global stability and long time decay of the ACNS system \eqref{equ-1} near the equilibrium $(0, \pm 1)$, hence, Theorem \ref{th2}. In this case, under the purturbation $\phi (t,x) = \varphi (t,x) \pm 1$, we reduce the equivalent system \eqref{g-equ-1}. The way to deal with the information of the boundary values is totally the same as that in proving the locall well-posedness to \eqref{equ-1}. The key ingredients to verify the global stability is to seek a new dissipation or damping structure on the unknown $\varphi (t,x)$. Fortunately, by the perturbation $\phi = \varphi \pm 1$, the derivative $f' (\phi) = f' (\varphi \pm 1)$ of the physical relevant energy density $f (\phi)$ in the $\phi$-equation of \eqref{equ-1} will generate an additional damping effective $\tfrac{2}{\eps^2} \varphi$, which gaurantees us to prove the global existence of the ACNS system near the equilibrium $(0, \pm 1)$.

On the other hand, due to the boundary conditions $\partial_t^j u |_{\partial \Omega} = 0$ for $j \geq 0$, there holds the Poincar\'e inequality $\| \partial_t^j u \| \lesssim \| \nabla \partial_t^j u \|$. We thereby imply that
\begin{equation*}
	\begin{aligned}
		\mathds{E}_j (t) \lesssim \mathds{D}_j (t) \quad (\forall \, j \geq 0) \,,
	\end{aligned}
\end{equation*}
where $\mathds{E}_j (t)$ and $\mathds{D}_j (t)$ are defined in \eqref{GED-j-order}. Then by the arguments in the end of Section \ref{Sec:Global}, we can justify the global solution near $(0, \pm 1)$ admits the exponetial decay $e^{- c_\# t}$ for some constant $c_\# > 0$.

\subsection{Organization of this paper}	

In the next section, we give some preliminaries, in particular review the general ADN theory. In Section \ref{Sec:Apriori}, we derive three types of the a priori estimates of the system \eqref{equ-1}: 1) $H^2$-estimates; 2) $H^2$-estimates of the higher order time derivatives; 3) The estimates for the higher order time-space mixed derivatives. In Section \ref{Sec:Local}, based on the a priori estimates, we prove the local well-posedness by employing the iteration methods. In Section \ref{Sec:Global}, we prove the global classical solution for the system \eqref{g-equ-1}-\eqref{equ-3} near the equilibrium $(0, \pm 1)$. Moreover, the time exponetial decay $e^{- c_\# t}$ is also obtained. In Appendix \ref{Sec:Appendix}, we give the proof of Lemma \ref{Lmm-U-Phi}.

\section{Preliminaries}\label{Sec:ADN}

\subsection{Agmon-Douglis-Nirenberg theory}

In order to deal with the high order spatial derivatives of the solutions $(u, \phi)$, we shall employ the well-known Agmon-Douglis-Nirenberg (briefly, ADN) theory \cite{ADN-2}. For convenience of readers, we sketch the theory here. More precisely, they studied the general linear elliptic system on the bounded smooth domain $\Omega \subseteq \R^n$ with the following forms:
\begin{equation}\label{ES}
	\left\{
	\begin{aligned}
		\sum_{j=1}^N l_{ij} (x, \partial) u_j (x) = F_i (x) \ (i = 1, \cdots , N) \textrm{ on } \Omega \,, \\
		\sum_{j=1}^N B_{hj} (x, \partial) u_j (x) = \phi_h (x) \ (h = 1,\cdots,m) \textrm{ on } \partial \Omega \,,
	\end{aligned}
	\right.
\end{equation}
where the $l_{ij} (x, \partial)$, linear differential operators, are polynomials in $\partial$ with coefficients depending on $x \in \Omega$. The orders of these operators are be assumed to depend on two groups of integer weights, $s_1,\cdots, s_N$ and $t_1, \cdots, t_N$, attached to the equations and to the unknowns, respectively, $s_i$ corresponding to the $i$-th equation and $t_j$ to the $j$-th dependent unknown $u_j$. The manner of the dependence is represented by the inequality
\begin{equation}
	\begin{aligned}
		\deg l_{ij} (x, \Xi) \leq s_i + t_j \,, \ i,j = 1, \cdots, N \,,
	\end{aligned}
\end{equation}
where ``$\deg$" refers of course to the degree in $\Xi$, and $s_i \leq 0$. Moreover, $l_{ij}=0$ if $s_i + t_j < 0$. The ellipticity of \eqref{ES} is characterized by
\begin{equation}
	\begin{aligned}
		L (x,\Xi) : = \det (l_{ij}' (x, \Xi)) \neq 0 \quad \textrm{ for real } \Xi \neq 0 \,,
	\end{aligned}
\end{equation}
where $l_{ij}' (x, \Xi)$ consists of the terms in $l_{ij} (x, \Xi)$ which are just of the order $s_i + t_j$. Furthermore, the following {\bf supplementary condition on $L$} should be imposed:
\begin{enumerate}
	\item[\bf (SC)] {\em $L(x, \Xi)$ is of even degree $2m$ (with respect to $\Xi$). For every pair of linearly independent real vectors $\Xi, \Xi'$, the polynomial $L(x, \Xi + \tau \Xi')$ in the complex variable $\tau$ has exactly $m$ roots $\tau_k^+ (x, \Xi)$ ($k = 1, \cdots, m$) with positive imaginary part, i.e., $\mathrm{Im} \tau_k^+ (x, \Xi) > 0$.}
\end{enumerate}
{\em Uniform ellipticity} will be required in the sense that there is a positive constant $A$ such that
\begin{equation}
	\begin{aligned}
		A^{-1} |\Xi|^{2m} \leq |L(x, \Xi)| \leq A |\Xi|^{2m}
	\end{aligned}
\end{equation}
for every real vector $\Xi = (\xi_1, \cdots, \xi_n)$ and for every point $x$ in the closure of the domain $\Sigma$, where $m = \tfrac{1}{2} \deg ( L(x, \Xi) ) > 0$.

In the boundary value of \eqref{ES}, $m$ is exactly $\tfrac{1}{2} \deg ( L(x, \Xi) )$. The linear boundary operator $B_{hj} (x,\partial)$ are of complex coefficients depending on $x$. The orders of $B_{hj} (x,\partial)$, like those of the operators $l_{ij} (x, \partial)$, depend on two groups of integer weights, in this case the group $t_1, \cdots, t_N$ already attached to the dependent unknowns and a new group $r_1, \cdots, r_m$ of which $r_h$ pertains to the $h$-th boundary condition, $h = 1,\cdots,m$. The exact dependence is expressed by the inequality
\begin{equation}
	\begin{aligned}
		\deg B_{hj} (x,\Xi) \leq r_h + t_j \,,
	\end{aligned}
\end{equation}
and $B_{hj} = 0$ when $r_h + t_j < 0$. Moreover, the following {\bf complementing boundary condition} on the boundary operator should also be imposed:
\begin{enumerate}
	\item[\bf (CBC)] {\em For any $x \in \partial \Omega$ and any real, non-zero vector $\Xi$ tangent to $\partial \Omega$ at $x$, let us regard $M^+ (x,\Xi, \tau) = \prod_{h=1}^m (\tau - \tau_h^+ (x, \Xi))$ and the elements of the matrix
		\begin{equation}
			\begin{aligned}
				\sum_{j=1}^N B_{hj}' (x, \Xi + \tau \mathbf{n}) L^{jk} (x, \Xi + \tau \mathbf{n})
			\end{aligned}
		\end{equation}
		as polynomial in the indeterminate $\tau$. The rows of the latter matrix are required to be linearly independent modulo $M^+ (x, \Xi, \tau)$, i.e.,
		\begin{equation*}
			\begin{aligned}
				\sum_{h=1}^m C_h \sum_{j=1}^N B'_{hj} L^{jk} \equiv 0 \ (\mathrm{mod} M^+) \,,
			\end{aligned}
		\end{equation*}
		only if the constant $C_h$ are all zero. Here $\mathbf{n}$ is the normal to $\partial \Omega$ at $x$, $B'_{hj} (x, \Xi)$ consists of the terms in $B_{hj} (x, \Xi)$ which are just of the order $r_h + t_j$, and $(L^jk (x, \Xi + \tau \mathbf{n}))$ denotes the matrix adjoint to $(l'_{ij} (x, \Xi + \tau \mathbf{n}))$.}
\end{enumerate}

Specifically, the operators $l_{ij} (x,\partial)$ and $B_{hj} (x,\partial)$ are of the forms
\begin{equation}
	\begin{aligned}
		l_{ij} (x, \partial) = \sum_{|\alpha| = 0}^{s_i + t_j} a_{ij, \alpha} (x) \partial^\alpha \,, \quad B_{hj} (x, \partial) = \sum_{|\beta| = 0}^{r_h + t_j} b_{hj, \beta} (x) \partial^\beta \,,
	\end{aligned}
\end{equation}
where $\alpha$ and $\beta$ denote multi-indices indicative of the precise differentiation involved. Then the following results hold.

\begin{proposition}[Theorem 10.5 of ADN \cite{ADN-2}]\label{ADN-theory}
	Let $\Omega \subseteq \R^n$ be an open bounded domain of class $C^l$, $l \geq l_1 : = \max (0, r_h + 1)$ and $p > 1$. Assume that $a_{ij, \alpha} \in C^{l - s_i} (\bar{\Omega})$, $b_{hj, \beta} \in C^{l - r_h} (\partial \Omega)$, $F_i \in W^{l-s_i,p} (\Omega)$ and $\phi_h \in W^{l-r_h - \frac{1}{p}, p} (\partial \Omega)$\footnote{$W^{l- r_h - \frac{1}{p}, p} (\partial \Omega) = \gamma_0 W^{l-r_h, p} (\Omega)$ and is equipped with the image norm $$ \| \psi \|_{W^{l-r_h - \frac{1}{p}, p} (\partial \Omega)} = \inf_{\gamma_0 v = \psi} \| v \|_{W^{l-r_h, p} (\Omega)} \,, $$ where $\gamma_0$ is the trace operator on $\partial \Omega$.}. A constant $K$ exists such that, if $\| u_j \|_{W^{l_1 + t_j, p} (\Omega)}$ is finite for $j = 1,\cdots, N$, then $\| u_j \|_{W^{l + t_j, p} (\Omega)}$ also is finite, and
	\begin{equation}
		\begin{aligned}
			\| u_j \|_{W^{l + t_j, p} (\Omega)} \leq K \Big( \sum_{i=1}^N \| F_i \|_{W^{l-s_i, p} (\Omega)} + \sum_{h=1}^m \| \phi_h \|_{W^{l-r_h - \frac{1}{p}, p} (\partial \Omega)} + \sum_{j=1}^N \| u_j \|_{L^p (\Omega)} \Big) \,.
		\end{aligned}
	\end{equation}
	$K > 0$ is dependent on $p$, $l$, $t_i$, $s_j$, $r_h$, $\Omega$, $a_{ij, \alpha}$ and $b_{hj, \beta}$.
\end{proposition}

Remark that if the solution to \eqref{ES} is unique, the term $\sum_{j=1}^N \| u_j \|_{L^p (\Omega)}$ on the right can be omitted.

By employing the ADN theory in Proposition \ref{ADN-theory}, Temam \cite{Temam-1977} proved the following conclusion.
\begin{lemma}[Proposition 2.2, Chapter I of \cite{Temam-1977}]\label{Lmm-Stokes}
	Let $\Omega$ be an open bounded set of class $C^r$, $r = \max (m+2,2)$, integer $m \geq 0$. Let us suppose that
	\begin{equation}
		\begin{aligned}
			u \in W^{1,p} (\Omega), \ q \in L^p (\Omega) \,, \ 1 < p < + \infty \,,
		\end{aligned}
	\end{equation}
	are solutions of the generalized Stokes problem \eqref{GS}:
	\begin{equation}\label{GS}
		\begin{aligned}
			- \mu \Delta u + \nabla q = & f \quad \textrm{in } \Omega \,, \\
			\nabla \cdot u = & g \quad \textrm{in } \Omega \,, \\
			u = & \phi \quad \textrm{on } \partial \Omega \,.
		\end{aligned}
	\end{equation}
	If $f \in W^{m,p} (\Omega)$, $g \in W^{m+1, p} (\Omega)$ and $\phi \in W^{m+2-\frac{1}{p}, p} (\partial \Omega)$, then $u \in W^{m+2, p} (\Omega)$, $q \in W^{m+1,p} (\Omega)$, and there exists a constant $c_0 (p, \mu, m, \Omega)$ such that
	\begin{equation}\label{Stokes-Bnd}
		\begin{aligned}
			\| u \|_{W^{m+2,p} (\Omega)} + & \| q \|_{W^{m+1, p} (\Omega) / \R} \\
			& \leq c_0 \Big\{ \| f \|_{W^{m,p} (\Omega)} + \| g \|_{W^{m+1,p} (\Omega)} + \| \phi \|_{W^{m+2-\frac{1}{p}, p} (\partial \Omega)} + d_p \| u \|_{L^p (\Omega)} \Big\} \,,
		\end{aligned}
	\end{equation}
	where $d_p = 0$ for $p \geq 2$, $d_p = 1$ for $1 < p < 2$.
\end{lemma}

Next, the usual $L^p$-theory of elliptic equation can be reduced by the ADN theory in Proposition \ref{ADN-theory}. One takes the following Laplace equation into consideration:
\begin{equation}\label{phi-Laplace}
	\begin{aligned}
		\Delta \phi = \tfrac{\partial^2}{\partial x_1^2} \phi + \cdots + \tfrac{\partial^2}{\partial x_n^2} \phi = & h \quad \textrm{ in } \Omega \,, \\
		\tfrac{\partial \phi}{\partial \n} = & g \quad \textrm{ in } \partial \Omega \,.
	\end{aligned}
\end{equation}
By letting $\phi_i = \tfrac{\partial}{\partial x_i} \phi$ with $i = 1, \cdots, n$, and $\phi_{n+1} = \phi$, one gains a first order differential system
\begin{equation}\label{1order-ES}
	\left\{
	\begin{aligned}
		& - \phi_1 + \tfrac{\partial}{\partial x_1} \phi_{n+1} = 0 \,, \\
		& \qquad \cdots \cdots \\
		& - \phi_n + \tfrac{\partial}{\partial x_n} \phi_{n+1} = 0 \,, \\
		&\tfrac{\partial}{\partial x_1} \phi_1 + \cdots + \tfrac{\partial}{\partial x_n} \phi_n = h
	\end{aligned}
	\right.
\end{equation}
on $\Omega$ with the boundary condition
\begin{equation}\label{BC-1order-ES}
	\begin{aligned}
		\sum_{j =1}^n \n_j \phi_j = g \quad \textrm{on } \partial \Omega \,,
	\end{aligned}
\end{equation}
where $\n_j$ denotes the $j$-th component of the normal vector $\n$ to $\partial \Omega$. When the weights
\begin{equation*}
	\begin{aligned}
		t_1 = \cdots = t_n = 1 \,, \ t_{n+1} = 2
	\end{aligned}
\end{equation*}
are assigned to $\phi_1, \cdots, \phi_n, \phi_{n+1}$, respectively, and
\begin{equation*}
	\begin{aligned}
		s_1 = \cdots = s_n = -1 \,, \ s_{n+1} = 0
	\end{aligned}
\end{equation*}
to the first, $\cdots$, the $n$-th, and the $(n+1)$-th equations, respectively, the characteristic determinant of \eqref{1order-ES} thereby is
\begin{equation*}
	L (x, \Xi) =
	\left|
	\begin{array}{ccccc}
		-1 & 0 & \cdots & 0 & \xi_1 \\
		0 & -1 & \cdots & 0 & \xi_2 \\
		\vdots & \vdots & \vdots & \vdots & \vdots \\
		0 & 0 & \cdots & - 1 & \xi_n \\
		\xi_1 & \xi_2 & \cdots & \xi_n & 0
	\end{array}
	\right|
	= (-1)^n ( \xi_1^2 + \cdots \xi_n^2 ) = (-1)^n |\Xi|^2 \,,
\end{equation*}
whose degree is $2$, i.e., $m = \tfrac{1}{2} \deg L (x, \Xi) = 1$. The system \eqref{1order-ES} is thus elliptic. Moreover, the equations \eqref{1order-ES} only need one boundary condition \eqref{BC-1order-ES}. It is obvious that the (SC) condition on $L(x,\Xi)$ holds, and the root $\tau^+ (\Xi, \Xi')$ of $L(x, \Xi + \tau \Xi') = 0$ with positive imaginary is
\begin{equation*}
	\begin{aligned}
		\tau^+ (\Xi, \Xi') = \tfrac{- \Xi \cdot \Xi' + \sqrt{-1} \sqrt{|\Xi|^2 |\Xi'|^2 - |\Xi \cdot \Xi'|^2}}{|\Xi'|^2} \,.
	\end{aligned}
\end{equation*}
Moreover, in \eqref{BC-1order-ES}, the weight $r_1 = -1$ is assigned to the only boundary condition. It is easy to verify that
\begin{equation*}
	\begin{aligned}
		M^+ (x,\Xi, \tau) = \tau - \tau^+ (\Xi, \mathbf{n} ) = \tau - \sqrt{-1} |\Xi| \,.
	\end{aligned}
\end{equation*}
Observe that $B_{1j} = \n_j$ for $1 \leq j \leq n$ and $B_{1,n+1} = 0$. We have $B_{1j} = B_{1j}'$. A direct calculation shows that the matrix $(L^{jk} (x, \Xi))_{1 \leq j,k \leq n+1}$ is
\begin{equation*}
	|\Xi|^{-2}
	\left(
	\begin{array}{ccccc}
		\xi_1^2 - |\Xi|^2 & \xi_1 \xi_2 & \cdots & \xi_1 \xi_n & \xi_1 \\
		\xi_2 \xi_1 & \xi_2^2 - |\Xi|^2 & \cdots & \xi_2 \xi_n & \xi_2 \\
		\vdots & \vdots & \vdots & \vdots & \vdots \\
		\xi_n \xi_1 & \xi_n \xi_2 & \cdots & \xi_n^2 - |\Xi|^2 & \xi_n \\
		\xi_1 & \xi_2 & \cdots & \xi_n & 1
	\end{array}
	\right) \,.
\end{equation*}
By $\n \cdot \Xi = 0$ and $|\n| = 1$, the vector with elements $C_1 \sum_{j=1}^{n+1} B'_{1j} L^{jk} (x, \Xi + \tau \n)$ is then equal to
\begin{equation*}
	\begin{aligned}
		C_1 ( \tau \Xi - |\Xi|^2 \n , 0 ) \,,
	\end{aligned}
\end{equation*}
which is zero modulo $M^+$ if and only if $C_1 = 0$. Therefore, the complementing boundary condition (CBC) holds.

Therefore, the ADN theory in Proposition \ref{ADN-theory} directly concludes the results:

\begin{lemma}\label{Lmm-Laplace}
	Let $\Omega$ be an open bounded domain of class $C^l$, $l \geq 0$ and $p > 1$. Assume that $h \in W^{l,p} (\Omega)$, $g \in W^{l+1 - \frac{1}{p}, p} (\partial \Omega)$, and $\phi \in W^{2,p} (\Omega)$ is a solution to the boundary value problem \eqref{phi-Laplace}. Then, a constant $K > 0$ exists such that $\phi \in W^{l+2, p} (\Omega)$ and
	\begin{equation}
		\begin{aligned}
			\| \phi \|_{W^{l+2,p} (\Omega)} \leq K \big( \| h \|_{W^{l,p} (\Omega)} + \| g \|_{W^{l+1 - \frac{1}{p}, p} (\partial \Omega)} + \| \phi \|_{W^{1,p} (\Omega)} \big) \,.
		\end{aligned}
	\end{equation}
\end{lemma}

\subsection{Boundary conditions for ACNS system \eqref{equ-1}}

In this paper, the boundary values for ACNS system \eqref{equ-1} is imposed on \eqref{BC-ACNS}, i.e.,
\begin{equation*}
	\begin{aligned}
		u |_{\partial \Omega} = 0 \,, \quad \tfrac{\partial \phi}{\partial \mathbf{n}} |_{\partial \Omega} = 0 \,.
	\end{aligned}
\end{equation*}
The goal of this paper is to investigate the well-posedness of the ACNS equations \eqref{equ-1} in the Sobolev space $H^s (\Omega)$ for large integer $s > 0$. To deal with the boundary values of the higher order derivatives is thereby one of the {\em key points} of this paper. However, the values of the higher order spatial derivatives restricted on the boundary $\partial \Omega$ is impossible to be controlled by the boundary values \eqref{BC-ACNS}, i.e., $u |_{\partial \Omega} = 0 \,, \tfrac{\partial \phi}{\partial \mathbf{n}} |_{\partial \Omega} = 0$. The idea is to convert the boundary values of the higher order spatial derivatives into that of the higher order time derivatives by the constitutive of the equations and ADN theory. For the boundary values of higher order time derivatives, the conditions \eqref{BC-ACNS}, i.e., $u |_{\partial \Omega} = 0 \,, \tfrac{\partial \phi}{\partial \mathbf{n}} |_{\partial \Omega} = 0$, show that, for $j \in \mathbb{N}$,
\begin{equation}\label{BC-HighTimeD}
	\begin{aligned}
		\partial_t^j u |_{\partial \Omega} = 0 \,, \quad \tfrac{\partial}{\partial \mathbf{n}} \partial_t^j \phi |_{\partial \Omega} = 0 \,,
	\end{aligned}
\end{equation}
provided that they are all well-defined.

\subsection{Some basic estimates}

In this subsection, we first give some calculus inequalities, which will be frequently used later.

For the functions $u (t,x)$ and $\phi (t, x)$ satisfying the boundary conditions \eqref{BC-HighTimeD} with integers $j \geq 0$, one has
\begin{equation}
	\begin{aligned}
		\int_\Omega \nabla \partial_t^j u d x = \int_{\partial \Omega} \partial_t^j u \otimes \n d S = 0 \,, \quad \int_\Omega \Delta \partial_t^j \phi d x = \int_{\partial \Omega} \tfrac{\partial}{\partial \n} \partial_t^j \phi d S = 0 \,.
	\end{aligned}
\end{equation}
Then the Gagliardo-Nirenberg interpolation inequality gives the following results.
\begin{lemma}\label{Lmm-Calculus}
	For any $2 < p \leq 6$, $6 < r \leq \infty$ and integer $j \geq 0$, there hold
	\begin{equation}
		\begin{aligned}
			& \| \partial_t^j u \|_{L^p} \lesssim \| \partial_t^j u \|^{\frac{3}{p} - \frac{1}{2}} \| \nabla \partial_t^j u \|^{\frac{3}{2} - \frac{3}{p} } \,, \\
			& \| \partial_t^j u \|_{L^\infty} \lesssim \| \Delta \partial_t^j u \|^\frac{1}{2} \| \nabla \partial_t^j u \|^\frac{1}{2} \,, \\
			& \| \nabla \partial_t^j u \|_{L^p} \lesssim \| \nabla \partial_t^j u \|^{\frac{3}{p} - \frac{1}{2}} \| \Delta \partial_t^j u \|^{\frac{3}{2} - \frac{3}{p} } \,, \\
			& \| \partial_t^j \phi \|_{L^r} \lesssim \| \partial_t^j \phi \|^{\frac{1}{4} + \frac{3}{2 r}} \| \Delta \partial_t^j \phi \|^{\frac{3}{4} - \frac{3}{2 r}} + \| \partial_t^j \phi \| \,, \\
			& \| \nabla \partial_t^j \phi \|_{L^p} \lesssim \| \nabla \partial_t^j \phi \|^{\frac{3}{p} - \frac{1}{2}} \| \Delta \partial_t^j \phi \|^{\frac{3}{2} - \frac{3}{p} } + \| \nabla \partial_t^j \phi \| \,, \\
			& \| \Delta \partial_t^j \phi \|_{L^p} \lesssim \| \Delta \partial_t^j \phi \|^{\frac{3}{p} - \frac{1}{2}} \| \nabla \Delta \partial_t^j \phi \|^{\frac{3}{2} - \frac{3}{p} } \,,
		\end{aligned}
	\end{equation}
    provided that the right-hand side of the quantities are all finite.
\end{lemma}

\section{The A Priori Estimates}\label{Sec:Apriori}
	In this section, we devoted to the a priori estimate for the system \eqref{equ-1}-\eqref{equ-2}. We divided this section into three parts. We first obtain the closed $H^2$-estimates of \eqref{equ-1}. Then the higher order time derivative estimates are derived. At the end, the higher order spatial derivative estimates are established by employing the ADN theory.
	
\subsection{Preparation}
	We first introduce the following energy functional $E_j (t)$ and dissipation functional $D_j (t)$ for $j \geq 0$,
\begin{equation}\label{EjDj}
	\begin{split}
	{E}_j(t)  = &  \| \partial_t^{j} u \|_{L^2_{\rho(\phi)}}^2 + \mu \| \nabla \partial_t^{j} u \|^2 + \| \partial_t^{j} \phi \|^2 + (\gamma \lambda+1) \| \nabla \partial_t^{j}\phi \|^2 + \gamma\lambda\| \Delta \partial_t^{j} \phi \|^2,\\
		{D}_j(t)  = & \mu\| \nabla \partial_t^{j} u \|^2 + \| \partial_t^j u_t \|_{L^2_{\rho(\phi)}}^2 + \gamma \lambda \| \nabla \partial_t^{j} \phi \|^2 +  \| \partial_t^j \phi_t \|^2 + \gamma \lambda \| \Delta \partial_t^{j} \phi \|^2 \\
		&  \quad + \| \nabla \partial_t^j \phi_t \|^2 + \kappa \| \Delta  \partial_t^{j} u \|^2 + \kappa \| \nabla \Delta \partial_t^{j}  \phi \|^2 + \kappa \| \partial_t^j p \|^2_1 \,,
	\end{split}
\end{equation}
where the positive constant $\kappa > 0$ in $D_j (t)$ will be determined later.

Moreover, the following bounds will be used later.

\begin{lemma}\label{Lmm-K}
	Let $K_u$ and $K_\phi$ be defined in \eqref{Ku} and \eqref{Kphi} below, i.e.,
	\begin{equation*}
		\begin{aligned}
			K_u = \rho(\phi) u \cdot \nabla u + \lambda \nabla \phi \Delta\phi \,, \ K_\phi = u \cdot \nabla \phi + \gamma \lambda f' (\phi) + \gamma \rho' (\phi) \tfrac{|u|^2}{2} \,.
		\end{aligned}
	\end{equation*}
    Then, for any integer $k \geq 0$, there hold
    \begin{equation}\label{Ku-Kphi-bnd}
    	\begin{aligned}
    		\| \partial_t^k K_u \| \leq & \delta \sum_{0 \leq j \leq k} \big( \| \Delta \partial_t^j u \| + \| \nabla \Delta \partial_t^j \phi \| \big) + C_\delta \sum_{0 \leq j \leq k} (1 + E_j^4 (t)) E_j^\frac{5}{2} (t) \,,  \\
    		\| \partial_t^k K_\phi \|_1  \leq & \delta \sum_{0 \leq j \leq k} \| \Delta \partial_t^j u \| + C_\delta \sum_{0 \leq j \leq k} ( 1 + E_j^4 (t) ) E_j^\frac{5}{2} (t)
    	\end{aligned}
    \end{equation}
    for any small $\delta > 0$ and some $C_\delta > 0$. Here $E_j (t)$ are defined in \eqref{EjDj}.
\end{lemma}

\begin{proof}
	We first dominate the quantity $\| \partial_t^k K_\phi \|$. Note that
	\begin{equation*}
		\begin{aligned}
			\| \partial_t^k [\rho (\phi) u \cdot \nabla u ] \| \leq C \sum_{a+b+c = k} \| \partial_t^a \rho (\phi) \|_{L^\infty} \| \partial_t^b u \|_{L^4} \| \nabla \partial_t^c u \|_{L^4} \,.
		\end{aligned}
	\end{equation*}
	By the expression of $\rho (\phi)$ in \eqref{rho-def} and Lemma \ref{Lmm-Calculus}, one knows that
	\begin{equation*}
		\begin{aligned}
			\| \partial_t^a \rho (\phi) \|_{L^\infty} \leq & C ( 1 + \sum_{a' \leq a} \| \partial_t^{a'} \phi \|^2_{L^\infty} ) \\
			\leq & C + C \sum_{a' \leq a} ( \| \partial_t^{a'} \phi \|^\frac{1}{2} \| \Delta \partial_t^{a'} \phi \|^\frac{3}{2} + \| \partial_t^{a'} \phi \|^2 ) \\
			\leq & C + C \sum_{a' \leq a} E_{a'} (t) \,.
		\end{aligned}
	\end{equation*}
	Moreover, Lemma \ref{Lmm-Calculus} indicates that
	\begin{equation*}
		\begin{aligned}
			\| \partial_t^b u \|_{L^4} \| \nabla \partial_t^c u \|_{L^4} \lesssim \| \partial_t^b u \|^\frac{1}{4} \| \nabla \partial_t^b u \|^\frac{3}{4} \| \nabla \partial_t^c u \|^\frac{1}{4} \| \Delta \partial_t^c u \|^\frac{3}{4} \lesssim E_b^\frac{1}{2} (t) E_c^\frac{1}{8} (t) \| \Delta \partial_t^c u \|^\frac{3}{4} \,.
		\end{aligned}
	\end{equation*}
	Collecting the above estimates and employing the Young's inequality, we have
	\begin{equation}\label{Ku-bdn-1}
		\begin{aligned}
			\| \partial_t^k [\rho (\phi) u \cdot \nabla u ] \| \leq C \sum_{a+b+c = k} \Big(1 + \sum_{a' \leq a} E_{a'} (t) \Big) E_b^\frac{1}{2} (t) E_c^\frac{1}{8} (t) \| \Delta \partial_t^c u \|^\frac{3}{4} \\
			\leq \delta \sum_{0 \leq j \leq k} \| \Delta \partial_t^j u \| + C_\delta \sum_{0 \leq j \leq k} ( 1 + E_j^4 (t) ) E_j^\frac{5}{2} (t)
		\end{aligned}
	\end{equation}
	for any small $\delta > 0$ and some $C_\delta > 0$. By the similar arguments in the previous estimates, one can calculate that
	\begin{equation}\label{Ku-bnd-2}
		\begin{aligned}
			\| \partial_t^k ( \lambda \nabla \phi \Delta \phi ) \| \leq & C \sum_{a+b = k} E_a^\frac{1}{2} (t) E_b^\frac{1}{8} (t) \| \nabla \Delta \phi \|^\frac{3}{4} \\
			\leq & \delta \sum_{0 \leq j \leq k} \| \nabla \Delta \partial_t^j \phi \| + C_\delta \sum_{0 \leq j \leq k} E_j^\frac{5}{2} (t) \,.
		\end{aligned}
	\end{equation}
	Then \eqref{Ku-bdn-1} and \eqref{Ku-bnd-2} imply the first inequality in \eqref{Ku-Kphi-bnd}.
	
	Next we focus on the norms $\| \partial_t^k K_\phi \|_1$. By the similar arguments in \eqref{Ku-bdn-1}, one can obtain
	\begin{equation*}
		\begin{aligned}
			\| \partial_t^k (u \cdot \nabla \phi) \| \leq C \sum_{a+b = k} E_a^\frac{1}{2} E_b^\frac{1}{2} (t) \leq C \sum_{0 \leq j \leq k} E_j (t) \,.
		\end{aligned}
	\end{equation*}
	Recalling $f' (\phi) = \tfrac{1}{\varepsilon^2} (\phi^2 - 1) \phi$ with $-1 \leq \phi \leq 1$, one knows
	\begin{equation*}
		\begin{aligned}
			\| \gamma \lambda \partial_t^k f' (\phi) \| \leq & C \sum_{a+b = k} \| \partial_t^a (\phi^2 - 1) \partial_t^b \phi \| \\
			\leq & C \sum_{a+b = k} \big( 1 + \sum_{a' \leq a} \| \partial_t^{a'} \phi \|^2_{L^\infty} \big) \| \partial_t^b \phi \| \\
			\leq & C \sum_{a+b = k} \Big[ 1 + \sum_{a' \leq a} \big( \| \partial_t^{a'} \phi \|^\frac{1}{2} \| \Delta \partial_t^{a'} \phi \|^\frac{3}{2} + \| \partial_t^{a'} \phi \|^2 \big) \Big] \| \partial_t^b \phi \| \\
			\leq & C \sum_{0 \leq j \leq k} (1 + E_j (t)) E_j^\frac{1}{2} (t) \,,
		\end{aligned}
	\end{equation*}
	where Lemma \ref{Lmm-Calculus} has been utilized. As similar as in \eqref{Ku-bdn-1}, it infers that
	\begin{equation*}
		\begin{aligned}
			\| \partial_t^k [ \gamma \rho' (\phi) \tfrac{|u|^2}{2} ] \| \leq & C \sum_{a+b+c = k} \Big( 1 + \sum_{a' \leq a} E_{a'} (t) \Big) E_b^\frac{1}{2} (t) E_c^\frac{1}{8} \| \nabla \partial_t^c u \|^\frac{3}{4} \\
			\leq & C \sum_{a+b+c = k} \Big( 1 + \sum_{a' \leq a} E_{a'} (t) \Big) E_b^\frac{1}{2} (t) E_c^\frac{1}{8} E_c^\frac{3}{8} (t) \\
			\leq & C \sum_{0 \leq j \leq k} ( 1 + E_j (t) ) E_j (t) \,.
		\end{aligned}
	\end{equation*}
	Consequently, we gain
	\begin{equation}\label{Kphi-1}
		\begin{aligned}
			\| \partial_t^k K_\phi \| \leq C \sum_{0 \leq j \leq k} ( 1 + E_j (t) ) E_j (t) \,.
		\end{aligned}
	\end{equation}

	Furthermore, one observes that
	\begin{equation*}
		\begin{aligned}
			\nabla \partial_t^k K_\phi = \partial_t^k ( \nabla u \cdot \nabla \phi ) + \partial_t^k ( u \cdot \nabla^2 \phi ) + \tfrac{\gamma \lambda}{\varepsilon^2} \nabla \partial_t^k [ (\phi^2 - 1) \phi ] + \gamma \nabla \partial_t^k [ \rho' (\phi) \tfrac{|u|^2}{2} ] \,.
		\end{aligned}
	\end{equation*}
	By Lemma \ref{Lmm-Calculus} and the H\"older inequality, it is implied that
	\begin{equation*}
		\begin{aligned}
			\| \partial_t^k (\nabla u \cdot \nabla \phi) \| \leq & C \sum_{a+b = k} \| \nabla \partial_t^a u \|_{L^4} \| \nabla \partial_t^b \phi \|_{L^4} \\
			\leq & C \sum_{a+b = k} \| \nabla \partial_t^a u \|^\frac{1}{4} \| \Delta \partial_t^a u \|^\frac{3}{4} \big( \| \nabla \partial_t^b \phi \|^\frac{1}{4} \| \Delta \partial_t^b \phi \|^\frac{3}{4} + \| \nabla \partial_t^b \phi \| \big) \\
			\leq & C \sum_{a+b = k} E_a^\frac{1}{8} (t) E_b^\frac{1}{2} (t) \| \Delta \partial_t^a u \|^\frac{3}{4} \\
			\leq & \tfrac{\delta}{3} \sum_{0 \leq j \leq k} \| \Delta \partial_t^j u \| + C_\delta \sum_{0 \leq j \leq k} E_j^\frac{5}{2} (t)
		\end{aligned}
	\end{equation*}
	for any small $\delta > 0$ and some $C_\delta > 0$. Then by Lemma \ref{Lmm-Calculus},
	\begin{equation*}
		\begin{aligned}
			\| \partial_t^k (u \cdot \nabla^2 \phi ) \| \leq C \sum_{a+b = k} \| \partial_t^a u \|_{L^\infty} \| \nabla^2 \partial_t^b \phi \| \leq C \sum_{a+b = k} \| \nabla \partial_t^a u \|^\frac{1}{2} \| \Delta \partial_t^a u \|^\frac{1}{2} \| \Delta \partial_t^b \phi \| \\
			\leq C \sum_{a+b = k} E_a^\frac{1}{4} (t) E_b^\frac{1}{2} (t) \| \Delta \partial_t^a u \|^\frac{1}{2} \leq \tfrac{\delta}{3} \sum_{0 \leq j \leq k} \| \Delta \partial_t^j u \| + C_\delta \sum_{0 \leq j \leq k} E_j^\frac{3}{2} (t) \,.
		\end{aligned}
	\end{equation*}
	Moreover, by the similar arguments in \eqref{Ku-bdn-1}, there hold
	\begin{equation*}
		\begin{aligned}
			\| \tfrac{\gamma \lambda}{\varepsilon^2} \| \nabla \partial_t^k [(\phi^2 - 1) \phi ] \| \leq C \sum_{0 \leq j \leq k} (1 + E_j (t)) E_j^\frac{1}{2} (t)
		\end{aligned}
	\end{equation*}
	and
	\begin{equation*}
		\begin{aligned}
			\| \gamma \nabla \partial_t^k [ \rho' (\phi) \tfrac{|u|^2}{2} ] \| \leq \tfrac{\delta}{3} \sum_{0 \leq j \leq k} \| \Delta \partial_t^j u \| + C_\delta \sum_{0 \leq j \leq k} ( 1 + E_j^4 (t) ) E_j^\frac{5}{2} (t) \,.
		\end{aligned}
	\end{equation*}
	Therefore, we have
	\begin{equation}\label{Kphi-2}
		\begin{aligned}
			\| \nabla \partial_t^k K_\phi \| \leq \delta \sum_{0 \leq j \leq k} \| \Delta \partial_t^j u \| + C_\delta \sum_{0 \leq j \leq k} ( 1 + E_j^4 (t) ) E_j^\frac{5}{2} (t) \,.
		\end{aligned}
	\end{equation}
	As a result, the second inequality in \eqref{Ku-Kphi-bnd} is concluded by \eqref{Kphi-1} and \eqref{Kphi-2}. Then the proof of Lemma \ref{Lmm-K} is completed.
\end{proof}

\subsection{$H^2$-estimates for ACNS equations \eqref{equ-1}}\label{Subsec:3_2}

We first derive the $H^2$-estimates of the ACNS system \eqref{equ-1}, which will contain the major structures of the energy functionals. More precisely, the following time differential inequality holds.

\begin{lemma}\label{Lmm-H2-t0}
	Let $(u, p, \phi)$ be a sufficiently smooth solution to \eqref{equ-1} over $(t,x) \in [0,T) \times \Omega$. Then there are constant $C > 0$ and small constant $\kappa > 0$ in the definition of $D_0 (t)$ of \eqref{EjDj} such that
	\begin{equation}\label{E0D0-bnd}
		\begin{aligned}
			\tfrac{d}{dt} E_0(t) + D_0(t) \leq C ( 1 + E_0^{12}(t) ) E_0(t)
		\end{aligned}
	\end{equation}
    for all $t \in [0,T)$, where the functionals $E_0 (t)$ and $D_0 (t)$ are defined in \eqref{EjDj}.
\end{lemma}

\begin{proof}[Proof of Lemma \ref{Lmm-H2-t0}]
		
One first takes $L^2$-inner product of the first equation of \eqref{equ-1} by dot with $ u + u_t $. There thereby holds
    \begin{equation*}
		\begin{split}
		\int_\Omega \rho(\phi) u_t ( u + u_t ) dx & + \int_\Omega \rho(\phi) ( u \cdot \nabla u ) ( u + u_t) dx + \int_\Omega \nabla p ( u + u_t) dx
	\\ &	=  \mu \int_\Omega \Delta u ( u + u_t) dx - \lambda \int_\Omega \nabla \cdot (\nabla \phi \otimes \nabla \phi) ( u + u_t) dx.
		\end{split}
	\end{equation*}
By the boundary conditions \eqref{BC-ACNS} and \eqref{BC-HighTimeD}, one knows that
\begin{equation*}
	\begin{aligned}
		u|_{\partial \Omega} = 0 \,, \quad u_t |_{\partial \Omega} = 0 \,.
	\end{aligned}
\end{equation*}
Moreover, $\nabla \cdot u=0 $ implies $ \nabla \cdot u_t = (\nabla \cdot u)_t= 0 $. Thus, integrating by parts over $x\in \Omega $ reduces to
\begin{equation}\label{H2-1}
	\begin{split}
		\tfrac{1}{2} \tfrac{d}{dt} ( \| u \|_{ L_{\rho(\phi)}^2 }^2 + & \mu \| \nabla u \|^2 )  + \mu \|\nabla u\|^2 + \| u_t \|_{ L_{\rho(\phi)}^2 }^2 = \langle K_u, - u - u_t  \rangle + \tfrac{1}{2} \langle \rho'(\phi) \phi_t , |u|^2 \rangle \,,
	\end{split}
\end{equation}
where
\begin{equation}\label{Ku}
	\begin{aligned}
		K_u = \rho(\phi) u \cdot \nabla u + \lambda \nabla \phi \Delta\phi \,.
	\end{aligned}
\end{equation}

We then multiply the third equation of \eqref{equ-1} by $\phi + \phi_t $, which means that
\begin{equation*}
	\begin{split}
		\tfrac{1}{2} \tfrac{d}{dt} \| \phi \|^2 &  + \| \phi_t \|^2 + \int_\Omega u\cdot\nabla\phi (\phi + \phi_t ) dx  \\ & =  \int_\Omega  \gamma \lambda \Delta \phi ( \phi + \phi_t ) dx - \int_\Omega ( \gamma\lambda  f'(\phi) + \gamma \rho'(\phi)\tfrac{|u|^2}{2} ) (\phi + \phi_t ) dx .
	\end{split}
\end{equation*}
Integrating by parts over $x \in \Omega$ and the boundary value $ \tfrac{\partial \phi}{\partial \n} |_{\partial \Omega} =0 $ in \eqref{BC-ACNS} imply that
\begin{equation*}
	\begin{split}
		& \int_\Omega \Delta \phi \phi dx = -  \int_\Omega |\nabla \phi |^2 dx +  \int_{\partial \Omega} \phi \tfrac{\partial \phi}{\partial \n} dS = - \| \nabla \phi \|^2 \,,\\
		& \int_\Omega  \Delta \phi \phi_t dx = - \int_\Omega \nabla \phi \cdot \nabla \phi_t dx +  \int_{\partial \Omega} \tfrac{\partial \phi}{\partial \n} \phi_t dS = -\tfrac{1}{2} \tfrac{d}{dt}  \| \nabla \phi \|^2 \,. \\
	\end{split}
\end{equation*}
Thus, one gains
\begin{equation}\label{H2-2}
	\begin{aligned}
		\tfrac{1}{2} \tfrac{d}{d t} ( \| \phi \|^2 + \| \nabla \phi \|^2 ) + \| \phi_t \|^2 + \| \nabla \phi \|^2 = - \langle u \cdot \nabla \phi + \gamma\lambda  f'(\phi) + \gamma \rho'(\phi)\tfrac{|u|^2}{2} \,, \phi + \phi_t \rangle \,,
	\end{aligned}
\end{equation}
where
\begin{equation}\label{Kphi}
	\begin{aligned}
		K_\phi = u \cdot \nabla \phi + \gamma\lambda  f'(\phi) + \gamma \rho'(\phi)\tfrac{|u|^2}{2} \,.
	\end{aligned}
\end{equation}

Furthermore, from taking $L^2$-inner product of the third equation of \eqref{equ-1} by dot with $ \Delta \phi + \Delta \phi_t $, we deduce that
\begin{equation*}
	\begin{split}
		\int_\Omega & \phi_t (\Delta \phi + \Delta \phi_t) dx =  \tfrac{1}{2} \tfrac{d}{dt} \gamma \lambda \| \Delta \phi \|^2 + \gamma \lambda \| \Delta \phi \|^2 - \int_\Omega K_\phi (\Delta \phi + \Delta \phi_t) dx.
	\end{split}
\end{equation*}
The conditions \eqref{BC-HighTimeD} indicate that $\tfrac{\partial \phi_t}{\partial \n} |_{\partial \Omega} = 0$. Then the integration by parts over $x \in \Omega$ tells us
\begin{equation*}
	\begin{split}
		& \int_\Omega \phi_t \Delta \phi_t  dx = - \int_\Omega |\nabla \phi_t |^2 dx + \int_{\partial \Omega} \phi_t \tfrac{\partial \phi_t}{\partial \n} dx = - \| \nabla \phi_t \|^2 ,\\
		& \int_\Omega K_\phi \Delta \phi_t  dx  = - \int_\Omega \nabla K_\phi \cdot \nabla \phi_t  dx + \int_{\partial \Omega} K_\phi \tfrac{\partial \phi_t}{\partial \n} dS  = - \langle \nabla K_\phi \,, \nabla \phi_t  \rangle \,.
	\end{split}
\end{equation*}
Consequently, one has
\begin{equation}\label{H2-3}
	\begin{aligned}
		\tfrac{1}{2} \tfrac{d}{d t} \gamma \lambda \| \Delta \phi \|^2 + \| \nabla \phi_t \|^2 + \gamma \lambda \| \Delta \phi \|^2 = \langle K_\phi \,, \Delta \phi \rangle - \langle \nabla K_\phi \,, \nabla \phi_t \rangle \,.
	\end{aligned}
\end{equation}
Then, by the equalities \eqref{H2-1}, \eqref{H2-2} and \eqref{H2-3}, there holds
\begin{equation}\label{H2-4}
	\begin{aligned}
		& \tfrac{1}{2} \tfrac{d}{dt} ( \| u \|_{ L_{\rho(\phi)}^2 }^2 + \mu \| \nabla u \|^2 + \| \phi \|^2 + (\gamma \lambda+1) \| \nabla \phi \|^2 + \gamma \lambda \| \Delta \phi \|^2 ) \\
		& + \mu \|\nabla u\|^2 + \| u_t \|_{ L_{\rho(\phi)}^2 }^2 + \gamma \lambda \| \nabla \phi \|^2 + \| \phi_t \|^2 + \gamma \lambda \| \Delta \phi \|^2 + \| \nabla \phi_t \|^2 \\
		= & - \langle K_u , u_t + u \rangle + \langle K_\phi \,, \Delta \phi - \phi_t - \phi \rangle  - \langle \nabla K_\phi \,, \nabla \phi_t \rangle + \tfrac{1}{2} \langle \rho'(\phi) \phi_t , |u|^2 \rangle \,.
	\end{aligned}
\end{equation}
By the H\"older inequality, Lemma \ref{Lmm-K} and the bound $\| u_t \|^2 \leq c_1 \| u_t \|^2_{L^2_{\rho (\phi)}}$ with $c_1 = \tfrac{\rho_1 + \rho_2}{\rho_1 \rho_2} > 0$ derived by \eqref{Low-Bnd-rho-phi}, the first three terms in the right-hand side of \eqref{H2-4} can be bounded by
\begin{equation}\label{H2-5}
	\begin{aligned}
		& \| K_u \| ( \| u \| + \| u_t \|_{L^2_{\rho (\phi)}} ) + \| K_\phi \|_1 ( \| (\phi, \Delta \phi) \| + \| (\phi_t, \nabla \phi_t ) \| ) \\
		\leq & \delta_1 \| (\phi_t, \nabla \phi_t) \|^2 + c_1 \delta_1 \| u_t\|^2_{L^2_{\rho (\phi)}} + C_{\delta_1} ( \| K_u \|^2 + \| K_\phi \|_1^2 ) + \| (u, \phi, \Delta \phi) \|^2 \\
		\leq & \delta_1 \| (\phi_t, \nabla \phi_t) \|^2 + c_1 \delta_1 \| u_t\|^2_{L^2_{\rho (\phi)}} + 2 C_{\delta_1} \delta^2 \big( \| \Delta u \|^2 + \| \nabla \Delta \phi \|^2 \big) \\
		& + C E_0 (t) + 2 C_{\delta_1} C_\delta^2 (1 + E_0^4 (t))^2 E_0^5 (t) \,,
	\end{aligned}
\end{equation}
where the small $\delta_1 > 0$ is to be determined. Moreover, the last term in the right-hand side of \eqref{H2-4} can be controlled by
\begin{equation}\label{H2-6}
	\begin{aligned}
		\tfrac{1}{2} \langle \rho'(\phi) \phi_t , |u|^2 \rangle \leq & C \| \phi_t \|_{L^3} \| u \|^2_{L^3} \leq C \| \phi_t \|^\frac{1}{2} \| \nabla \phi_t \|^\frac{1}{2} \| u \|_{L^2_{\rho (\phi)}} \| \nabla u \| \\
		\leq & \delta_1 \| (\phi_t, \nabla \phi_t) \|^2 + C_{\delta_1} \| (u, \nabla u ) \|^4 \\
		\leq & \delta_1 \| (\phi_t, \nabla \phi_t) \|^2 + C_{\delta_1} E_0^2 (t) \,,
	\end{aligned}
\end{equation}
where Lemma \ref{Lmm-Calculus} has been utilized. Then, plugging \eqref{H2-5} and \eqref{H2-6} into \eqref{H2-4} reduces to
\begin{equation}\label{H2-7}
	\begin{aligned}
		& \tfrac{1}{2} \tfrac{d}{dt} ( \| u \|_{ L_{\rho(\phi)}^2 }^2 + \mu \| \nabla u \|^2 + \| \phi \|^2 + (\gamma \lambda+1) \| \nabla \phi \|^2 + \gamma \lambda \| \Delta \phi \|^2 ) \\ & + \mu \|\nabla u\|^2 + \| u_t \|_{ L_{\rho(\phi)}^2 }^2 + \gamma \lambda \| \nabla \phi \|^2 + \| \phi_t \|^2 + \gamma \lambda \| \Delta \phi \|^2 + \| \nabla \phi_t \|^2 \\
		\leq & 2 \delta_1 \| (\phi_t, \nabla \phi_t) \|^2 + c_1 \delta_1 \| u_t\|^2_{L^2_{\rho (\phi)}} + 2 C_{\delta_1} \delta^2 \big( \| \Delta u \|^2 + \| \nabla \Delta \phi \|^2 \big) \\
		& + \tilde{C}_{\delta, \delta_1} (1 + E_0^{12} (t)) E_0 (t)
	\end{aligned}
\end{equation}
for small $\delta, \delta_1 > 0$ to be determined.

Remark that the differential inequality \eqref{H2-7} is still not closed, due to the uncontrolled quantity $2 C_{\delta_1} \delta^2 \big( \| \Delta u \|^2 + \| \nabla \Delta \phi \|^2 \big)$ in the right-hand side of \eqref{H2-7}. In order to control this quantity, our way is to employ Lemma \ref{Lmm-Stokes}, a corollary of ADN theory in Proposition \ref{ADN-theory}, and the $\phi$-equation of \eqref{equ-1}.

From the $u$-equation of \eqref{equ-1} and \eqref{BC-ACNS}, one has
\begin{equation}
	\begin{aligned}
		- \mu \Delta u + \nabla p = & U(u, \phi) \quad \ \textrm{in } \Omega \,, \\
		\nabla \cdot u = & 0 \qquad \qquad \textrm{in } \Omega \,, \\
		u = & 0 \qquad \qquad \textrm{on } \partial \Omega \,,
	\end{aligned}
\end{equation}
where
\begin{equation}\label{U-u-phi}
	\begin{aligned}
		U(u, \phi) = - \rho (\phi) (u_t + u \cdot \nabla u) - \lambda \nabla \cdot ( \nabla \phi \otimes \nabla \phi ) \,.
	\end{aligned}
\end{equation}
Then the inequality \eqref{Stokes-Bnd} in Lemma \ref{Lmm-Stokes} indicates that
\begin{equation}\label{Delta-u-bnd-1}
	\begin{aligned}
		\| \Delta u \|^2 + \| p \|_1^2 \leq & \| u \|^2_2 + \| p \|^2_1 \leq c_0^2 \| U (u, \phi) \|^2 \\
		\leq & \tfrac{1}{2} c_1 \| u_t \|^2_{L^2_{\rho (\phi)}} + C \| u \|_{L^4}^2 \| \nabla u \|^2_{L^4} + C \| \Delta \phi \|^2_{L^4} \| \nabla \phi \|^2_{L^4} \,,
	\end{aligned}
\end{equation}
where $c_1 = 6 c_0^2 \max\{\rho_1, \rho_2\} > 0$ and the last inequality is derived from \eqref{U-u-phi}, \eqref{Low-Bnd-rho-phi} and the H\"older inequality. Recalling the definition of $E_0 (t)$ in \eqref{EjDj}, one derives from \eqref{Delta-u-bnd-1} and Lemma \ref{Lmm-Calculus} that
\begin{equation}\label{Delta-u-bnd}
	\begin{aligned}
		\| \Delta u \|^2 + \| p \|^2_1 \leq & \tfrac{1}{2} c_1 \| u_t \|^2_{L^2_{\rho (\phi)}} + C \| \Delta u \|^\frac{3}{2} E_0^\frac{5}{4} (t) + \| \nabla \Delta \phi \|^\frac{3}{2} E_0^\frac{5}{4} (t) \\
		\leq & \tfrac{1}{2} c_1 \| u_t \|^2_{L^2_{\rho (\phi)}} + \tfrac{1}{4} \| \Delta u \|^2 + \tfrac{1}{2} \| \nabla \Delta \phi \|^2 + C E_0^5 (t) \,.
	\end{aligned}
\end{equation}

Next one focuses on the uncontrolled quantity $\| \nabla \Delta \phi \|^2$. In order to deal with it, one will employ the $\phi$-equation in \eqref{equ-1}. More precisely,
\begin{equation*}
	\begin{aligned}
		\Delta \phi = \Phi (u, \phi) : = \tfrac{1}{\gamma \lambda} (\phi_t + u \cdot \nabla \phi) + \big[ f' (\phi) + \tfrac{1}{\lambda} \rho' (\phi) \tfrac{|u|^2}{2} \big] \,,
	\end{aligned}
\end{equation*}
which implies that
\begin{equation}\label{ND-phi-1}
	\begin{aligned}
		\| \nabla \Delta \phi \|^2 \leq \tfrac{1}{2} c_3 \| \nabla \phi_t \|^2 + C \| \nabla (u \cdot \nabla \phi) \|^2 + C \| \nabla f' (\phi) \|^2 + C \| \nabla [ \rho' (\phi) \tfrac{|u|^2}{2} ] \|^2 \,,
	\end{aligned}
\end{equation}
where $c_3 = \tfrac{8}{\gamma \lambda} > 0$. By Lemma \ref{Lmm-Calculus}, there further hold
\begin{equation}\label{ND-phi-2}
	\begin{aligned}
		\| \nabla (u \cdot \nabla \phi) \|^2 \lesssim & \| \nabla u \nabla \phi \|^2 + \| u \nabla^2 \phi \|^2 \lesssim \| \nabla u \|^2_{L^4} \| \nabla \phi \|^2_{L^4} + \| u \|^2_{L^\infty} \| \nabla^2 \phi \|^2 \\
		\lesssim & \| \nabla u \|^\frac{1}{2} \| \Delta u \|^\frac{3}{2} ( \| \nabla \phi \|^\frac{1}{2} \| \Delta \phi \|^\frac{3}{2} + \| \nabla \phi \|^2 ) + \| \Delta u \| \| \nabla u \| \| \Delta \phi \|^2 \\
		\lesssim & \| \Delta u \|^\frac{3}{2} E_0^\frac{5}{4} (t) + \| \Delta u \| E_0^\frac{3}{2} (t) \,.
	\end{aligned}
\end{equation}
Moreover, by \eqref{equ-2} and $-1 \leq \phi \leq 1$, one easily knows
\begin{equation}\label{ND-phi-3}
	\begin{aligned}
		\| \nabla f' (\phi) \|^2 \lesssim \| \nabla \phi \|^2 \lesssim E_0 (t) \,.
	\end{aligned}
\end{equation}
Furthermore, it infers from Lemma \ref{Lmm-Calculus} that
\begin{equation}\label{ND-phi-4}
	\begin{aligned}
		\| \nabla [ \rho' (\phi) \tfrac{|u|^2}{2} ] \|^2 \lesssim & \| \nabla \phi \|_{L^3}^2 \| u \|^4_{L^3} + \| \nabla u \|^2_{L^4} \| u \|^2_{L^4} \\
		\lesssim & \| u \|^2 \| \nabla u \|^2 ( \| \nabla \phi \| \| \Delta \phi \| + \| \nabla \phi \|^2 ) + \| \nabla u \|^\frac{1}{2} \| \Delta u \|^\frac{3}{2} \| u \|^\frac{1}{2} \| \nabla u \|^\frac{3}{2} \\
		\lesssim & E_0^3 (t) + \| \Delta u \|^\frac{3}{2} E_0^\frac{5}{4} (t) \,.
	\end{aligned}
\end{equation}
From substituting the bounds \eqref{ND-phi-2}, \eqref{ND-phi-3} and \eqref{ND-phi-4} into \eqref{ND-phi-1}, it thereby infers that
\begin{equation}\label{ND-phi-bnd}
	\begin{aligned}
		\| \nabla \Delta \phi \|^2 \leq & \tfrac{1}{2} c_3 \| \nabla \phi_t \|^2 + C \| \Delta u \|^\frac{3}{2} E_0^\frac{5}{4} (t) + C \| \Delta u \| E_0^\frac{3}{2} (t) + C E_0 (t) + C E_0^3 (t) \\
		\leq & \tfrac{1}{2} c_3 \| \nabla \phi_t \|^2 + \tfrac{1}{4} \| \Delta u \|^2 + C ( 1 + E_0^4 (t) ) E_0 (t) \,.
	\end{aligned}
\end{equation}
Then, \eqref{Delta-u-bnd} and \eqref{ND-phi-bnd} imply that
\begin{equation}\label{H2-8}
	\begin{aligned}
		\| \Delta u \|^2 + \| \nabla \Delta \phi \|^2 + \| p \|^2_1 \leq c_2 \| u_t \|^2_{L^2_{\rho (\phi)}} + c_3 \| \nabla \phi_t \|^2  + C (1 + E_0^4 (t)) E_0 (t) \,.
	\end{aligned}
\end{equation}

Now, in \eqref{H2-7}, we first take $\delta_1 > 0$ such that $c_1 \delta_1 \leq \tfrac{1}{4}$ and $2 \delta_1 \leq \tfrac{1}{4}$. Then we choose a constant $\kappa > 0$ such that $\kappa c_2 \leq \tfrac{1}{2}$ and $\kappa c_3 \leq \tfrac{1}{4}$. We finally take $\delta > 0$ in \eqref{H2-7} such that $2 C_{\delta_1} \delta^2 = \tfrac{1}{2} \kappa > 0$. Therefore, from adding \eqref{H2-7} to the $\kappa$ times of \eqref{H2-8}, one gains
\begin{equation*}
	\begin{aligned}
		& \tfrac{1}{2} \tfrac{d}{dt} \big( \| u \|_{ L_{\rho(\phi)}^2 }^2 + \mu \| \nabla u \|^2 + \| \phi \|^2 + (\gamma \lambda+1) \| \nabla \phi \|^2 + \gamma \lambda \| \Delta \phi \|^2 \big) + \mu \|\nabla u\|^2 + \tfrac{1}{2} \| u_t \|_{ L_{\rho(\phi)}^2 }^2 \\
		& + \gamma \lambda \| \nabla \phi \|^2 +  \tfrac{1}{2} \| \phi_t \|^2 + \gamma \lambda \| \Delta \phi \|^2 + \tfrac{1}{2} \| \nabla \phi_t \|^2 +  \tfrac{\kappa}{2} \|\Delta u\|^2 + \tfrac{\kappa}{2} \| \nabla \Delta \phi \|^2 + \kappa \| p \|^2_1 \\
        & \leq C (1 + E_0^{12} (t)) E_0 (t) \,,
	\end{aligned}
\end{equation*}
which concludes the inequality \eqref{E0D0-bnd}. Therefore, the proof of Lemma \ref{Lmm-H2-t0} is completed.
\end{proof}

\subsection{$H^2$-estimates for $\partial_t^k$-derivatives of ACNS system with integers $k \geq 1$}\label{Subsec:3_3}

Note that the functionals $E_0 (t)$ and $D_0 (t)$ in the a priori estimate in Lemma \ref{Lmm-H2-t0} only involve the second order spatial derivatives of $u$ and the third order spatial derivatives of $\phi$. In order to investigate the information of the higher order spatial derivatives of $(u, \phi)$, we initially dominate the higher order time derivatives of $(u, \phi)$. Then we control the higher order spatial derivative by employing the ADN theory and the structures of the ACNS system. More precisely, the following differential inequality holds.

\begin{lemma}\label{Lmm-H2-tk}
	Let $(u, p, \phi)$ be a sufficiently smooth solution to \eqref{equ-1} over $(t,x) \in [0, T) \times \Omega$. Then there are constant $C > 0$ and small $\kappa > 0$ in the definition of $D_k (t)$ of \eqref{EjDj} such that
	\begin{equation}\label{EkDk-bnd}
		\begin{aligned}
			\tfrac{d}{d t} E_k (t) + D_k (t) \leq C \sum_{0 \leq j \leq k-1} ( 1 + E_j^2 (t) ) D_j (t) + C \sum_{0 \leq j \leq k} (1 + E_j^{12} (t) ) E_j (t)
		\end{aligned}
	\end{equation}
    for all $t \in [0,T)$, where the functionals $E_j (t)$ and $D_j (t)$ for $0 \leq j \leq k$ are defined in \eqref{EjDj}.
\end{lemma}

\begin{proof}
For $k \geq 1$, by applying $\partial_t^k$ to \eqref{equ-1} and combining with the boundary conditions \eqref{BC-HighTimeD}, one gains
\begin{equation}\label{k-equ}
	\left\{
	    \begin{aligned}
	    	\rho (\phi) \partial_t^k u_t + \sum_{1 \leq j \leq k} C_k^j \partial_t^j \rho (\phi) \partial_t^{k-j} u_t + \partial_t^k \big[ \rho (\phi) (u \cdot \nabla u) \big] + \nabla \partial_t^k p \\
	    	= \mu \Delta \partial_t^k u - \lambda \nabla \cdot \partial_t^k ( \nabla \phi \otimes \nabla \phi ) \,, \\[2mm]
	    	\nabla \cdot \partial_t^k u = 0 \,, \\[2mm]
	    	\partial_t^k \phi_t + \partial_t^k (u \cdot \nabla \phi) = \gamma \lambda \Delta \partial_t^k \phi - \gamma \lambda \partial_t^k f' (\phi) - \gamma \partial_t^k \big[ \rho' (\phi) \tfrac{|u|^2}{2} \big] \,, \\[2mm]
	    	\partial_t^k u |_{\partial \Omega} = 0 \,, \quad \tfrac{\partial}{\partial \n} \partial_t^k \phi |_{\partial \Omega} = 0 \,.
	    \end{aligned}
	\right.
\end{equation}
We take $L^2$ inner prouduct of the equation \eqref{k-equ}$_1$ with $\partial_t^k u + \partial_t^k u_t$. Then there holds
\begin{equation*}
	\begin{aligned}
		\int_\Omega \rho (\phi) \partial_t^k u_t ( \partial_t^k u + \partial_t^k u_t ) d x + \int_\Omega \partial_t^k \big[ \rho (\phi) (u \cdot \nabla u) \big] ( \partial_t^k u + \partial_t^k u_t ) d x \\
		+ \int_\Omega \nabla \partial_t^k p ( \partial_t^k u + \partial_t^k u_t ) d x = \mu \int_\Omega \Delta \partial_t^k u ( \partial_t^k u + \partial_t^k u_t ) d x \\
		- \lambda \int_\Omega \nabla \cdot \partial_t^k ( \nabla \phi \otimes \nabla \phi ) ( \partial_t^k u + \partial_t^k u_t ) d x - \sum_{1 \leq j \leq k} C_k^j \int_\Omega \partial_t^j \rho (\phi) \partial_t^{k-j} u_t ( \partial_t^k u + \partial_t^k u_t ) d x \,.
	\end{aligned}
\end{equation*}
By the boundary conditions in \eqref{k-equ}, one has
\begin{equation*}
	\begin{aligned}
		\partial_t^k u |_{\partial \Omega} = \partial_t^k u_t |_{\partial \Omega} = 0 \,.
	\end{aligned}
\end{equation*}
Moreover, $\nabla \cdot \partial_t^k u = 0$ implies $\nabla \cdot \partial_t^k u_t = 0$. It thereby infers from integrating by parts over $x \in \Omega$ that
\begin{equation*}
	\begin{aligned}
		& \tfrac{1}{2} \tfrac{d}{d t} \big( \| \partial_t^k u \|^2_{L^2_{\rho (\phi)}} + \mu \| \nabla \partial_t^k u \|^2 \big) + \| \partial_t^k u_t \|^2_{L^2_{\rho (\phi)}} + \mu \| \nabla \partial_t^k u \|^2 \\
		= & - \langle \partial_t^k K_u , \partial_t^k u + \partial_t^k u_t \rangle + \tfrac{1}{2} \langle \rho' (\phi) \phi_t , |\partial_t^k u|^2 \rangle - \sum_{1 \leq j \leq k} C_k^j \langle \partial_t^j \rho (\phi) \partial_t^{k-j} u_t ,, \partial_t^k u + \partial_t^k u_t \rangle \,,
	\end{aligned}
\end{equation*}
where $K_u$ is given in \eqref{Ku}. Then, in the $\phi$-equation of \eqref{k-equ}, and then take $L^2$-inner product with $\partial_t^k \phi + \partial_t^k \phi_t + \Delta \partial_t^k \phi + \Delta \partial_t^k \phi_t$. Together with the boundary conditions $\tfrac{\partial}{\partial \n } \partial_t^k \phi |_{\partial \Omega} = \tfrac{\partial}{\partial \n } \partial_t^k \phi_t |_{\partial \Omega} = 0$, integrating by parts over $x \in \Omega$ yields that
\begin{equation*}
	\begin{aligned}
		& \tfrac{1}{2} \tfrac{d}{dt} \big( \| \partial_t^k \phi \|^2 + ( \gamma \lambda + ) \| \nabla \partial_t^{k}\phi \|^2 + \gamma \lambda\| \Delta \partial_t^{k} \phi \|^2 \big) \\
		& + \gamma \lambda \| \nabla \partial_t^{k} \phi \|^2 + \| \partial_t^{k} \phi_t \|^2 + \gamma \lambda \| \Delta \partial_t^{k} \phi \|^2 + \| \nabla \partial_t^{k} \phi_t \|^2 \\
        = & \langle \partial_t^{k} K_\phi, \Delta \partial_t^{k} \phi - \partial_t^{k} \phi - \partial_t^{k} \phi_t \rangle - \langle \nabla \partial_t^k K_\phi , \nabla \partial_t^k \phi_t \rangle \,,
	\end{aligned}
\end{equation*}	
where $K_\phi$ is given in \eqref{Kphi}. Consequently, there holds
\begin{equation}\label{H2tk-1}
	\begin{aligned}
		& \tfrac{1}{2} \tfrac{d}{dt} \big( \| \partial_t^{k} u \|_{L^2_{\rho(\phi)}}^2 + \mu \| \nabla \partial_t^{k} u \|^2 + \| \partial_t^{k} \phi \|^2 + ( \gamma \lambda + 1 ) \| \nabla \partial_t^{k}\phi \|^2 + \gamma\lambda\| \Delta \partial_t^{k} \phi \|^2 \big) \\
		& + \mu\| \nabla \partial_t^{k} u \|^2 +\| \partial_t^{k} u_t \|_{L^2_{\rho(\phi)}}^2 + \gamma \lambda \| \nabla \partial_t^{k} \phi \|^2 + \| \partial_t^{k} \phi_t \|^2 + \gamma \lambda \| \Delta \partial_t^{k} \phi \|^2 + \| \nabla \partial_t^{k} \phi_t \|^2 \\
		= & \underbrace{ - \langle \partial_t^k K_u , \partial_t^k u + \partial_t^k u_t \rangle + \langle \partial_t^{k} K_\phi, \Delta \partial_t^{k} \phi - \partial_t^{k} \phi - \partial_t^{k} \phi_t \rangle - \langle \nabla \partial_t^k K_\phi , \nabla \partial_t^k \phi_t \rangle }_{A_1} \\
		& + \underbrace{ \tfrac{1}{2} \langle \rho' (\phi) \phi_t , |\partial_t^k u|^2 \rangle }_{A_2} \ \underbrace{ - \sum_{1 \leq j \leq k} C_k^j \langle \partial_t^j \rho (\phi) \partial_t^{k-j} u_t \,, \partial_t^k u + \partial_t^k u_t \rangle }_{A_3} \,.
	\end{aligned}
\end{equation}

By the similar arguments in \eqref{H2-5}, the term $A_1$ in the right-hand side of \eqref{H2tk-1} can be bounded by
\begin{equation}\label{A1}
	\begin{aligned}
		A_1 \leq \delta_2 \| (\partial_t^k \phi_t, \nabla \partial_t^k \phi_t) \|^2 + c_1 \delta_2 \| \partial_t^k u_t \|^2_{L^2_{\rho (\phi)}} + 2 C_{\delta_2} \delta^2 ( \| \Delta \partial_t^k u \|^2 + \| \nabla \Delta \partial_t^k \phi \|^2 ) \\
		+ \tilde{C}_{\delta_2} \delta^2 \sum_{0 \leq j \leq k-1} ( \| \Delta \partial_t^j u \|^2 + \| \nabla \Delta \partial_t^j \phi \|^2 ) + \tilde{C}_{\delta, \delta_2} \sum_{0 \leq j \leq k} (1 + E_j^{12} (t)) E_j (t)
	\end{aligned}
\end{equation}
for any small $\delta, \delta_2 > 0$. Moreover, the same arguments in \eqref{H2-6} imply that
\begin{equation}\label{A2}
	\begin{aligned}
		A_2 \leq \delta_2' \| (\phi_t, \nabla \phi_t) \|^2 + C_{\delta_2'} E_k^2 (t) \leq \delta_2' D_0 (t) + C_{\delta_2'} E_k^2 (t)
	\end{aligned}
\end{equation}
for small $\delta_2' > 0$ to be determined.

It remains to control the quantity $A_3$ in \eqref{H2tk-1}. By the H\"older inequality and the definition of $E_j (t)$ in \eqref{EjDj}, one has
\begin{equation*}
	\begin{aligned}
		A_3 \leq & C \sum_{1 \leq j \leq k} \| \partial_t^j \rho (\phi) \|_{L^\infty} \| \partial_t^{k-j} u \| \big(\| \partial_t^k u \| + \| \partial_t^k u_t \| \big) \\
		\leq & C \sum_{1 \leq j \leq k} \| \partial_t^j \rho (\phi) \|_{L^\infty} E_{k-j}^\frac{1}{2} (t) \big(E_k^\frac{1}{2} (t) + \| \partial_t^k u_t \|_{L^2_{\rho (\phi)}} \big) \,.
	\end{aligned}
\end{equation*}
Recalling the definition of $\rho (\phi)$ in \eqref{rho-def} and Lemma \ref{Lmm-Calculus}, one has
\begin{equation*}
	\begin{aligned}
		\| \partial_t^j \rho (\phi) \|_{L^\infty} \leq & C \sum_{0 \leq a \leq j} ( 1 + \| \partial_t^a \phi \|_{L^\infty} ) \| \partial_t^a \phi \|_{L^\infty} \\
		\leq & C \sum_{0 \leq a \leq j} ( 1 + \| \partial_t^a \phi \|^\frac{1}{4} \| \Delta \partial_t^a \phi \|^\frac{3}{4} + \| \partial_t^a \phi \| ) ( \| \partial_t^a \phi \|^\frac{1}{4} \| \Delta \partial_t^a \phi \|^\frac{3}{4} + \| \partial_t^a \phi \| ) \\
		\leq & C \sum_{0 \leq a \leq j} ( 1 + E_a^\frac{1}{2} (t) ) E_a^\frac{1}{2} (t) \,.
	\end{aligned}
\end{equation*}
It therefore infers that
\begin{equation}\label{A3}
	\begin{aligned}
		A_3 \leq & C \sum_{1 \leq j \leq k} \sum_{0 \leq a \leq j} ( 1 + E_a^\frac{1}{2} (t) ) E_a^\frac{1}{2} (t) E_{k-j}^\frac{1}{2} (t) \big(E_k^\frac{1}{2} (t) + \| \partial_t^k u_t \|_{L^2_{\rho (\phi)}} \big) \\
		\leq & c_1 \delta_2 \| \partial_t^k u_t \|^2_{L^2_{\rho (\phi)}} + C_{\delta_2} \sum_{0 \leq j \leq k} (1 + E_j^2 (t)) E_j (t)
	\end{aligned}
\end{equation}
for and small $\delta_2 > 0$. From substituting \eqref{A1}, \eqref{A2} and \eqref{A3} into \eqref{H2tk-1}, it is derived that
\begin{equation}\label{H2tk-2}
	\begin{aligned}
		& \tfrac{1}{2} \tfrac{d}{dt} \big( \| \partial_t^{k} u \|_{L^2_{\rho(\phi)}}^2 + \mu \| \nabla \partial_t^{k} u \|^2 + \| \partial_t^{k} \phi \|^2 + ( \gamma \lambda + 1 ) \| \nabla \partial_t^{k}\phi \|^2 + \gamma\lambda\| \Delta \partial_t^{k} \phi \|^2 \big) \\
		& + \mu\| \nabla \partial_t^{k} u \|^2 +\| \partial_t^{k} u_t \|_{L^2_{\rho(\phi)}}^2 + \gamma \lambda \| \nabla \partial_t^{k} \phi \|^2 + \| \partial_t^{k} \phi_t \|^2 + \gamma \lambda \| \Delta \partial_t^{k} \phi \|^2 + \| \nabla \partial_t^{k} \phi_t \|^2 \\
		\leq & \delta_2 \| ( \partial_t^k \phi_t , \nabla \partial_t^k \phi_t ) \|^2 + 2 c_1 \delta_2 \| \partial_t^k u_t \|^2_{L^2_{\rho (\phi)}} + 2 C_{\delta_2} \delta^2 ( \| \Delta \partial_t^k u \|^2 + \| \nabla \Delta \partial_t^k \phi \|^2 ) \\
		& + \delta_2' D_0 (t) + C_{\delta, \delta_2, \delta_2'} \sum_{0 \leq j \leq k} ( 1 + E_j^{12} (t) ) E_j (t) \,.
	\end{aligned}
\end{equation}

Next we turn to control the quantity $\| \Delta \partial_t^k u \|^2 + \| \nabla \Delta \partial_t^k \phi \|^2$ by using the ADN theory and the constitutive of the equations \eqref{k-equ}. From the $u$-equation of \eqref{k-equ}, one has
\begin{equation}
	\begin{aligned}
		- \mu \Delta \partial_t^k u + \nabla \partial_t^k p = & U_k (u, \phi) \quad \ \textrm{in } \Omega \,, \\
		\nabla \cdot \partial_t^k u = & 0 \qquad \qquad \ \textrm{in } \Omega \,, \\
		\partial_t^k u = & 0 \qquad \qquad \ \textrm{on } \partial \Omega \,,
	\end{aligned}
\end{equation}
where
\begin{equation}\label{Uk-u-phi}
	\begin{aligned}
		U_k(u, \phi) = - \partial_t^k \big[ \rho (\phi) ( u_t + u \cdot \nabla u ) \big] - \lambda \nabla \cdot \partial_t^k ( \nabla \phi \otimes \nabla \phi ) \,.
	\end{aligned}
\end{equation}
By the ADN theory in Lemma \ref{Lmm-Stokes}, one derives that
\begin{equation*}
	\begin{aligned}
		\| \Delta \partial_t^k u \|^2 + \| \partial_t^k p \|^2 \leq c_0^2 \| U_k (\phi, u) \|^2 \,.
	\end{aligned}
\end{equation*}
Employing the Sobolev embedding theory in Lemma \ref{Lmm-Calculus} and the bounds \eqref{Low-Bnd-rho-phi} of $\rho (\phi)$, one implies that
	\begin{align*}
		\| U_k (u, \phi) \|^2 \leq & \tfrac{1}{2} \rho_2 \| \partial_t^k u_t \|^2_{L^2_{\rho (\phi)}} + \tfrac{1}{4 c_0^2} \| \Delta \partial_t^k u \|^2 + \tfrac{1}{4 c_0^2} \| \nabla \Delta \partial_t^k \phi \|^2 \\
		+ & C \sum_{0 \leq j \leq k-1} \Big[ \| \Delta \partial_t^j u \|^2 + \| \nabla \Delta \partial_t^j \phi \|^2 \\
		& \qquad \qquad \qquad + ( 1 + \| \partial_t^j \phi \|^4 + \| \Delta \partial_t^j \phi \|^4 ) \| \partial_t^j u_t \|^2_{L^2_{\rho (\phi)}} \Big] \\
		& + C \sum_{0 \leq j \leq k} \Big[ ( 1 + \| \partial_t^j \phi \|^{16} + \| \Delta \partial_t^j \phi \|^{16} ) ( \| \partial_t^j u \|^{10}_{L^2_{\rho (\phi)}} + \| \nabla \partial_t^j u \|^{10} ) \\
		& \qquad \qquad \qquad + ( \| \nabla \partial_t^j \phi \|^8 + \| \Delta \partial_t^j \phi \|^8 ) \| \Delta \partial_t^j \phi \|^2 \Big] \\
		\leq & \tfrac{1}{2} \rho_2 \| \partial_t^k u_t \|^2_{L^2_{\rho (\phi)}} + \tfrac{1}{4} \| \Delta \partial_t^k u \|^2 + \tfrac{1}{4} \| \nabla \Delta \partial_t^k \phi \|^2 \\
		& + C \sum_{0 \leq j \leq k-1} ( 1 + E_j^2 (t) ) D_j (t) + C \sum_{0 \leq j \leq k} (1 + E_j^8 (t) ) E_j^5 (t) \,.
	\end{align*}
Consequently, one has
\begin{equation}\label{H2tk-3}
	\begin{aligned}
		\| \Delta \partial_t^k u \|^2 + \| \partial_t^k p \|^2 \leq & \tfrac{1}{2} c_1 \| \partial_t^k u_t \|^2_{L^2_{\rho (\phi)}} + \tfrac{1}{4} \| \Delta \partial_t^k u \|^2 + \tfrac{1}{4} \| \nabla \Delta \partial_t^k \phi \|^2 \\
		& + C \sum_{0 \leq j \leq k-1} ( 1 + E_j^2 (t) ) D_j (t) + C \sum_{0 \leq j \leq k} (1 + E_j^8 (t) ) E_j^5 (t) \,.
	\end{aligned}
\end{equation}
Here the functionals $E_j (t)$ and $D_j (t)$ are defined in \eqref{EjDj}.

Next we estimate the quantity $\| \nabla \Delta \partial_t^k \phi \|^2$. The $\phi$-equation in \eqref{k-equ} gives us
\begin{equation}\label{H2tk-4}
	\begin{aligned}
		\Delta \partial_t^k \phi = \Phi_k (u, \phi) : = \tfrac{1}{\gamma \lambda} \partial_t^k (\phi_t + u \cdot \nabla \phi) + \partial_t^k \big[ f' (\phi) + \tfrac{1}{\lambda} \rho' (\phi) \tfrac{|u|^2}{2} \big] \,.
	\end{aligned}
\end{equation}
It infers from the Sobolev embedding theory in Lemma \ref{Lmm-Calculus} and the bound $- 1 \leq \phi (t,x) \leq 1$ that
	\begin{align*}
		\| \nabla \partial_t^k (u \cdot \nabla \phi ) \|^2 \leq & C \sum_{0 \leq j \leq k} \big( \| \Delta \partial_t^j u \|^\frac{3}{2} + \| \nabla \Delta \partial_t^j \phi \|^\frac{3}{2} \big) \\
		& \qquad \qquad \times \big( \| \partial_t^j u \|^\frac{5}{2}_{L^2_{\rho (\phi)}} + \| \nabla \partial_t^j u \|^\frac{5}{2} + \| \nabla \partial_t^j \phi \|^\frac{5}{2} + \| \Delta \partial_t^j \phi \|^\frac{5}{2} \big) \\
		\leq & C \sum_{0 \leq j \leq k} E_j^\frac{5}{4} (t) \big( \| \Delta \partial_t^j u \|^\frac{3}{2} + \| \nabla \Delta \partial_t^j \phi \|^\frac{3}{2} \big) \,,
	\end{align*}
and
\begin{equation*}
	\begin{aligned}
		\| \nabla \partial_t^k f' (\phi) \|^2 \leq C \sum_{0 \leq j \leq k} \big( 1 + \| \partial_t^j \phi \|^4 + \| \Delta \partial_t^j \phi \|^4 \big) \| \nabla \partial_t^j \phi \|^2 \leq C \sum_{0 \leq j \leq k} (1 + E_j^2 (t) ) E_j (t) \,,
	\end{aligned}
\end{equation*}
and
\begin{equation*}
	\begin{aligned}
		\| \nabla \partial_t^k [ \rho' (\phi) \tfrac{|u|^2}{2} ] \|^2 \leq & C \sum_{0 \leq j \leq k} \big( \| \nabla \partial_t^j \phi \|^2 + \| \Delta \partial_t^j \phi \|^2 \big) \| \nabla \partial_t^j u \|^4 \\
		+ C & \sum_{0 \leq j \leq k} (1 + \| \partial_t^j \phi \|^2 + \| \Delta \partial_t^j \phi \|^2 ) ( \| \partial_t^j u \|^\frac{5}{2}_{L^2_{\rho (\phi)}} + \| \nabla \partial_t^j u \|^\frac{5}{2} ) \| \Delta \partial_t^j u \|^\frac{3}{2} \\
		\leq & C \sum_{0 \leq j \leq k} \big[ E_j^3 (t) + (1 + E_j (t)) E_j^\frac{5}{4} (t) \| \Delta \partial_t^j u \|^\frac{3}{2} \big] \,.
	\end{aligned}
\end{equation*}
One thereby has
\begin{equation}\label{H2tk-5}
	\begin{aligned}
		\| \nabla \Phi_k (u, \phi) \|^2 \leq & \tfrac{1}{2} c_3 \| \nabla \partial_t^k \phi_t \|^2 + C \sum_{0 \leq j \leq k} (1 + E_j^2 (t) ) E_j (t) \\
		& + C \sum_{0 \leq j \leq k} (1 + E_j (t)) E_j^\frac{5}{4} (t) \big( \| \Delta \partial_t^j u \|^\frac{3}{2} + \| \nabla \Delta \partial_t^j \phi \|^\frac{3}{2} \big) \,.
	\end{aligned}
\end{equation}
Then, the Young's inequality and \eqref{H2tk-4}-\eqref{H2tk-5} imply that
	\begin{align}\label{H2tk-6}
		\no \| \nabla \Delta \partial_t^k \phi \|^2 \leq & \tfrac{1}{2} c_3 \| \nabla \partial_t^k \phi_t \|^2 + \tfrac{1}{4} \| \Delta \partial_t^k u \|^2 + \tfrac{1}{4} \| \nabla \Delta \partial_t^k \phi \|^2 \\
		\no & + \sum_{0 \leq j \leq k-1} ( \| \Delta \partial_t^j u \|^2 + \| \nabla \Delta \partial_t^j \phi \|^2 ) + C \sum_{0 \leq j \leq k} (1 + E_j^8 (t) ) E_j (t) \\
		\no \leq & \tfrac{1}{2} c_3 \| \nabla \partial_t^k \phi_t \|^2 + \tfrac{1}{4} \| \Delta \partial_t^k u \|^2 + \tfrac{1}{4} \| \nabla \Delta \partial_t^k \phi \|^2 \\
		& +  \sum_{0 \leq j \leq k-1} D_j (t) + C \sum_{0 \leq j \leq k} (1 + E_j^8 (t) ) E_j (t) \,.
	\end{align}
It there follows from \eqref{H2tk-3} and \eqref{H2tk-6} that
\begin{equation}\label{H2tk-7}
	\begin{aligned}
		\| \Delta \partial_t^k u \|^2 + \| \nabla \Delta \partial_t^k \phi \|^2 + \| \partial_t^k p \|^2 \leq & c_1 \| \partial_t^k u_t \|^2_{L^2_{\rho (\phi)}} + c_3 \| \nabla \partial_t^k \phi_t \|^2 \\
		+ C \sum_{0 \leq j \leq k-1} & ( 1 + E_j^2 (t) ) D_j (t) + C \sum_{0 \leq j \leq k} (1 + E_j^{12} (t) ) E_j (t) \,.
	\end{aligned}
\end{equation}

Now, in \eqref{H2tk-2}, we first take $\delta_2 > 0$ such that $\delta_2 \leq \tfrac{1}{4}$ and $2 c_1 \delta_2 \leq \tfrac{1}{4}$. Then we take a constant $\kappa > 0$ such that $\kappa c_1 \leq \tfrac{1}{4}$ and $\kappa c_3 \leq \tfrac{1}{4}$. We finally take $\delta > 0$ in \eqref{H2tk-2} such that $2 C_{\delta_2} \delta^2 = \tfrac{1}{2} \kappa > 0$. From adding \eqref{H2tk-2} to the $\kappa$ times of \eqref{H2tk-7}, it follows that
\begin{equation*}
	\begin{aligned}
		& \tfrac{1}{2} \tfrac{d}{dt} \big( \| \partial_t^{k} u \|_{L^2_{\rho(\phi)}}^2 + \mu \| \nabla \partial_t^{k} u \|^2 + \| \partial_t^{k} \phi \|^2 + ( \gamma \lambda + 1 ) \| \nabla \partial_t^{k}\phi \|^2 + \gamma\lambda\| \Delta \partial_t^{k} \phi \|^2 \big) \\
		& + \mu\| \nabla \partial_t^{k} u \|^2 + \tfrac{1}{2} \| \partial_t^{k} u_t \|_{L^2_{\rho(\phi)}}^2 + \gamma \lambda \| \nabla \partial_t^{k} \phi \|^2 + \tfrac{1}{2} \| \partial_t^{k} \phi_t \|^2 + \gamma \lambda \| \Delta \partial_t^{k} \phi \|^2 \\
		& + \tfrac{1}{2} \| \nabla \partial_t^{k} \phi_t \|^2 + \tfrac{\kappa}{2} \| \Delta \partial_t^k u \|^2 + \tfrac{\kappa}{2} \| \nabla \Delta \partial_t^k \phi \|^2 + \kappa \| \partial_t^k p \|^2 \\
		\leq & C \sum_{0 \leq j \leq k-1} ( 1 + E_j^2 (t) ) D_j (t) + C \sum_{0 \leq j \leq k} (1 + E_j^{12} (t) ) E_j (t) \,,
	\end{aligned}
\end{equation*}
which implies the inequality \eqref{EkDk-bnd}. The proof of Lemma \ref{Lmm-H2-tk} is therefore completed.
\end{proof}

\subsection{Estimates for higher order spatial derivatives}\label{Subsec:3_4}

As shown in Lemma \ref{Lmm-H2-t0} and Lemma \ref{Lmm-H2-tk}, the energy $E_j (t)$ and dissipative rates $D_j (t)$ involve at most the third order spatial derivatives. In this subsection, by employing the ADN theory, we aim at investigating the information of the higher order spatial derivatives of the solutions $(u, p, \phi)$ to the ACNS system \eqref{equ-1}. For any fixed integer $\ell \geq 0$, the equations \eqref{equ-1} with the boundary conditions \eqref{BC-ACNS} indicate that
\begin{equation}
	\begin{aligned}
		- \mu \Delta \partial_t^\ell u + \nabla \partial_t^\ell p = & U_\ell (u, \phi) \quad \ \textrm{in } \Omega \,, \\
		\nabla \cdot \partial_t^\ell u = & 0 \qquad \qquad \ \textrm{in } \Omega \,, \\
		\partial_t^\ell u = & 0 \qquad \qquad \ \textrm{on } \partial \Omega \,,
	\end{aligned}
\end{equation}
and
\begin{equation}
	\begin{aligned}
		\Delta \partial_t^\ell \phi = & \Phi_\ell (u, \phi) \qquad \textrm{in } \Omega \,, \\
		\tfrac{\partial}{\partial \n} \partial_t^\ell \phi = & 0 \qquad \qquad \ \ \ \textrm{on } \partial \Omega \,,
	\end{aligned}
\end{equation}
where $U_\ell (u, \phi)$ and $\Phi_\ell (u, \phi)$ are defined in \eqref{Uk-u-phi} and \eqref{H2tk-4}, respectively. Namely,
\begin{equation}\label{U-Phi-l}
	\begin{aligned}
		& U_\ell (u, \phi) = - \partial_t^\ell \big[ \rho (\phi) ( u_t + u \cdot \nabla u ) \big] - \lambda \nabla \cdot \partial_t^\ell ( \nabla \phi \otimes \nabla \phi ) \,, \\
		& \Phi_\ell (u, \phi) = \tfrac{1}{\gamma \lambda} \partial_t^\ell (\phi_t + u \cdot \nabla \phi) + \partial_t^\ell \big[ f' (\phi) + \tfrac{1}{\lambda} \rho' (\phi) \tfrac{|u|^2}{2} \big] \,.
	\end{aligned}
\end{equation}
Then, the ADN theory given in Lemma \ref{Lmm-Stokes}-\ref{Lmm-Laplace} follows that for any integer $s \geq 0$,
\begin{equation}\label{u-U-bnd}
	\begin{aligned}
		\| \partial_t^\ell u \|^2_{s+2} + \| \partial_t^\ell p \|^2_{s+1} \leq C \| U_\ell (u, \phi) \|^2_s \,,
	\end{aligned}
\end{equation}
and
\begin{equation}\label{phi-Phi-bnd}
	\begin{aligned}
		\| \partial_t^\ell \phi \|^2_{s+2} \leq C \big( \| \Phi_\ell (u, \phi) \|^2_s + \| \partial_t^\ell \phi \|^2_1 \big)
	\end{aligned}
\end{equation}
provided that the quantities $\| U_\ell (u, \phi) \|^2_s$ and $\| \Phi_\ell (u, \phi) \|^2_s$ are both finite. Here $\| \partial_t^\ell \phi \|^2_1 \lesssim E_\ell (t)$.

Then, we will control the above two quantities in terms of the functionals $E_j (t)$. For notation simplicity, we denote by
\begin{equation}\label{EUPhi-ls}
	\begin{aligned}
	    \mathbf{E}_\ell (t) : = \sum_{0 \leq j \leq \ell} E_j (t) \,, \ \mathbf{U}_{\ell, s} : = \sum_{0 \leq j \leq \ell} \| U_\ell (u, \phi) \|^2_s \,, \ \bm{\Phi}_{\ell, s} : = \sum_{0 \leq j \leq \ell} \| \Phi_\ell (u, \phi) \|^2_s \,.
	\end{aligned}
\end{equation}
Remark that, for $\ell_1 \leq \ell_2$,
\begin{equation}\label{Monotone-l}
	\begin{aligned}
		\mathbf{E}_{\ell_1} (t) \leq \mathbf{E}_{\ell_2} (t) \,, \quad \mathbf{U}_{\ell_1, s} \leq \mathbf{U}_{\ell_2, s} \,, \quad \bm{\Phi}_{\ell_1, s} \leq \bm{\Phi}_{\ell_2, s} \,,
	\end{aligned}
\end{equation}
and for $s_1 \leq s_2$,
\begin{equation}\label{Monotone-s}
	\begin{aligned}
		\mathbf{U}_{\ell, s_1} \leq \mathbf{U}_{\ell, s_2} \,, \quad \bm{\Phi}_{\ell, s_1} \leq \bm{\Phi}_{\ell, s_2} \,.
	\end{aligned}
\end{equation}

\begin{lemma}\label{Lmm-U-Phi}
	For the integers $\ell \geq 0$ and $s \geq 0$, the following estimates hold:
	\begin{enumerate}
		\item For $s = 0$,
		\begin{equation}\label{UPhi-l-0}
			\begin{aligned}
				\| U_\ell (u, \phi) \|^2 \lesssim & \eps_0 ( \mathbf{U}_{\ell,0} + \bm{\Phi}_{\ell, 1} ) + (1 + \mathbf{E}_{\ell+1}^{12} (t) ) \mathbf{E}_{\ell+1} (t) \,, \\
				\| \Phi_\ell (u, \phi) \|^2 \lesssim & ( 1 + \mathbf{E}_{\ell+1}^2 (t) ) \mathbf{E}_{\ell+1} (t) \,.
			\end{aligned}
		\end{equation}
		
		\item For $s = 1$,
		\begin{equation}\label{UPhi-l-1}
			\begin{aligned}
				\| U_\ell (u, \phi) \|^2_1 \lesssim & \eps_0 ( \mathbf{U}_{\ell,0} + \bm{\Phi}_{\ell, 1} ) + ( 1 + \mathbf{E}_{\ell+1}^2 (t) ) \mathbf{U}^2_{\ell, 0} + \bm{\Phi}_{\ell, 1}^2 + (1 + \mathbf{E}_{\ell+1}^{12} (t) ) \mathbf{E}_{\ell+1} (t) \,, \\
				\| \Phi_\ell (u, \phi) \|^2_1 \lesssim & \eps_0 ( \mathbf{U}_{\ell, 0} + \bm{\Phi}_{\ell, 1} ) + ( 1 + \mathbf{E}_{\ell+1}^8 (t) ) \mathbf{E}_{\ell+1} (t) \,.
			\end{aligned}
		\end{equation}
		
		\item For $s \geq 2$,
		\begin{equation}\label{UPhi-l-s}
			\begin{aligned}
				\| U_\ell (u, \phi) \|_s^2 \lesssim & (1 + \mathbf{E}^2_\ell (t) + \bm{\Phi}_{\ell, s-2} ) ( \mathbf{U}_{\ell + 1, s-2} + \mathbf{U}_{\ell, s-1} ) + \bm{\Phi}_{\ell, s} + \mathbf{E}^2_\ell (t) \,, \\
				\| \Phi_\ell (u, \phi) \|_s^2 \lesssim & \bm{\Phi}_{\ell + 1, s-2} + ( 1 + \mathbf{U}_{\ell, s-2}^2 + \bm{\Phi}_{\ell, s-2}^2 ) \bm{\Phi}_{\ell, s-2} \\
				& \qquad \qquad \qquad  \ \ + ( 1 + \mathbf{E}_\ell (t) ) \mathbf{U}^2_{\ell, s-2} + ( 1 + \mathbf{E}^2_\ell (t) ) \mathbf{E}_\ell (t) \,.
			\end{aligned}
		\end{equation}
	    Here $\eps_0 > 0$ is sufficiently small to determined.
	\end{enumerate}
\end{lemma}

The proof will be given in Appendix \ref{Sec:Appendix} later.

\begin{corollary}\label{Coro-UPhi-bnd}
	Let $\ell \geq 0$ and $s \geq 1$. Then
	\begin{equation}\label{UPhi-s-Bnd}
		\begin{aligned}
			\bm{\Phi}_{\ell, 0} \lesssim (1 + \mathbf{E}_{\ell+1}^2 (t)) \mathbf{E}_{\ell+1} (t) \,, \quad \mathbf{U}_{\ell, s-1} + \bm{\Phi}_{\ell, s} \lesssim (1 + \mathbf{E}_{l+s}^{\aleph_s} (t) ) \mathbf{E}_{\ell+s} (t) \,,
		\end{aligned}
	\end{equation}
    where $\aleph_s : = \tfrac{11}{3} \cdot 4^s - \tfrac{8}{3} > 0$.
\end{corollary}

\begin{proof}
	First, by the definition $\bm{\Phi}_{\ell, 0}$ in \eqref{EUPhi-ls}, the second inequality in \eqref{UPhi-l-0} and \eqref{Monotone-l} implies that
	\begin{equation*}
		\begin{aligned}
			\bm{\Phi}_{\ell, 0} = \sum_{0 \leq j \leq \ell} \| \Phi_j (u, \phi) \|^2 \lesssim \sum_{0 \leq j \leq \ell} (1 + \mathbf{E}_{j+1}^2 (t)) \mathbf{E}_{j+1} (t) \lesssim ( 1 + \mathbf{E}_{\ell+1}^2 (t) ) \mathbf{E}_{\ell+1} (t) \,,
		\end{aligned}
	\end{equation*}
	that is, the first bound in \eqref{UPhi-s-Bnd} holds.
	
	We next control the quantity $\mathbf{U}_{\ell, s-1} + \bm{\Phi}_{\ell, s}$ for $s \geq 1$ as in \eqref{UPhi-s-Bnd} by induction arguments.
	
	\underline{\tt Case 1. $s = 1$.}
	
	Observe that the first inequality in \eqref{UPhi-l-0} reduces to
	\begin{equation*}
		\begin{aligned}
			\mathbf{U}_{\ell, 0} \lesssim \eps_0 ( \mathbf{U}_{\ell,0} + \bm{\Phi}_{\ell, 1} ) + (1 + \mathbf{E}_{\ell+1}^{12} (t) ) \mathbf{E}_{\ell+1} (t) \,.
		\end{aligned}
	\end{equation*}
	Moreover, the second inequality in \eqref{UPhi-l-1} indicates that
	\begin{equation*}
		\begin{aligned}
			\bm{\Phi}_{\ell, 1} \lesssim \eps_0 ( \mathbf{U}_{\ell, 0} + \bm{\Phi}_{\ell, 1} ) + ( 1 + \mathbf{E}_{\ell+1}^8 (t) ) \mathbf{E}_{\ell+1} (t) \,.
		\end{aligned}
	\end{equation*}
	By taking $\eps_0 > 0$ small enough, one has
	\begin{equation}\label{UPhi-1-bnd}
		\begin{aligned}
			\mathbf{U}_{\ell, 0} + \bm{\Phi}_{\ell, 1} \lesssim (1 + \mathbf{E}_{\ell+1}^{\aleph_1} (t) ) \mathbf{E}_{\ell+1} (t) \,,
		\end{aligned}
	\end{equation}
	where $\aleph_1 : = \tfrac{11}{3} \cdot 4^1 - \tfrac{8}{3} = 12 > 0$. Namely, the second bound in \eqref{UPhi-s-Bnd} holds for the case $s = 1$.
	
	\underline{\tt Case 2. $s = 2$.}
	
	The bounds \eqref{Monotone-l}-\eqref{Monotone-s} and \eqref{UPhi-l-1} indicate that
	\begin{equation*}
		\begin{aligned}
			\mathbf{U}_{\ell,1} \lesssim & \eps_0 ( \mathbf{U}_{\ell, 0} + \bm{\Phi}_{\ell, 1} ) + ( 1 + \mathbf{E}_{\ell+1}^2 (t) ) ( \mathbf{U}_{\ell, 0} + \bm{\Phi}_{\ell, 1} )^2 + (1 + \mathbf{E}_{\ell+1}^{12} (t) ) \mathbf{E}_{\ell+1} (t) \\
			\lesssim & (1 + \mathbf{E}_{\ell+1}^{27} (t) ) \mathbf{E}_{\ell+1} (t) \lesssim (1 + \mathbf{E}_{\ell+1}^{\aleph_2} (t) ) \mathbf{E}_{\ell+1} (t) \,,
		\end{aligned}
	\end{equation*}
	due to $\aleph_2 = \tfrac{1}{3} \cdot 4^2 - \tfrac{8}{3} = 56 > 27$. Moreover, by \eqref{UPhi-l-s}, \eqref{Monotone-l}-\eqref{Monotone-s}, the first inequality in \eqref{UPhi-s-Bnd} and \eqref{UPhi-1-bnd}, one has
	\begin{equation*}
		\begin{aligned}
			\bm{\Phi}_{\ell,2} \lesssim & \bm{\Phi}_{\ell+1, 0} + ( 1 + \mathbf{U}_{\ell, 0}^2 + \bm{\Phi}_{\ell, 0}^2 ) \bm{\Phi}_{\ell, 0} + ( 1 + \mathbf{E}_\ell (t) ) \mathbf{U}^2_{\ell, 0} + ( 1 + \mathbf{E}^2_\ell (t) ) \mathbf{E}_\ell (t) \\
			\lesssim & ( 1 + \mathbf{E}^{2 \aleph_1 + 4}_{\ell+2} (t) ) \mathbf{E}_{\ell+2} (t) \lesssim ( 1 + \mathbf{E}^{\aleph_2}_{\ell+2} (t) ) \mathbf{E}_{\ell+2} (t)
		\end{aligned}
	\end{equation*}
	where we have used $\aleph_2 = 56 > 2 \aleph_1 + 4 = 28$. Consequently,
	\begin{equation*}
		\begin{aligned}
			\mathbf{U}_{\ell,1} + \bm{\Phi}_{\ell,2} \lesssim ( 1 + \mathbf{E}^{\aleph_2}_{\ell+2} (t) ) \mathbf{E}_{\ell+2} (t) \,,
		\end{aligned}
	\end{equation*}
	namely, the second inequality in \eqref{UPhi-s-Bnd} holds for $s=2$. It remains to prove the cases $s \geq 3$.
	
	\underline{\tt Case 3. The Induction Hypotheses for $k = 2, 3, \cdots, s -1$.}
	
	Assume that
	\begin{equation}\label{Hypotheses}
		\begin{aligned}
			\mathbf{U}_{\ell, k-1} + \bm{\Phi}_{\ell, k} \lesssim ( 1 + \mathbf{E}^{\aleph_k}_{\ell+k} (t) ) \mathbf{E}_{\ell+k} (t) \,, \ k = 2, 3, \cdots, s -1 \textrm{ and } \ell \geq 0 \,.
		\end{aligned}
	\end{equation}
	
	\underline{\tt Case 4. Consider the case $s $.}
	
	The bounds \eqref{UPhi-l-s} and \eqref{Monotone-l}-\eqref{Monotone-s} indicate that for $s \geq 3$,
	\begin{equation}\label{U-1BND}
		\begin{aligned}
			\mathbf{U}_{\ell, s-1} \lesssim (1 + \mathbf{E}_\ell^2 (t) + \bm{\Phi}_{\ell, s-2}) \mathbf{U}_{\ell+1, s-2} + \bm{\Phi}_{\ell+1, s-1} + \mathbf{E}_\ell^2 (t) \,,
		\end{aligned}
	\end{equation}
	and
	\begin{equation*}
		\begin{aligned}
			\bm{\Phi}_{\ell, s} \lesssim \bm{\Phi}_{\ell+1, s-1} + (1 + \mathbf{E}_\ell (t)) \mathbf{U}_{\ell, s-1}^2 + \mathbf{U}_{\ell, s-2}^3 + \bm{\Phi}_{\ell, s-2}^3 + (1 + \mathbf{E}_\ell^2 (t)) \mathbf{E}_\ell (t) \,.
		\end{aligned}
	\end{equation*}
	One thereby has
	\begin{equation*}
		\begin{aligned}
			\mathbf{U}_{\ell, s-1} + \bm{\Phi}_{\ell, s} \lesssim & (1 + \mathbf{E}_\ell (t)) \big[ ( \mathbf{U}_{\ell+1, s-2} + \bm{\Phi}_{\ell+1, s-1} ) + ( \mathbf{U}_{\ell+1, s-2} + \bm{\Phi}_{\ell+1, s-1} )^3 \big] \\
			& + (1 + \mathbf{E}_\ell (t)) \mathbf{U}_{\ell, s-1}^2 + (1 + \mathbf{E}_\ell^2 (t)) \mathbf{E}_\ell (t) \,.
		\end{aligned}
	\end{equation*}
	Moreover, by \eqref{U-1BND}, it can be easily implied that
	\begin{equation*}
		\begin{aligned}
			(1 + \mathbf{E}_\ell (t)) \mathbf{U}_{\ell, s-1}^2 \lesssim (1 + \mathbf{E}_\ell^5 (t)) \big[ ( \mathbf{U}_{\ell+1, s-2} + \bm{\Phi}_{\ell+1, s-1} ) + ( \mathbf{U}_{\ell+1, s-2} + \bm{\Phi}_{\ell+1, s-1} )^4 \big] \\
			+ (1 + \mathbf{E}_\ell^4 (t)) \mathbf{E}_\ell (t) \,.
		\end{aligned}
	\end{equation*}
	As a consequence, there hold
	\begin{equation*}
		\begin{aligned}
			\mathbf{U}_{\ell, s-1} + \bm{\Phi}_{\ell, s} \lesssim (1 + \mathbf{E}_\ell^5 (t)) \big[ ( \mathbf{U}_{\ell+1, s-2} + \bm{\Phi}_{\ell+1, s-1} ) + ( \mathbf{U}_{\ell+1, s-2} + \bm{\Phi}_{\ell+1, s-1} )^4 \big] \\
			+ (1 + \mathbf{E}_\ell^4 (t)) \mathbf{E}_\ell (t) \,.
		\end{aligned}
	\end{equation*}
	Denote by
	\begin{equation*}
		\begin{aligned}
			\bm{\Xi}_{\ell, s} : = \mathbf{U}_{\ell, s-1} + \bm{\Phi}_{\ell, s} \,, \quad A_\ell : = 1 + \mathbf{E}_\ell^5 (t) \,, \quad B_\ell : = (1 + \mathbf{E}_\ell^4 (t)) \mathbf{E}_\ell (t) \,.
		\end{aligned}
	\end{equation*}
	Then we have
	\begin{equation}\label{InductionRel}
		\begin{aligned}
			\bm{\Xi}_{\ell, s} \lesssim A_\ell ( \bm{\Xi}_{\ell+1, s-1} + \bm{\Xi}_{\ell+1, s-1}^4 ) + B_\ell \,.
		\end{aligned}
	\end{equation}
	As in \eqref{Monotone-l}-\eqref{Monotone-s}, $\bm{\Xi}_{\ell, s}$ admits the same properties of $\mathbf{U}_{\ell,s}$ and $\bm{\Phi}_{\ell,s}$. The quantities $A_\ell$ and $B_\ell$ have the same property of $\mathbf{E}_\ell (t)$ in \eqref{Monotone-l}. Note that the Induction Hypotheses \eqref{Hypotheses} indicate
	\begin{equation*}
		\begin{aligned}
			\bm{\Xi}_{\ell+1, s-1} \lesssim ( 1 + \mathbf{E}^{\aleph_{s-1}}_{\ell+s} (t) ) \mathbf{E}_{\ell+s} (t) \,.
		\end{aligned}
	\end{equation*}
    Then, together with the induction relation \eqref{InductionRel},
	\begin{equation*}
		\begin{aligned}
			\bm{\Xi}_{\ell, s} \lesssim & A_\ell \Big\{ ( 1 + \mathbf{E}^{\aleph_{s-1}}_{\ell+s} (t) ) \mathbf{E}_{\ell+s} (t) + ( 1 + \mathbf{E}^{\aleph_{s-1}}_{\ell+s} (t) )^4 \mathbf{E}^4_{\ell+s} (t) \Big\} + B_\ell \\
			\lesssim & A_{\ell + s} ( 1 + \mathbf{E}_{\ell+s}^{4 \aleph_{s-1} + 3} (t) ) \mathbf{E}_{\ell+s} (t) + B_{\ell+s} \\
			\lesssim & ( 1 + \mathbf{E}_{\ell+s}^{4 \aleph_{s-1} + 8} (t) ) \mathbf{E}_{\ell+s} (t) \,,
		\end{aligned}
	\end{equation*}
	where the last second inequality is derived from the facts $A_\ell \leq A_{\ell+s}$ and $B_\ell \leq B_{\ell+s}$. Observe that $\aleph_s = \tfrac{11}{3} \cdot 4^s - \tfrac{8}{3}$ satisfies $\aleph_s = 4 \aleph_{s-1} + 8$. It then follows that
	\begin{equation*}
		\begin{aligned}
			\bm{\Xi}_{\ell, s} \lesssim  ( 1 + \mathbf{E}_{\ell+s}^{\aleph_s} (t) ) \mathbf{E}_{\ell+s} (t) \,.
		\end{aligned}
	\end{equation*}
	Then the Induction Principle concludes the second bound \eqref{UPhi-s-Bnd}. The proof of Corollary \ref{Coro-UPhi-bnd} is completed.
\end{proof}

By the bounds \eqref{u-U-bnd}-\eqref{phi-Phi-bnd} and Corollary \ref{Coro-UPhi-bnd}, one knows that for $s \geq 1$,
\begin{equation*}
	\begin{aligned}
		\| \partial_t^\ell u \|^2_{s+1} + \| \partial_t^\ell p \|^2_s + \| \partial_t^\ell \phi \|^2_{s+2} \lesssim \mathbf{U}_{\ell, s-1} + \bm{\Phi}_{\ell, s} + \mathbf{E}_\ell (t) \lesssim \big( 1 + \mathbf{E}_{\ell + s}^{\aleph_s} (t) \big) \mathbf{E}_{\ell + s} (t) \,,
	\end{aligned}
\end{equation*}
and by the definition of $\mathbf{E}_\ell (t)$ in \eqref{EUPhi-ls},
\begin{equation*}
	\begin{aligned}
		\| \partial_t^\ell u \|^2_1 + \| \partial_t^\ell p \|^2 + \| \partial_t^\ell \phi \|^2_2 \lesssim \mathbf{E}_\ell (t) + \| \partial_t^\ell p \|^2_1 \lesssim \mathbf{E}_\ell (t) + \big( 1 + \mathbf{E}_{\ell + 1}^{\aleph_1} (t) \big) \mathbf{E}_{\ell + 1} (t) \,.
	\end{aligned}
\end{equation*}
As a result, for all $\ell + s \leq \Lambda$,
\begin{equation*}
	\begin{aligned}
		\sum_{\ell + s \leq \Lambda} \big( \| \partial_t^\ell u \|^2_{s+1} + \| \partial_t^\ell p \|^2_s + \| \partial_t^\ell \phi \|^2_{s+2} \big) \lesssim \sum_{\ell + s \leq \Lambda} \big( 1 + \mathbf{E}_{\ell + s}^{\aleph_s} (t) \big) \mathbf{E}_{\ell + s} (t) \\
		\lesssim \big( 1 + \mathbf{E}_\Lambda^{\aleph_\Lambda} (t) \big) \mathbf{E}_\Lambda (t) \,.
	\end{aligned}
\end{equation*}
Consequently, we summarize the following results.

\begin{lemma}[Estimates for time-spatial mixed derivatives]\label{Lmm-HSDE}
	Let $0 < T \leq \infty$ and integer $\Lambda \geq 2$. Assume that $(u,p, \phi)$ is a sufficiently smooth solution to the ACNS system \eqref{equ-1} over $t \in [0, T)$ and the smooth bounded domain $\Omega$ with the boundary conditions \eqref{BC-ACNS}. Let $\mathbf{E}_\Lambda (t)$ be given in \eqref{EUPhi-ls}.  If $\sup\limits_{0 \leq t < T} \mathbf{E}_\Lambda (t) < \infty$, then there is a constant $C > 0$ such that
	\begin{equation*}
		\sum_{\ell + s \leq \Lambda} \Big( \| \partial_t^\ell u \|^2_{s+1} + \| \partial_t^\ell p \|^2_s + \| \partial_t^\ell \phi \|^2_{s+2} \Big) (t) \leq C \big( 1 + \mathbf{E}_\Lambda^{\aleph_\Lambda} (t) \big) \mathbf{E}_\Lambda (t) \,,
	\end{equation*}
	where $\aleph_\Lambda = \tfrac{11}{3} \cdot 4^\Lambda - \tfrac{8}{3} > 0$.
\end{lemma}

\section{Local well-posedness of \eqref{equ-1}}\label{Sec:Local}

In this section, we prove the local well-posedness of the ACNS equations \eqref{equ-1} with boundary values \eqref{BC-ACNS} and initial data \eqref{equ-3}. We first construct the linear approximate system by iteration scheme. The key step is to prove the existence of the uniform positive time lower bound to the approximate system and the uniform energy bounds based on the a priori estimates in Lemmas \ref{Lmm-H2-t0}-\ref{Lmm-H2-tk}-\ref{Lmm-HSDE}. Finally, by the compactness arguments, we can justify the local existence results of \eqref{equ-1}.

We first construct the approximate system by the iteration as follows: for all integer $ n \geq 0$,
\begin{equation}\label{ias}
	\left\{
	\begin{array}{l}
	\rho (\phi^n) (	\partial_{t} u^{n+1}  + u^{n} \cdot \nabla u^{n+1} ) + \nabla p^{n+1} = \nabla \cdot ( \mu \nabla u^{n+1} - \lambda \nabla \phi^n \otimes \nabla\phi^n ) \,, \\[2mm]
		\qquad \qquad \nabla \cdot u^{n+1} = 0 \,, \\[2mm]
   ( \dot\phi^{n+1} = ) \ \partial_t \phi^{n+1} + u^n \cdot \nabla \phi^{n+1} = \gamma ( \lambda \Delta \phi^{n+1} - \lambda f'(\phi^n) - \rho' (\phi^n) \tfrac{|u^n|^2}{2}) \,, \\[2mm]
		( u^{n+1}, \phi^{n+1} ) |_{t=0} = ( u^{in} (x), \phi^{in} (x) ) \in \R^3 \times \R \,, \\[2mm]
		u^{n+1}|_{\partial\Omega}=0, \ \tfrac{\partial}{\partial \n} \phi^{n+1}|_{\partial\Omega}=0 \,.
	\end{array}
	\right.
\end{equation}
The iteration starts from
\begin{equation}\label{iasi}
	(u^0(t,x), \phi^0(t,x)) \equiv (u^{in}(x), \phi^{in}(x)) \,.
\end{equation}

By the standard linear theory, there hold:
\begin{lemma}\label{Lmm-AppSyst}
	Suppose that $ \Lambda \geq 2$ and the initial data $ ( u^{in}, \phi^{in} ) (x) $ satisfies $u^{in} \in H^{2\Lambda + 2} $ and $ \phi^{in} \in H^{2\Lambda + 3}$ with $- 1 \leq \phi^{in} \leq 1$. Then there is a maximal number $T^{*}_{n+1}>0$ such that the system \eqref{ias} admits a unique solution $(u^{n+1}, p^{n+1} \phi^{n+1}) (t,x)$ satisfying $-1 \leq \phi^{n+1} (t,x) \leq 1$ and
	\begin{equation*}
		\begin{aligned}
			& \partial_t^\ell u^{n+1} \in C (0, T_{n+1}^*; H^{\Lambda - \ell + 1}) \cap L^2 (0, T_{n+1}^*; H^2) \,, \partial_t^\ell u_t^{n+1} \in L^2 (0, T_{n+1}^*; L^2_{\rho (\phi^n)}) \,, \\
			& \partial_t^\ell p^{n+1} \in C (0, T_{n+1}^*; H^{\Lambda - \ell}) \,, \ \partial_t^\ell \phi^{n+1} \in C (0, T_{n+1}^*; H^{\Lambda - \ell + 2}) \,, \partial_t^\ell \phi_t^{n+1} \in L^2 (0, T_{n+1}^*; H^1)
		\end{aligned}
	\end{equation*}
    for $0 \leq \ell \leq \Lambda$.
\end{lemma}

We remark that $T^{*}_{n+1} \leq T^{*}_{n}$.

The next goal is to prove existence of the uniform lower bound $T>0$ of the time sequence $T_{n+1}^*$ of the approximate system \eqref{ias}-\eqref{iasi}. We introduce the approximate energy functional $E_{j,n+1}(t)$ and the approximate dissipation functional $D_{j, n+1}(t)$ as follows:
\begin{equation*}
	\begin{aligned}
		 {E}_{j,n+1}(t) = & \| \partial_t^{j} u^{n+1} \|_{L^2_{\rho(\phi^n)}}^2 + \mu \| \nabla \partial_t^{j} u^{n+1} \|^2 + \| \partial_t^{j} \phi^{n+1} \|^2  \\
		 & + ( \gamma \lambda + 1 ) \| \nabla \partial_t^{j}\phi ^{n+1}\|^2 + \gamma\lambda\| \Delta \partial_t^{j} \phi^{n+1} \|^2 \,,
	\end{aligned}
\end{equation*}
and
\begin{align*}
   {D}_{j,n+1}(t) = & \mu\| \nabla \partial_t^{j} u^{n+1} \|^2 +\| \partial_t^{j} u^{n+1}_t \|_{L^2_{\rho(\phi^n)}}^2 + \gamma \lambda \| \nabla \partial_t^{j} \phi^{n+1} \|^2 + \| \partial_t^{j} \phi^{n+1}_t \|^2 + \gamma \lambda \| \Delta \partial_t^{j} \phi^{n+1} \|^2 \\
   & + \| \nabla \partial_t^{j} \phi^{n+1}_t \|^2 + \kappa \| \partial_t^{j} \Delta u^{n+1} \|^2 + \kappa \|\partial_t^{j} \nabla \Delta \phi^{n+1} \|^2 + \kappa \| \partial_t^j p^{n+1} \|^2_1 \,.
\end{align*}
Here $j \geq 0$, and $\kappa > 0$ is given in Lemma \ref{Lmm-H2-t0} and Lemma \ref{Lmm-H2-tk}.

\begin{lemma}\label{Lmm-UnfBnd}
Let $(u^{n+1}, p^{n+1}, \phi^{n+1})$ be the solution to the iterative approximate system \eqref{ias}-\eqref{iasi} constructed in Lemma \ref{Lmm-AppSyst}. Then, for sufficiently large $M \gg E_\Lambda^{in} : = \| u^{in} \|^2_{2 \Lambda + 2} + \| \phi^{in} \|^2_{2 \Lambda + 3}$, there is a $T>0$, depending only on the all coefficients, $M$ and $E_\Lambda^{in}$, such that
\begin{equation}
	\begin{aligned}
		\sum_{j=0}^\Lambda E_{j,n+1} (t) + \sum_{j=0}^\Lambda \int_0^t D_{j,n+1}(t) dt \leq M
	\end{aligned}
\end{equation}
holds for all $n \geq 0$ and $t \in [0, T]$.
\end{lemma}

\begin{proof}
	From the similar arguments in Lemma \ref{Lmm-H2-t0} and Lemma \ref{Lmm-H2-tk}, one easily knows that for all $t \in [0, T_{n+1}^* )$,
	\begin{equation}\label{E0-n+1}
		\begin{aligned}
			\tfrac{d}{d t} E_{0, n+1} (t) + D_{0, n+1} (t) \leq C \big( 1 + E_{0, n+1}^{12} (t) \big) E_{0, n+1} (t) \,,
		\end{aligned}
	\end{equation}
	and
	\begin{equation}\label{Ek-n+1}
		\begin{aligned}
			\tfrac{d}{d t} E_{k, n+1} (t) + D_{k, n+1} (t) \leq & C \sum_{0 \leq j \leq k} \big( 1 + E_{j, n+1}^{12} (t) \big) E_{j, n+1} (t) \\
			& + C \sum_{0 \leq j \leq k-1} (1 + E_{j, n+1}^2 (t)) D_{j, n+1} (t)
		\end{aligned}
	\end{equation}
	with $1 \leq k \leq \Lambda$. Note that $E_{k, n+1} (0) \leq C_k E_\Lambda^{in}$ for $0 \leq k \leq \Lambda$ and some $C_k > 0$, where the quantity $E_\Lambda^{in}$ is defined in Theorem \ref{th1}. We will take $M_0$, $M_1$, $\cdots$, $M_\Lambda$ such that
	\begin{equation}\label{Mk}
		\begin{aligned}
			M = \sum_{k=0}^\Lambda M_k \,, \quad M_k \gg C_k E_\Lambda^{in} \geq  E_{k,n+1}(0) \geq 0 \,, \quad k = 0, 1, \cdots, \Lambda \,.
		\end{aligned}
	\end{equation}
	For $k=0, 1, \cdots, \Lambda$, define
	\begin{equation}\label{Tk-n+1}
		\begin{aligned}
			T_{n+1}^k : = \sup \Big\{ \tau \in [0, T_{n+1}^{k-1} ); \sup_{t\in[0,\tau]} E_{k, n+1} (t) + \int_0^\tau D_{k, n+1} (t) d \tau \leq M_k \Big\} \geq 0 \,,
		\end{aligned}
	\end{equation}
	where we have used the convection $T_{n+1}^{-1} : = T_{n+1}^* > 0$. It is easy to see that
	\begin{equation*}
		\begin{aligned}
			0 \leq T_{n+1}^\Lambda \leq \cdots \leq T_{n+1}^1 \leq T_{n+1}^0 \leq T_{n+1}^* \,.
 		\end{aligned}
	\end{equation*}

	By the continuity of the functionals $E_{k,n+1}(t)$ guaranteed by Lemma \ref{Lmm-AppSyst}, the fact \eqref{Mk} implies
	\begin{equation}\label{Tk-n+1+0}
		\begin{aligned}
			T_{n+1}^k > 0 \quad \textrm{for } k =0, 1, \cdots, \Lambda \,.
		\end{aligned}
	\end{equation}
	Then, \eqref{E0-n+1} and \eqref{Tk-n+1} indicate that
	\begin{equation*}
		\begin{aligned}
			E_{0, n+1} (t) + \int_0^t D_{0, n+1} (t') d t' \leq E_{0, n+1} (0) + P_0 t \leq C_0 E_\Lambda^{in} + P_0 t \quad (\forall t \in [0, T_{n+1}^0 )) \,,
		\end{aligned}
	\end{equation*}
	where $P_0 : = C (1 + M_0^{12}) M_0 > 0$. Together with \eqref{Ek-n+1} and \eqref{Tk-n+1}, we inductively establish that for $k = 1, \cdots, \Lambda$,
	\begin{equation*}
		\begin{aligned}
			E_{k, n+1} (t) + \int_0^t D_{k, n+1} (t') d t' \leq & E_{k, n+1} (0) + Q_{k-1} \sum_{0 \leq j \leq k-1}  E_{j, n+1} (0) + P_k t \\
			\leq & C_k E_\Lambda^{in} + Q_{k-1} \sum_{0 \leq j \leq k-1} C_j E_\Lambda^{in} + P_k t \quad (\forall t \in [0, T_{n+1}^k ) )
		\end{aligned}
	\end{equation*}
	where the constants $Q_{k-1} = Q_j (M_0, \cdots, M_{k-1}) > 0$ and $P_k = P_k (M_0, \cdots, M_k) > 0$.
	
	We now take
	\begin{equation*}
		\begin{aligned}
			M_0 > \Upsilon_0 : = C_0 E_\Lambda^{in} \,, \ M_k > \Upsilon_k : = C_k E_\Lambda^{in} + Q_{k-1} \sum_{0 \leq j \leq k-1} C_j E_\Lambda^{in} \,, \ k =1,2, \cdots , \Lambda \,.
		\end{aligned}
	\end{equation*}
	We then choose
	\begin{equation*}
		\begin{aligned}
			t_\star : = \min \big\{ \tfrac{M_k - \Upsilon_k}{2 P_k} ; k = 0, 1, \cdots, \Lambda \big\} > 0 \,,
		\end{aligned}
	\end{equation*}
	which means that $P_k t < \tfrac{M_k - \Upsilon_k}{2} $ for $t \in [0, t_\star )$ and $k = 0, 1, \cdots, \Lambda$. Consequently, one has
	\begin{equation*}
		\begin{aligned}
			E_{k, n+1} (t) + \int_0^t D_{k, n+1} (t') d t' \leq \Upsilon_k + P_k t < \Upsilon_k + \tfrac{M_k - \Upsilon_k}{2} = \tfrac{M_k + \Upsilon_k}{2} < M_k
		\end{aligned}
	\end{equation*}
	for $k = 0, 1, \cdots, \Lambda$ and $0 \leq t < t_\star$. By the definitions of $T_{n+1}^k$ in \eqref{Tk-n+1}, one sees
	\begin{equation*}
		\begin{aligned}
			T_{n+1}^k \geq t_\star > 0 \,, \ k = 0, 1, \cdots, \Lambda \,.
		\end{aligned}
	\end{equation*}
	We also take $M = M_0 + \cdots + M_\Lambda > \Upsilon_0 + \cdots \Upsilon_\Lambda$. Consequently, the conclusions in Lemma \ref{Lmm-UnfBnd} hold.
\end{proof}

\noindent{\bf Proof of the Theorem\ref{th1}: Local well-posedness.}
By Lemma \ref{Lmm-UnfBnd}, we know that for any fixed $M \gg E^{in}_\Lambda$ given in Lemma \ref{Lmm-UnfBnd}, there is a $T>0$ such that for all integer $n\geq 0$ and $t\in[0,T]$,
\begin{equation}\label{ED-bnd-M}
	\begin{split}
		\sum_{j=0}^\Lambda E_{j,n+1} (t) + \sum_{j=0}^\Lambda \int_0^t D_{j,n+1}(t) dt \leq M \,.
	\end{split}
\end{equation}
Moreover, by the similar arguments in Lemma \ref{Lmm-HSDE} and the above uniform bound \eqref{ED-bnd-M}, one has
\begin{equation*}
	\begin{aligned}
		\sum_{\ell + s \leq \Lambda} \Big( \| \partial_t^\ell u^{n+1} \|^2_{s+1} + \| \partial_t^\ell p^{n+1} \|^2_s + \| \partial_t^\ell \phi^{n+1} \|^2_{s+2} \Big) (t) \leq C (\Lambda, M) < \infty
	\end{aligned}
\end{equation*}
uniformly in $t \in [0, T]$ and $n \geq 0$. Then by compactness arguments and Arzela-Ascoli Theorem, we obtain that the system \eqref{equ-1}-\eqref{equ-2} with boudnary conditions \eqref{BC-ACNS} admits a solution $(u, p, \phi) (t,x)$ satisfying
\begin{equation*}
	\begin{aligned}
		& \partial_t^\ell u \in L^\infty (0, T; H^{\Lambda - \ell + 1}) \cap L^2 (0, T; H^2) \,, \partial_t^\ell u_t \in L^2 (0, T; L^2_{\rho (\phi)}) \,, \\
		& \partial_t^\ell p \in L^\infty (0, T; H^{\Lambda - \ell}) \,, \ \partial_t^\ell \phi \in L^\infty (0, T; H^{\Lambda - \ell + 2}) \,, \partial_t^\ell \phi_t \in L^2 (0, T; H^1)
	\end{aligned}
\end{equation*}
for $0 \leq \ell \leq \Lambda$. Moreover, there hold
\begin{equation*}
	\begin{split}
		\sum_{j=0}^\Lambda E_j (t) + \sum_{j=0}^\Lambda \int_0^t D_j (t) dt \leq M
	\end{split}
\end{equation*}
and
\begin{equation}\label{Energy-Bnd}
	\begin{aligned}
		\sum_{\ell + s \leq \Lambda} \Big( \| \partial_t^\ell u \|^2_{s+1} + \| \partial_t^\ell p \|^2_s + \| \partial_t^\ell \phi \|^2_{s+2} \Big) (t) \leq C (\Lambda, M)
	\end{aligned}
\end{equation}
uniformly in $t \in [0,T]$. Since $-1 \leq \phi^{in} (x) \leq 1$, the maximal principle of the parabolic equation shows that $- 1 \leq \phi (t,x) \leq 1$. Hence the existence result in Theorem \ref{th1} holds.

Then we will prove the uniqueness of the solution to \eqref{equ-1}-\eqref{equ-2} with boundary conditions \eqref{BC-ACNS}. Assume that $(u_1, p_1, \phi_1)$ and $(u_1, p_1, \phi_1)$ are the two solutions to \eqref{equ-1}-\eqref{equ-2}-\eqref{BC-ACNS}. Denote by
\begin{equation*}
	\begin{aligned}
		u^d : = u_1 - u_2 \,, \ p^d : = p_1 - p_2 \,, \ \phi^d : = \phi_1 - \phi_2 \,.
	\end{aligned}
\end{equation*}
Then $(u^d, p^d, \phi^d)$ subjects to the following system
\begin{equation}\label{Diff-eq}
	\left\{
	    \begin{aligned}
	    	\rho (\phi_1) \big( \partial_t u^d + u^d \cdot \nabla u_1 + u_2 \cdot \nabla u^d \big) + & [ \rho (\phi_1) - \rho (\phi_2) ] ( \partial_t u_2 + u_2 \cdot \nabla u_2 ) + \nabla p^d \\
	    	= & \mu \Delta u^d - \lambda \nabla \cdot ( \nabla \phi^d \otimes \nabla \phi_1 + \nabla \phi_2 \otimes \nabla \phi^d ) \,, \\
	    	\nabla \cdot u^d = & 0 \,, \\
	    	\partial_t \phi^d + u^d \cdot \nabla \phi_1 + u_2 \cdot \nabla \phi^d = & \gamma \lambda \Delta \phi^d - \gamma \lambda [f' (\phi_1) - f' (\phi_2)] \\
	    	- \gamma [\rho' & (\phi_1) - \rho' (\phi_2) ] \tfrac{|u_1|^2}{2} - \tfrac{1}{2} \gamma \rho' (\phi_2) (u_1 + u_2) \cdot u^d \,, \\
	    	u^d |_{\partial \Omega} = & 0 \,, \quad \tfrac{\partial}{\partial \n} \phi^d |_{\partial \Omega} = 0 \,, \\
	    	u^d |_{t=0} = & 0 \,, \quad \phi^d |_{t=1} = 0 \,.
	    \end{aligned}
	\right.
\end{equation}
By multiplying $u^d$ by the $u^d$-equation and multiplying $\phi^d + \Delta \phi^d$ by the $\phi^d$-equation in \eqref{Diff-eq}, one easily has
\begin{equation*}
	\begin{aligned}
		\tfrac{1}{2} \tfrac{d}{d t} \big( \| u^d \|^2_{L^2_{\rho (\phi_1)}} + \| \phi^d \|^2 + \| \nabla \phi^d \|^2 \big) + \mu \| \nabla u^d \|^2 + \gamma \lambda \big( \| \nabla \phi^d \|^2 + \| \Delta \phi^d \|^2 \big) = R_1 + R_2 + R_3 \,,
	\end{aligned}
\end{equation*}
where
\begin{equation*}
	\begin{aligned}
		R_1 = & \tfrac{1}{2} \langle \partial_t \rho (\phi_1) + \nabla \cdot ( \rho (\phi_1) u_2 ) , |u^d|^2 \rangle + \lambda \langle \nabla \phi^d \otimes \nabla \phi_1 + \nabla \phi_2 \otimes \nabla \phi^d , \nabla u^d \rangle \\
		& - \langle \rho (\phi_1) u^d \cdot \nabla u_1 + [ \rho (\phi_1) - \rho (\phi_2) ] ( \partial_t u_2 + u_2 \cdot \nabla u_2 ) , u^d \rangle \,,
	\end{aligned}
\end{equation*}
and
\begin{equation*}
	\begin{aligned}
		R_2 = - \langle u^d \cdot \nabla \phi_1 + \gamma \lambda [ f' (\phi_1) - f' (\phi_2) ] + \tfrac{\gamma}{2} [\rho' (\phi_1) - \rho' (\phi_2)] |u_1|^2 + \tfrac{\gamma}{2} \rho' (\phi_2) (u_1 + u_2) \cdot u^d , \phi^d \rangle \,,
	\end{aligned}
\end{equation*}
and
\begin{equation*}
	\begin{aligned}
		R_3 = & \langle u^d \cdot \nabla \phi_1 + u_2 \cdot \nabla \phi^d  + \gamma \lambda [ f' (\phi_1) - f' (\phi_2) ] , \Delta \phi^d \rangle \\
		& + \tfrac{\gamma}{2} \langle [\rho' (\phi_1) - \rho' (\phi_2)] |u_1|^2 + \rho' (\phi_2) (u_1 + u_2) \cdot u^d , \Delta \phi^d \rangle \,.
	\end{aligned}
\end{equation*}
Together with the bound \eqref{Energy-Bnd}, the H\"older inequality implies that
\begin{equation*}
	\begin{aligned}
		R_1 \leq & \tfrac{\mu}{2} \| \nabla u^d \|^2 + C \| u^d \|^2_{L^2_{\rho (\phi_1)}} + C \| \phi^d \|^2 + C \| \nabla \phi^d \|^2 \,, \\
		R_2 \leq & C \| u^d \|^2_{L^2_{\rho (\phi_1)}} + C \| \phi^d \|^2 + C \| \nabla \phi^d \|^2 \,, \\
		R_3 \leq & \tfrac{\gamma \lambda}{2} \| \Delta \phi^d \|^2 + C \| u^d \|^2_{L^2_{\rho (\phi_1)}} + C \| \phi^d \|^2 + C \| \nabla \phi^d \|^2 \,.
	\end{aligned}
\end{equation*}
Consequently, we have
\begin{equation*}
	\begin{aligned}
		\tfrac{d}{d t} \mathscr{E}_d (t) \leq C \mathscr{E}_d (t)
	\end{aligned}
\end{equation*}
for $t \in [0, T]$, where $\mathscr{E}_d (t) = \| u^d \|^2_{L^2_{\rho (\phi_1)}} + \| \phi^d \|^2 + \| \nabla \phi^d \|^2 $. Note that $\mathscr{E}_d (0) = 0$. The Gr\"onwall inequality implies
\begin{equation*}
	\begin{aligned}
		\mathscr{E}_d (t) \leq \mathscr{E}_d (0) e^{C t} = 0 \,,
	\end{aligned}
\end{equation*}
which means that $u_1 = u_2$ and $\phi_1 = \phi_2$. Then the first equation in \eqref{Diff-eq} reduces to
\begin{equation*}
	\begin{aligned}
		\nabla p^d = 0 \,.
	\end{aligned}
\end{equation*}
Thus $p_1 = p_2 + C$ for any constant $C$. Namely, the pressure $p$ is unique up to a constant. The proof of Theorem $\ref{th1}$ is finished.

\section{Global stability near $(0, \pm 1)$}\label{Sec:Global}

In this section, we will prove the global classical existence and time decay rate of the ACNS system near the constant equilibrium $(0, \pm 1)$. More precisely, we prove the global solution to \eqref{g-equ-1}-\eqref{g-equ-3} with small initial data \eqref{g-equ-2}. Furthermore, the exponetial decay $e^{- c_\# t}$ of the global solution is also gained. The key point is that the term $ \tfrac{\gamma \lambda}{\eps^2} f' (\phi)$ in the $\phi$-equation of \eqref{equ-1} will generate an additional damping term $ \tfrac{2 \gamma \lambda}{\eps^2} \varphi$ under the perturbation $\phi = \varphi \pm 1$. With this damping structure, one can derive the uniform global-in-times energy estimates of the fluctuated system \eqref{g-equ-1}-\eqref{g-equ-3} withe intial data \eqref{g-equ-2}. The process of deriving the global estimates is similar to the a priori estimates in Section \ref{Sec:Apriori}. Thus we will introduce the energy functional $\mathcal{E}_j (t)$ and dissipation $\mathcal{D}_j (t)$ as follows: for $j \geq 0$,
\begin{equation}\label{GED-j}
	\begin{split}
		\mathcal{E}_j(t) = & \| \partial_t^{j} u \|_{L^2_{\varrho(\varphi)}}^2 + \mu \| \nabla \partial_t^{j} u \|^2 + \mu_0 \| \partial_t^{j} \varphi \|^2 + \mu_1 \| \nabla \partial_t^{j}\varphi \|^2  + \gamma\lambda\| \Delta \partial_t^{j} \varphi \|^2 , \\
		\mathcal{D}_j (t) = & \mu\| \nabla \partial_t^{j} u \|^2 + \| \partial_t^j u_t \|_{L^2_{\varrho(\varphi)}}^2 + \tfrac{ 2 \gamma \lambda }{\varepsilon^2}\| \partial_t^{j} \varphi \|^2 + \mu_2  \| \nabla \partial_t^{j} \varphi \|^2  + \| \partial_t^j \varphi_t \|^2 \\
		& + \gamma \lambda \| \Delta \partial_t^{j} \varphi \|^2 + \| \nabla \partial_t^j \varphi_t \|^2 + \kappa_g \| \Delta \partial_t^{j} u \|^2 + \kappa_g \| \nabla \Delta \partial_t^{j} \varphi \|^2 + \kappa_g \| \partial_t^j p \|^2_1 \,,
	\end{split}
\end{equation}
where
\begin{equation}\label{mu012}
	\begin{aligned}
		\mu_0 = 1 + \tfrac{2 \gamma\lambda}{\varepsilon^2} > 0 \,, \ \mu_1 = 1 + \gamma \lambda + \tfrac{2 \gamma\lambda}{\varepsilon^2} > 0 \,, \ \mu_2 = \gamma \lambda + \tfrac{ 2 \gamma\lambda}{\varepsilon^2} > 0 \,,
	\end{aligned}
\end{equation}
and the small constant $\kappa_g > 0$ in $\mathcal{D}_j (t)$ will be determined later.

\begin{lemma}\label{Lmm-HtDEst}
Let integer $\Lambda \geq 2$. Assume that $(u, p, \phi)$ is the solution to the system \eqref{equ-1}-\eqref{equ-3} on $[0, T]$ constructed in Theorem \ref{th1} and $\phi = \varphi \pm 1$. Then there exist small $\kappa_g > 0$ in $\mathcal{D}_j (t)$ and some positive constants $C > 0$, $\chi_k, \vartheta_k > 0$ $(0 \leq k \leq \Lambda)$, depending only on the all coefficients, $\Lambda$ and $\Omega$, such that
\begin{equation*}
	\begin{aligned}
		\tfrac{d}{d t} \mathfrak{E}_\Lambda (t) + \mathfrak{D}_\Lambda (t) \leq C ( 1 + \mathfrak{E}_\Lambda^\frac{5}{2} (t) ) \mathfrak{E}_\Lambda^\frac{1}{2} (t) \mathfrak{D}_\Lambda (t) \,,
	\end{aligned}
\end{equation*}
where
\begin{equation*}
	\begin{aligned}
		\mathfrak{E}_\Lambda (t) = \sum_{0 \leq k \leq \Lambda} \chi_k \mathcal{E}_k (t) \,, \quad \mathfrak{D}_\Lambda (t) = \sum_{0 \leq k \leq \Lambda} \vartheta_k \mathcal{D}_k (t) \,.
	\end{aligned}
\end{equation*}
\end{lemma}

\begin{proof}
	We will prove the lemma by two steps: 1. $H^2$-estimates for \eqref{g-equ-1}-\eqref{g-equ-2}; 2. $H^2$-estimates for higher order time derivatives.
	
	{\bf Step 1. $H^2$-estimates.} As similar in \eqref{H2-1}, we first Ttake $L^2$-inner product of the first equation of \eqref{g-equ-1} by $ u + u_t$. It thereby follows from integrating by parts over $x \in \Omega$ that
\begin{equation*}
	\begin{split}
		\tfrac{1}{2} \tfrac{d}{dt} ( \| u \|_{ L_{\varrho(\varphi)}^2 }^2 + & \mu \| \nabla u \|^2 )  + \mu \|\nabla u\|^2 + \| u_t \|_{ L_{\varrho(\varphi)}^2 }^2 \\
		= & \underbrace{ \langle \varrho(\varphi) u \cdot \nabla u + \lambda \nabla \cdot ( \nabla\varphi\otimes\nabla\varphi ), - u - u_t  \rangle }_{S_1} \,.
	\end{split}
\end{equation*}
As same in \eqref{H2-2} and \eqref{H2-3}, it infers from taking $L^2$-inner product of the $\phi$-equation in \eqref{g-equ-1} by dot with $\varphi + \varphi_t - \Delta \varphi - \Delta \varphi_t$ and integrating by parts that
\begin{equation*}
	\begin{split}
		& \tfrac{1}{2} \tfrac{d}{dt} \big( \mu_0 \| \varphi \|^2 + \mu_1 \| \nabla \varphi \|^2 + \gamma \lambda \| \Delta \varphi \|^2 \big) \\
		& + \tfrac{2 \gamma\lambda}{\eps^2} \| \varphi \|^2 + \mu_2 \| \nabla \varphi \|^2 + \| \varphi_t \|^2 + \gamma \lambda \| \Delta \varphi \|^2 + \| \nabla \varphi_t \|^2 \\
		& = \underbrace{ \langle u \cdot \nabla \varphi + \gamma \lambda h ( \varphi ) + \gamma \varrho'(\varphi)\tfrac{|u|^2}{2}, -\varphi - \varphi_t +\Delta \varphi \rangle }_{S_2} \\
		& + \underbrace{ \langle \nabla \big[ u \cdot \nabla \varphi + \gamma\lambda h ( \varphi )  + \gamma \varrho'(\varphi)\tfrac{|u|^2}{2} \big], - \nabla \varphi_t \rangle }_{S_3} \,,
	\end{split}
\end{equation*}
where $\mu_0, \mu_1, \mu_2 > 0$ are given in \eqref{mu012}. One thereby has
	\begin{align*}
		& \tfrac{1}{2} \tfrac{d}{dt} ( \| u \|_{ L_{\varrho(\varphi)}^2 }^2 + \mu \| \nabla u \|^2 + \mu_0 \| \varphi \|^2 + \mu_1 \| \nabla \varphi \|^2 + \gamma \lambda \| \Delta \varphi \|^2 ) \\
		& + \mu \|\nabla u\|^2 + \| u_t \|_{ L_{\varrho(\varphi)}^2 }^2 + \tfrac{2 \gamma\lambda}{\eps^2} \| \varphi \|^2 + \mu_2 \| \nabla \varphi \|^2 + \| \varphi_t \|^2 + \gamma \lambda \| \Delta \varphi \|^2 + \| \nabla \varphi_t \|^2\\
		= & S_1 + S_2 + S_3 \,.
	\end{align*}
We consider the term $S_1$. By Lemma \ref{Lmm-Calculus} and the H\"older inequality,
one infers that
\begin{equation}\label{S1}
	\begin{split}
	    S_1 \lesssim & \| u \|^\frac{1}{2} \| \nabla u \|^\frac{5}{2} + \| \nabla u \|^\frac{3}{2} \| \Delta u \|^\frac{1}{2} \|u_t\| + ( \| \nabla \varphi \| + \| \nabla \Delta \varphi \| ) \| \Delta \varphi \| (\|u\| + \|u_t\|) \\
	    \lesssim & ( \mathcal{E}_0^\frac{1}{2} (t) + \mathcal{E}_0 (t) ) \mathcal{D}_0 (t) \,,
	\end{split}
\end{equation}
where $\mathcal{E}_0 (t)$ and $\mathcal{D}_0 (t)$ are defined in \eqref{GED-j}. Similarly, the term $B_2$ can be bounded by
\begin{equation}\label{S2}
	\begin{split}
        S_2 \lesssim & ( \| \nabla \varphi \| + \| \nabla \Delta \varphi \| ) \|u\| (\|\varphi\| + \|\varphi_t\| + \|\Delta \varphi\|) + \|u\|^\frac{1}{2} \|\nabla u\|^\frac{3}{2} (\|\varphi\| + \|\varphi_t\| + \|\Delta \varphi\|) \\
        & + (\|\varphi\|^3 +  \|\nabla \varphi\|^3 + \| \varphi \|^\frac{1}{2} \| \nabla \varphi \|^\frac{3}{2} + \| \varphi\|^2 ) ( \|\varphi\| + \|\varphi_t\| + \|\Delta \varphi\|) \\
        \lesssim & ( \mathcal{E}_0^\frac{1}{2} (t) + \mathcal{E}_0 (t) ) \mathcal{D}_0 (t) \,.
	\end{split}
\end{equation}
Finally, the term $S_3$ can be bounded by
\begin{equation}\label{S3}
	\begin{split}
        S_3 \lesssim & \Big\{ ( \| \nabla \varphi \| + \| \nabla \Delta \varphi \| ) ( \|\nabla u\| + \| \varphi \| + \| u \|^\frac{1}{2} \| \nabla u \|^\frac{3}{2} ) \\
        & + ( \|\nabla \varphi\|^2 + \| \varphi \|^2 )(  \| \Delta \varphi \| + \| \nabla \varphi \|) \\
        & + ( \| \nabla u \|^\frac{3}{2} + \| \nabla u \|^\frac{1}{2}  \| \Delta \varphi \| ) \| \Delta u \|^\frac{1}{2} \Big\} \| \nabla \varphi_t \| \\
        \lesssim & ( \mathcal{E}_0^\frac{1}{2} (t) + \mathcal{E}_0 (t) ) \mathcal{D}_0 (t) \,.
	\end{split}
\end{equation}
Namely,
	\begin{align}\label{H2-t0-glb}
		\no & \tfrac{1}{2} \tfrac{d}{dt} ( \| u \|_{ L_{\varrho(\varphi)}^2 }^2 + \mu \| \nabla u \|^2 + \mu_0 \| \varphi \|^2 + \mu_1 \| \nabla \varphi \|^2 + \gamma \lambda \| \Delta \varphi \|^2 ) \\
		\no & + \mu \|\nabla u\|^2 + \| u_t \|_{ L_{\varrho(\varphi)}^2 }^2 + \tfrac{2 \gamma\lambda}{\eps^2} \| \varphi \|^2 + \mu_2 \| \nabla \varphi \|^2 + \| \varphi_t \|^2 + \gamma \lambda \| \Delta \varphi \|^2 + \| \nabla \varphi_t \|^2 \\
		\lesssim & ( \mathcal{E}_0^\frac{1}{2} (t) + \mathcal{E}_0 (t) ) \mathcal{D}_0 (t) \,.
	\end{align}

Observe that the quantities $\| \Delta u \|^2$ and $\| \nabla \Delta \varphi \|^2$ involved in $\mathcal{D}_0 (t)$ do not occur in the dissipative structures of \eqref{H2-t0-glb}. We thus need to control them by the ADN theory in Lemma \ref{Lmm-Stokes} and the constitive of the equations \eqref{g-equ-1}. More precisely, one has
\begin{equation*}
	\begin{aligned}
		- \mu \Delta u + \nabla p = & V (u, \varphi) \quad \ \textrm{in } \Omega \,, \\
		\nabla \cdot u = & 0 \qquad \qquad \textrm{in } \Omega \,, \\
		u = & 0 \qquad \qquad \textrm{on } \partial \Omega \,,
	\end{aligned}	
\end{equation*}
where
\begin{equation*}
	\begin{aligned}
		V (u, \varphi) = - \varrho (\varphi) ( u_t + u \cdot \nabla u ) - \lambda \nabla \cdot ( \nabla \varphi \otimes \nabla \varphi ) \,.
	\end{aligned}
\end{equation*}
Then the ADN theory in Lemma \ref{Lmm-Stokes} and the similar arguments in \eqref{Delta-u-bnd-1} show that
\begin{equation*}
	\begin{aligned}
		\| \Delta u \|^2 + \| p \|^2_1 \leq & c \| V(u, \varphi) \|^2 \leq c \| u_t \|^2_{L^2_{\varrho (\varphi)}} + C (1 + \| \varphi \|^4_2) \| u \|^2_1 \| \nabla u \|^2_1 + C \| \nabla \varphi \|^2_1 \| \Delta \phi \|^2_1 \\
		\leq & c \| u_t \|^2_{L^2_{\varrho (\varphi)}} + C (1 + \mathcal{E}_0^2 (t) ) \mathcal{E}_0 (t) \mathcal{D}_0 (t)
	\end{aligned}
\end{equation*}
for some constants $c, C > 0$.

Next we dominate the quantity $\| \nabla \Delta \varphi \|^2$. The $\varphi$-equation in \eqref{g-equ-1} indicates that
\begin{equation*}
	\begin{aligned}
		\Delta \varphi = \Psi (u, \varphi) : = \tfrac{1}{\gamma \lambda} ( \varphi_t + u \cdot \nabla \varphi ) + \tfrac{2}{\eps^2} \varphi + h (\varphi) + \tfrac{1}{\lambda} \varrho' (\varphi) \tfrac{|u|^2}{2} \,,
	\end{aligned}
\end{equation*}
which, by the similar arguments in \eqref{ND-phi-1}-\eqref{ND-phi-4}, implies that
\begin{equation*}
	\begin{aligned}
		\| \nabla \Delta \varphi \|^2 = & \| \nabla \Psi (u, \varphi) \|^2 \leq c ( \| \nabla \varphi_t \|^2 + \| \nabla \varphi \|^2 ) + C \| \nabla \phi \|^2_1 \| \nabla u \|^2_1 \\
		& + C (1 + \| \varphi \|^2_2) ( \| \varphi \|^4_2 + \| u \|^2_1 \| \nabla u \|^2_1 ) \\
		\leq & c ( \| \nabla \varphi_t \|^2 + \| \nabla \varphi \|^2 ) + C (1 + \mathcal{E}_0 (t)) \mathcal{E}_0 (t) \mathcal{D}_0 (t)
	\end{aligned}
\end{equation*}
for some positive constants $c, C > 0$. Consequently, one has
\begin{equation}\label{H2-t0-glb-1}
	\begin{aligned}
			\| \Delta u \|^2 + \| p \|^2_1 + \| \nabla \Delta \varphi \|^2 \leq c ( \| u_t \|^2_{L^2_{\varrho (\varphi)}} + \| \nabla \varphi_t \|^2 + \| \nabla \varphi \|^2 ) + C (1 + \mathcal{E}_0^2 (t)) \mathcal{E}_0 (t) \mathcal{D}_0 (t) \,.
	\end{aligned}
\end{equation}

We now take $\kappa_g > 0$ such that $\kappa_g c < \tfrac{1}{2} \min \{ 1, \mu_2 \} $. Therefore, from adding \eqref{H2-t0-glb} to the $\kappa_g$ times of \eqref{H2-t0-glb-1}, it follows that
\begin{equation*}
	\begin{aligned}
		\tfrac{1}{2} \tfrac{d}{dt} ( \| u \|_{ L_{\varrho(\varphi)}^2 }^2 + \mu \| \nabla u \|^2 + \mu_0 \| \varphi \|^2 + \mu_1 \| \nabla \varphi \|^2 + \gamma \lambda \| \Delta \varphi \|^2 ) + \mu \|\nabla u\|^2 + \tfrac{1}{2} \| u_t \|_{ L_{\varrho(\varphi)}^2 }^2 + \tfrac{2 \gamma\lambda}{\eps^2} \| \varphi \|^2 \\
		+ \tfrac{1}{2} \mu_2 \| \nabla \varphi \|^2 +  \| \varphi_t \|^2 + \gamma \lambda \| \Delta \varphi \|^2 + \tfrac{1}{2} \| \nabla \varphi_t \|^2 + \kappa_g \| \Delta u \|^2 + \kappa_g \| p \|^2_1 + \kappa_g \| \nabla \Delta \varphi \|^2 \\
		\lesssim ( 1 + \mathcal{E}_0^\frac{5}{2} (t) ) \mathcal{E}_0^\frac{1}{2} (t) \mathcal{D}_0 (t) \,,
	\end{aligned}
\end{equation*}
which concludes that
\begin{equation}\label{E-0}
	\begin{aligned}
		\tfrac{d}{d t} \mathcal{E}_0 (t) + \mathcal{D}_0 (t) \lesssim ( 1 + \mathcal{E}_0^\frac{5}{2} (t) ) \mathcal{E}_0^\frac{1}{2} (t) \mathcal{D}_0 (t) \,.
	\end{aligned}
\end{equation}

{\bf Step 2. Estimates for higher order time derivatives.} For $k \geq 1$, we apply $\partial_t^k$ to $\eqref{g-equ-1}$ and get
\begin{equation}\label{tk-varphi-u}
	\left\{
	    \begin{aligned}
	    	& \varrho (\varphi) \partial_t^k u_t +\sum_{1 \leq j \leq k} C_k^j \partial_t^j \varrho (\varphi) \partial_t^{k-j} u_t +\partial_t^k ( \varrho (\varphi) u \cdot \nabla u ) + \nabla \partial_t^k p \\
	    	& \qquad \qquad \qquad \qquad \qquad \qquad \qquad \qquad= \mu \Delta \partial_t^k u - \lambda \nabla \cdot \partial_t^k ( \nabla \varphi \otimes \nabla \varphi ) \,, \\
	    	& \qquad \qquad \qquad \qquad \qquad \qquad \nabla \cdot \partial_t^k u = 0 \,, \\
	    	& \partial_t^k \varphi_t + \partial_t^k ( u \cdot \nabla \varphi ) + \tfrac{2 \gamma \lambda}{\eps^2} \partial_t^k \varphi = \gamma \lambda \Delta \partial_t^k \varphi - \gamma \lambda \partial_t^k h (\varphi) - \gamma \partial_t^k ( \varrho' (\varphi) \tfrac{|u|^2}{2} ) \,.
	    \end{aligned}
	\right.
\end{equation}
Moreover, $(\partial_t^k u, \partial_t^k \varphi)$ subjects to the boundary conditions
\begin{equation}\label{BC-tk-varphi-u}
	\begin{aligned}
		\partial_t^k u |_{\partial \Omega} = 0 \,, \quad \tfrac{\partial}{\partial \n} \partial_t^k \varphi |_{\partial \Omega} = 0 \,.
	\end{aligned}
\end{equation}
Then, by employing the similar derivation of \eqref{H2tk-1}, i.e., combining with the boundary conditions \eqref{BC-tk-varphi-u} and taking $L^2$-inner products via dot with $\partial_t^k u + \partial_t^k u_t$ and $\partial_t^k \varphi + \partial_t^k \varphi_t - \Delta \partial_t^k \varphi - \Delta \partial_t^k \varphi_t$ in the first and third equation of \eqref{tk-varphi-u}, respectively, it follows that
\begin{equation*}
	\begin{split}
	    & \tfrac{1}{2} \tfrac{d}{dt} \Big( \| \partial_t^{k} u \|_{L^2_{\varrho(\varphi)}}^2 + \mu \| \nabla \partial_t^{k} u \|^2 + \mu_0 \| \partial_t^{k} \varphi \|^2 + \mu_1 \| \nabla \partial_t^{k}\varphi \|^2 + \gamma\lambda\| \Delta \partial_t^{k} \varphi \|^2 \Big) + \mu\| \nabla \partial_t^{k} u \|^2 \\
	    & + \| \partial_t^k u_t \|_{L^2_{\varrho(\varphi)}}^2 + \tfrac{2 \gamma\lambda}{\varepsilon^2}\| \partial_t^{k} \varphi \|^2 + \mu_2  \| \nabla \partial_t^{k} \varphi \|^2 + \| \partial_t^k \varphi_t \|^2 + \gamma \lambda \| \Delta \partial_t^{k} \varphi \|^2 + \| \nabla \partial_t^k \varphi_t \|^2 \\
	    & =  \underbrace{ \langle \sum_{1 \leq j \leq k} C_k^j \partial_t^j \varrho(\varphi) \partial_t^{k-j} u_t ,  - \partial_t^{k} u - \partial_t^{k} u_t \rangle }_{J_1} + \underbrace{ \langle X_k , - \partial_t^{k} u - \partial_t^k u_t \rangle }_{J_2} \\
        & + \underbrace{ \langle Y_k , - \partial_t^{k} \varphi - \partial_t^k \varphi_t + \Delta \partial_t^{k} \varphi \rangle}_{J_3} + \underbrace{ \langle \nabla Y_k , - \nabla \partial_t^k \varphi_t \rangle}_{J_4} \,,	
	\end{split}
\end{equation*}
where $\mu_0, \mu_1, \mu_2 > 0$ is given in \eqref{mu012}, and
\begin{equation*}
	\begin{aligned}
		X_k : = & \partial_t^{k} ( \varrho(\varphi) u \cdot \nabla u ) + \lambda \partial_t^{k} ( \nabla \varphi \Delta \varphi ) \,, \\
		Y_k : = & \partial_t^{k} ( u \cdot \nabla \varphi ) + \gamma\lambda \partial_t^{k} h (\varphi) + \gamma \partial_t^{k} ( \varrho'(\varphi) \tfrac{|u|^2}{2} ) \,.
	\end{aligned}
\end{equation*}

By the standard Sobolev theory and the similar arguments in \eqref{A1}-\eqref{A3}, it easily follows that
\begin{equation*}
	\begin{aligned}
		J_1 \leq & C \sum_{1 \leq j \leq k} \| \partial_t^j \varrho (\varphi) \|_2 \| \partial_t^{k-j} u_t \|_{L^2 (\varrho (\varphi))} \big( \| \nabla \partial_t^k u \| + \| \partial_t^k u_t \|_{L^2 (\varrho (\varphi))} \big) \\
		\leq & \tfrac{\mu}{8} \| \nabla \partial_t^k u \|^2 + \tfrac{1}{8} \| \partial_t^k u_t \|^2_{L^2 (\varrho (\varphi))} + \Big( 1 + \sum_{0 \leq i \leq k} \| \partial_t^i \varphi \|^2_2 \Big) \sum_{0 \leq j \leq k-1} \| \partial_t^j u_t \|^2_{L^2_{\varrho (\varphi)}} \\
		\leq & \tfrac{\mu}{8} \| \nabla \partial_t^k u \|^2 + \tfrac{1}{8} \| \partial_t^k u_t \|^2_{L^2 (\varrho (\varphi))} + C \sum_{0 \leq j \leq k-1} \mathcal{D}_j (t) + C \sum_{0 \leq j \leq k} \mathcal{E}_j (t) \mathcal{D}_j (t) \,,
	\end{aligned}
\end{equation*}
and
\begin{equation*}
	\begin{aligned}
		J_2 \leq & C \sum_{0 \leq j \leq k} \Big[ (1 + \| \partial_t^j \varphi \|^2_2) \| \partial_t^j u \|_1 \| \nabla \partial_t^j u \|_1 + \| \nabla \partial_t^j \varphi \|_1 \| \Delta \partial_t^j \varphi \|_1 \Big] \big( \| \nabla \partial_t^k u \| + \| \partial_t^k u_t \|_{L^2 (\varrho (\varphi))} \big) \\
		\leq & C \sum_{0 \leq j \leq k} ( 1 + \mathcal{E}_j (t) ) \mathcal{E}_j^\frac{1}{2} (t) \mathcal{D}_j (t) \,,
	\end{aligned}
\end{equation*}
and
\begin{equation*}
	\begin{aligned}
		J_3 \leq & C \sum_{0 \leq j \leq k} ( 1 + \| \partial_t^j \varphi \|_2 ) ( \| \nabla \partial_t^j u \|^2 + \| \partial_t^j \varphi \|^2_2 ) ( \| \partial_t^k \varphi \| + \| \partial_t^k \varphi_t \| + \| \Delta \partial_t^k \varphi \| ) \\
		\leq & C \sum_{0 \leq j \leq k} (1 + \mathcal{E}_j^\frac{1}{2} (t) ) \mathcal{E}_j^\frac{1}{2} (t) \mathcal{D}_j (t) \,,
	\end{aligned}
\end{equation*}
and
\begin{equation*}
	\begin{aligned}
		J_4 \leq & C \sum_{0 \leq j \leq k} \big( \| \nabla \partial_t^j u \| \| \Delta \partial_t^j \varphi \|_1 + \| \nabla \partial_t^j \varphi \|_1 \| \nabla \partial_t^j u \|_1 \big) \| \nabla \partial_t^k \varphi_t \| \\
		& + C \sum_{0 \leq j \leq k} (1 + \| \partial_t^j \varphi \|_2 ) \big( \| \partial_t^j \varphi \|^2_2 + \| \nabla \partial_t^j u \| \| \nabla \partial_t^j u \|_1 \big) \| \nabla \partial_t^k \varphi \| \\
		\leq & C \sum_{0 \leq j \leq k} ( 1 + \mathcal{E}_j^\frac{1}{2} (t) ) \mathcal{E}_j^\frac{1}{2} (t) \mathcal{D}_j (t) \,.
	\end{aligned}
\end{equation*}
Consequently, one has
	\begin{align}\label{H2-tk-glb}
		\no & \tfrac{1}{2} \tfrac{d}{dt} \Big( \| \partial_t^{k} u \|_{L^2_{\varrho(\varphi)}}^2 + \mu \| \nabla \partial_t^{k} u \|^2 + \mu_0 \| \partial_t^{k} \varphi \|^2 + \mu_1 \| \nabla \partial_t^{k}\varphi \|^2 + \gamma\lambda\| \Delta \partial_t^{k} \varphi \|^2 \Big) + \tfrac{7 \mu}{8} \| \nabla \partial_t^{k} u \|^2 \\
		\no & + \tfrac{7}{8} \| \partial_t^k u_t \|_{L^2_{\varrho(\varphi)}}^2 + \tfrac{2 \gamma\lambda}{\varepsilon^2}\| \partial_t^{k} \varphi \|^2 + \mu_2  \| \nabla \partial_t^{k} \varphi \|^2 + \| \partial_t^k \varphi_t \|^2 + \gamma \lambda \| \Delta \partial_t^{k} \varphi \|^2 + \| \nabla \partial_t^k \varphi_t \|^2 \\
		& \leq C \sum_{0 \leq j \leq k-1} \mathcal{D}_j (t) + C \sum_{0 \leq j \leq k} ( 1 + \mathcal{E}_j (t) ) \mathcal{E}_j^\frac{1}{2} (t) \mathcal{D}_j (t) \,.
	\end{align}

Note that the quantities $\| \Delta \partial_t^k u \|^2$ and $\| \nabla \Delta \partial_t^k \varphi \|^2$ involved in $\mathcal{D}_k (t)$ do not occur in the dissipative structures of \eqref{H2-tk-glb}. It thereby is requrired to dominate then by the ADN theory in Lemma \ref{Lmm-Stokes} and the constitive of the equations \eqref{tk-varphi-u}. To be more precise, one has
\begin{equation*}
	\begin{aligned}
		- \mu \Delta \partial_t^k u + \nabla \partial_t^k p = & V_k (u, \varphi) \quad \ \textrm{in } \Omega \,, \\
		\nabla \cdot \partial_t^k u = & 0 \qquad \qquad \ \textrm{in } \Omega \,, \\
		\partial_t^k u = & 0 \qquad \qquad \ \textrm{on } \partial \Omega \,,
	\end{aligned}	
\end{equation*}
where
\begin{equation*}
	\begin{aligned}
		V_k (u, \varphi) = - \partial_t^k \big[ \varrho (\varphi) ( u_t + u \cdot \nabla u ) \big] - \lambda \nabla \cdot \partial_t^k ( \nabla \varphi \otimes \nabla \varphi ) \,.
	\end{aligned}
\end{equation*}
Then the ADN theory in Lemma \ref{Lmm-Stokes} and the similar arguments in \eqref{H2tk-3} indicate that
\begin{equation*}
	\begin{aligned}
		\| \Delta \partial_t^k u \|^2 + \| \partial_t^k p \|_1^2 \leq & c_0 \| V_k (u, \varphi) \|^2 \leq c \| \partial_t^k u_t \|^2_{L^2 (\varrho (\varphi))} + C \sum_{0 \leq j \leq k} \| \nabla \partial_t^j \varphi \|^2_1 \| \Delta \partial_t^j \varphi \|_1^2 \\
		& + C \sum_{0 \leq j \leq k} (1 + \| \partial_t^j \varphi \|_2^4 ) \Big( \| \partial_t^j \varphi \|_2^2 \| \partial_t^j u_t \|^2_{L^2 (\varrho (\varphi))} + \| \nabla \partial_t^j u \|^2 \| \nabla \partial_t^j u \|^2_1 \Big) \\
		\leq & c \| \partial_t^k u_t \|^2_{L^2 (\varrho (\varphi))} + C \sum_{0 \leq j \leq k} ( 1 + \mathcal{E}_j^2 (t) ) \mathcal{E}_j (t) \mathcal{D}_j (t)
	\end{aligned}
\end{equation*}
for some positive constants $c, C > 0$.

Next we consider the quantity $\| \nabla \Delta \partial_t^k \varphi \|^2$. The $\varphi$-equation in \eqref{tk-varphi-u} indicates that
\begin{equation*}
	\begin{aligned}
		\Delta \partial_t^k \varphi = \Psi_k (u, \varphi) : = \tfrac{1}{\gamma \lambda} \partial_t^k ( \varphi_t + u \cdot \nabla \varphi ) + \tfrac{2}{\eps^2} \partial_t^k \varphi + \partial_t^k h (\varphi) + \tfrac{1}{\lambda} \partial_t^k \big[ \varrho' (\varphi) \tfrac{|u|^2}{2} \big] \,,
	\end{aligned}
\end{equation*}
which, by the similar arguments in \eqref{H2tk-5}, implies that
	\begin{align*}
		\| \nabla \Delta \partial_t^k \varphi \|^2 \leq & c (\| \nabla \partial_t^k \varphi \|^2 + \| \nabla \partial_t^k \varphi_t \|^2 ) + C \sum_{0 \leq j \leq k} (1 + \| \partial_t^j \varphi \|_2^2 ) \| \partial_t^j \varphi \|_2^4 \\
		& + C \sum_{0 \leq j \leq k} \Big( (1 + \| \partial_t^j \varphi \|_2^2 ) \| \nabla \partial_t^j u \|^2 \| \nabla \partial_t^j u \|^2_1 \\
		& \qquad \qquad + \| \nabla \partial_t^j \varphi \|^2_1 \| \nabla \partial_t^j u \|_1^2 + \| \nabla \partial_t^j u \|^2 \| \Delta \partial_t^j \varphi \|^2_1 \Big) \\
		\leq & c (\| \nabla \partial_t^k \varphi \|^2 + \| \nabla \partial_t^k \varphi_t \|^2 ) + C \sum_{0 \leq j \leq k} ( 1 + \mathcal{E}_j (t) ) \mathcal{E}_j (t) \mathcal{D}_j (t)
	\end{align*}
for some positive constants $c, C > 0$. As a result, there holds
\begin{equation}\label{H2-tk-glb-1}
	\begin{aligned}
		\| \Delta & \partial_t^k u \|^2 + \| \partial_t^k p \|_1^2 + \| \nabla \Delta \partial_t^k \varphi \|^2 \\
		& \leq c (\| \partial_t^k u_t \|^2_{L^2 (\varrho (\varphi))} + \| \nabla \partial_t^k \varphi \|^2 + \| \nabla \partial_t^k \varphi_t \|^2 ) + C \sum_{0 \leq j \leq k} ( 1 + \mathcal{E}_j^2 (t) ) \mathcal{E}_j (t) \mathcal{D}_j (t) \,.
	\end{aligned}
\end{equation}

We now take $\kappa_g > 0$ such that $\kappa_g c < \tfrac{1}{2} \min \{ \mu_2, \tfrac{3}{4} \}$. Therefore, from adding \eqref{H2-tk-glb} to the $k_g$ times of \eqref{H2-tk-glb-1}, it infers that
\begin{align}
	\no & \tfrac{1}{2} \tfrac{d}{dt} \Big( \| \partial_t^{k} u \|_{L^2_{\varrho(\varphi)}}^2 + \mu \| \nabla \partial_t^{k} u \|^2 + \mu_0 \| \partial_t^{k} \varphi \|^2 + \mu_1 \| \nabla \partial_t^{k}\varphi \|^2 + \gamma\lambda\| \Delta \partial_t^{k} \varphi \|^2 \Big) + \tfrac{7 \mu}{8} \| \nabla \partial_t^{k} u \|^2 \\
	\no & + \tfrac{1}{2} \| \partial_t^k u_t \|_{L^2_{\varrho(\varphi)}}^2 + \tfrac{2 \gamma\lambda}{\varepsilon^2}\| \partial_t^{k} \varphi \|^2 + \tfrac{1}{2} \mu_2  \| \nabla \partial_t^{k} \varphi \|^2 + \| \partial_t^k \varphi_t \|^2 + \gamma \lambda \| \Delta \partial_t^{k} \varphi \|^2 + \tfrac{1}{2} \| \nabla \partial_t^k \varphi_t \|^2 \\
	\no & \leq C \sum_{0 \leq j \leq k-1} \mathcal{D}_j (t) + C \sum_{0 \leq j \leq k} ( 1 + \mathcal{E}_j^\frac{5}{2} (t) ) \mathcal{E}_j^\frac{1}{2} (t) \mathcal{D}_j (t) \,,
\end{align}
which means that
\begin{equation}\label{E-k}
	\begin{aligned}
		\tfrac{d}{d t} \mathcal{E}_k (t) + \mathcal{D}_k (t) \leq C \sum_{0 \leq j \leq k-1} \mathcal{D}_j (t) + C \sum_{0 \leq j \leq k} ( 1 + \mathcal{E}_j^\frac{5}{2} (t) ) \mathcal{E}_j^\frac{1}{2} (t) \mathcal{D}_j (t)
	\end{aligned}
\end{equation}
for any $k \geq 1$. Consequently, by \eqref{E-0} and \eqref{E-k}, one inductively concludes the result in Lemma \ref{Lmm-HtDEst}.
\end{proof}

\noindent{\bf  Proof of Theorem \ref{th2}: global well-posedness with small initial data.}

By Lemma \ref{Lmm-HtDEst}, we know that
\begin{equation}\label{Global-Energy}
	\begin{aligned}
		\tfrac{d}{d t} \mathfrak{E}_\Lambda (t) + \mathfrak{D}_\Lambda (t) \leq C ( 1 + \mathfrak{E}_\Lambda^\frac{5}{2} (t) ) \mathfrak{E}_\Lambda^\frac{1}{2} (t) \mathfrak{D}_\Lambda (t) \,.
	\end{aligned}
\end{equation}
Observe that by the constitive of \eqref{g-equ-1},
\begin{equation*}
	\begin{split}
		\mathfrak{E}_\Lambda (0) & = \sum_{0 \leq k \leq \Lambda} \chi_k \mathcal{E}_k (0) \leq c_0 \big( \| u^{in} \|_{2 \Lambda + 1}^2 + \| \varphi^{in} \|_{2 \Lambda + 2}^2 \big) : = c_0 \mathcal{E}_\Lambda^{in} \le c_0 \upsilon_0
	\end{split}
\end{equation*}
for small $\upsilon_0 \in (0,1)$ to be determined. Then there is a sufficiently small $\upsilon_0 \in (0,1)$ such that if ${\mathcal{E}}^{in}_\Lambda \leq \upsilon_0$, then
\begin{equation}\label{IC-small}
	\begin{aligned}
		C ( 1 + \mathfrak{E}_\Lambda^\frac{5}{2} (0) ) \mathfrak{E}_\Lambda^\frac{1}{2} (0) \leq \tfrac{1}{4} \,.
	\end{aligned}
\end{equation}
Now we define
\begin{equation}\label{Def-global-T}
	T_* = \sup \big\{ \tau \geq 0 ; \sup_{t\in[0,\tau]} C ( 1 + \mathfrak{E}_\Lambda^\frac{5}{2} (t) ) \mathfrak{E}_\Lambda^\frac{1}{2} (t) \leq \tfrac{1}{2} \big\} \geq 0 \,.
\end{equation}
By the continuity of $\mathfrak{E}_\Lambda (t)$ and \eqref{IC-small}, one has $T_* > 0$.

Further claim that $ T_* = + \infty $. Indeed, if $ T_* < + \infty$, then the energy inequality \eqref{Global-Energy} implies that for all $ t \in [0,T_*] $
\begin{equation}\label{Global-Energy-1}
	\tfrac{d}{d t} \mathfrak{E}_\Lambda (t) + \tfrac{1}{2} \mathfrak{D}_\Lambda (t) \leq 0 \,,
\end{equation}
which means
\begin{equation*}
	\sup_{ t \in [0,T_*] } \mathfrak{E}_\Lambda (t) + \int_{0}^{T_*} \mathfrak{D}_\Lambda (t) \d t \leq \mathfrak{E}_\Lambda (0) \leq c_0 \upsilon_0 \,.
\end{equation*}
It therefore follows that
\begin{equation*}
	\sup_{ t \in [0,T_*] } C ( 1 + \mathfrak{E}_\Lambda^\frac{5}{2} (t) ) \mathfrak{E}_\Lambda^\frac{1}{2} (t) \leq C ( 1 + \mathfrak{E}_\Lambda^\frac{5}{2} (0) ) \mathfrak{E}_\Lambda^\frac{1}{2} (0) \leq \tfrac{1}{4} \,.
\end{equation*}
By the continuity of $\mathfrak{E}_\Lambda (t)$, there is a $ t^* > 0$ such that for all $t\in[0,T_* + t^*]$
\begin{equation*}
	C ( 1 + \mathfrak{E}_\Lambda^\frac{5}{2} (t) ) \mathfrak{E}_\Lambda^\frac{1}{2} (t) \leq \tfrac{3}{8} < \tfrac{1}{2} \,,
\end{equation*}
which contradict to the definition of $T_*$ in \eqref{Def-global-T}. Thus $T_* = + \infty$. Consequently, we have
\begin{equation*}
	\begin{aligned}
		\sup_{ t \geq 0 } \mathfrak{E}_\Lambda (t) + \int_{0}^{\infty} \mathfrak{D}_\Lambda (t) \d t \leq \mathfrak{E}_\Lambda (0) \leq c_0 \mathcal{E}_\Lambda^{in} \,.
	\end{aligned}
\end{equation*}

Moreover, since $\partial_t^k u |_{\partial \Omega} = 0$ and $\varrho (\varphi) \thicksim 1$, the Poincar\'e inequality indicates that $\| \partial_t^k u \|^2_{L^2_{\varrho (\varphi)}} \leq C \| \nabla \partial_t^k u \|^2$ for all $k \geq 0$. It thereby follows that
\begin{equation*}
	\begin{aligned}
		\tfrac{1}{2} \mathfrak{D}_\Lambda (t) \geq c_\# \mathfrak{E}_\Lambda (t) \,.
	\end{aligned}
\end{equation*}
Together with \eqref{Global-Energy-1}, it infers that $\tfrac{d}{d t} \mathfrak{E}_\Lambda (t) + c_\# \mathfrak{E}_\Lambda (t) \leq 0$, which means that
\begin{equation}\label{Global-Decay}
	\mathfrak{E}_\Lambda (t) \leq \mathfrak{E}_\Lambda (0) e^{- c_\# t} \leq c_0 \mathcal{E}_\Lambda^{in} e^{- c_\# t} \ \ (\forall \, t \geq 0) \,.
\end{equation}

Furtheremore, by the same arguments in Lemma \ref{Lmm-HSDE}, one has
\begin{equation*}
	\sum_{\ell + s \leq \Lambda} \Big( \| \partial_t^\ell u \|^2_{s+1} + \| \partial_t^\ell p \|^2_s + \| \partial_t^\ell \varphi \|^2_{s+2} \Big) (t) \leq C \big( 1 + \mathfrak{E}_\Lambda^{\aleph_\Lambda} (t) \big) \mathfrak{E}_\Lambda (t) \leq c_1 \mathcal{E}_\Lambda^{in} e^{- c_\# t}
\end{equation*}
for all $t \geq 0$, where $\aleph_\Lambda = \tfrac{11}{3} \cdot 4^\Lambda - \tfrac{8}{3} > 0$. Note that
\begin{equation*}
	\begin{aligned}
		\mathfrak{E}_\Lambda (t) \thicksim \sum_{0 \leq k \leq \Lambda} \mathds{E}_k (t) \,, \quad \mathfrak{D}_\Lambda (t) \thicksim \sum_{0 \leq k \leq \Lambda} \mathds{D}_k (t) \,.
	\end{aligned}
\end{equation*}
Then the proof of Theorem \ref{th2} is finished.
	
\appendix

\section{Proof of Lemma \ref{Lmm-U-Phi}}\label{Sec:Appendix}

The goal of this section is to justify the computations of bounds on $\| U_\ell (u, \phi) \|^2_s$ and $\Phi_\ell (u, \phi) \|^2_s$, namely, to prove Lemma \ref{Lmm-U-Phi}. Later, we will frequently use the following calculus inequalities:
\begin{equation}\label{Calculus-Inq}
	\begin{aligned}
		\| f_1 f_2 \cdots f_n \|_s \lesssim \| f_1 \|_s \| f_2 \|_s \cdots \| f_n \|_s \  (\forall s \geq 2) \,, \ \| f \|_{L^\infty} \lesssim \| f \|_2 \,, \ \| f \|_{L^4} \lesssim \| f \|_1 \,.
	\end{aligned}
\end{equation}

We now start to prove the results in Lemma \ref{Lmm-U-Phi}.

\begin{proof}[Proof of Lemma \ref{Lmm-U-Phi}]
	We will divide the proof into three steps.
	
	{\bf Step 1. $s=0$: To control $\| U_\ell (u, \phi) \|^2$ and $\| \Phi_\ell (u, \phi) \|^2$ for $\ell \geq 0$.}
	
	We first estimate the quantity $\| U_\ell (u, \phi) \|^2$. By the definition of $U_\ell (u, \phi)$ in \eqref{U-Phi-l}, it suffices to estimate the norms $\| \partial_t^\ell [\rho (\phi) u_t] \|^2  $, $\| \partial_t^\ell [\rho (\phi) u \cdot \nabla u ] \|^2$ and $\| \nabla \cdot \partial_t^\ell ( \nabla \phi \otimes \nabla \phi ) \|^2$.
	
	By the second inequality in \eqref{Calculus-Inq} and the H\"older inequality, it follows that
	\begin{equation*}
		\begin{aligned}
		    \| \partial_t^\ell [\rho (\phi) u_t] \|^2 \lesssim & \sum_{a+b=\ell} \| \partial_t^a \rho (\phi) \|^2_{L^\infty} \| \partial_t^{b+1} u \|^2 \lesssim \sum_{a+b=\ell} \| \partial_t^a \rho (\phi) \|^2_2 \| \partial_t^{b+1} u \|^2 \\
		    \lesssim & \sum_{0 \leq j \leq \ell} ( 1 + \| \partial_t^j \phi \|^4_2 ) \| \partial_t^{j+1} u \|^2 \lesssim \sum_{0 \leq j \leq \ell} (1 + E_j^2 (t)) E_{j+1} (t) \\
		    \lesssim & ( 1 + \mathbf{E}^2_{\ell+1} (t) ) \mathbf{E}_{\ell+1} (t) \,,
		\end{aligned}
	\end{equation*}
	where $\mathbf{E}_{\ell+1} (t)$ is defined in \eqref{EUPhi-ls}.
	
	Next, by Lemma \ref{Lmm-Calculus} and the last two inequalities in \eqref{Calculus-Inq},
	\begin{equation}\label{Ul-bnd-1}
		\begin{aligned}
			\| \partial_t^\ell [\rho (\phi) u \cdot \nabla u ] \|^2 \lesssim & \sum_{a+b+c=\ell} \| \partial_t^a \rho (\phi) \|^2_{L^\infty} \| \partial_t^b u \|^2_{L^4} \| \nabla \partial_t^c u \|^2_{L^4} \\
			\lesssim & \sum_{a+b+c=\ell} \| \partial_t^a \rho (\phi) \|^2_2 \| \partial_t^b u \|^2_1 \| \nabla \partial_t^c u \|^\frac{1}{2} \| \Delta \partial_t^c u \|^\frac{3}{2} \\
			\lesssim & \sum_{0 \leq j \leq \ell} (1 + \| \partial_t^j \phi \|^4_2) \| \partial_t^j u \|^\frac{5}{2} \| \partial_t^j u \|^\frac{3}{2}_2 \\
			\lesssim & \sum_{0 \leq j \leq \ell} (1 + E_j^2 (t)) E_j^\frac{5}{4} (t) \| U_j (u, \phi) \|^\frac{3}{2} \\
			\lesssim & \eps_0 \mathbf{U}_{\ell,0} + (1 + \mathbf{E}_\ell^8 (t) ) \mathbf{E}_\ell^5 (t)
		\end{aligned}
	\end{equation}
	for small $\eps_0 > 0$ to be determined, where the last second inequality is derived from \eqref{u-U-bnd}. Here $\mathbf{U}_{\ell,0}$ is defined in \eqref{EUPhi-ls}.
	
	Moreover, from Lemma \ref{Lmm-Calculus}, \eqref{Calculus-Inq} and \eqref{phi-Phi-bnd}, it infers that
	\begin{equation*}
		\begin{aligned}
			\| \nabla \cdot \partial_t^\ell ( \nabla \phi \otimes \nabla \phi ) \|^2 \lesssim & \sum_{a+b=\ell} \| \nabla \partial_t^a \phi \|^2_{L^4} \| \nabla^2 \partial_t^b \phi \|^2_{L^4} \lesssim \sum_{a+b=\ell} \| \nabla \partial_t^a \phi \|^2_1 \| \Delta \partial_t^b \phi \|^\frac{1}{2} \| \nabla \Delta \partial_t^b \phi \|^\frac{3}{2} \\
			\lesssim & \sum_{0 \leq j \leq \ell} E_j^\frac{5}{4} (t) \big( \| \Phi_j (u, \phi) \|_1^\frac{3}{2} + E_j^\frac{3}{4} (t) \big) \lesssim \eps_0 \bm{\Phi}_{\ell, 1} + (1 + \mathbf{E}_\ell^3 (t)) \mathbf{E}_\ell^2 (t) \,.
		\end{aligned}
	\end{equation*}
	
	Collecting the above three bounds, we conclude the first inequality in \eqref{UPhi-l-0} about the quantity $\| U_\ell (u, \phi) \|^2$.
	
	We then estimate the quantity $\| \Phi_\ell (u, \phi) \|^2$. Note that, by the definition of $\Phi_\ell (u, \phi)$ in \eqref{U-Phi-l},
	\begin{equation*}
		\begin{aligned}
			\| \Phi_\ell (u, \phi) \|^2 \lesssim \| \partial_t^{\ell+1} \phi \|^2 + \| \partial_t^\ell (u \cdot \nabla \phi) \|^2 + \| \partial_t^\ell f' (\phi) \|^2 + \| \partial_t^\ell [\rho' (\phi) |u|^2] \|^2 \,.
		\end{aligned}
	\end{equation*}
	By the similar arguments in \eqref{Ul-bnd-1}, one can easily derive that
	\begin{equation*}
		\begin{aligned}
			\| \partial_t^\ell (u \cdot \nabla \phi) \|^2 + \| \partial_t^\ell [\rho' (\phi) |u|^2] \|^2 \lesssim (1 + \mathbf{E}_\ell (t)) \mathbf{E}_\ell^2 (t) \,.
		\end{aligned}
	\end{equation*}
	Recall that $f' (\phi) = \tfrac{1}{\eps^2} (\phi^2 - 1) \phi$. Then the inequalities in \eqref{Calculus-Inq} indicate that
	\begin{equation*}
		\begin{aligned}
			\| \partial_t^\ell f' (\phi) \|^2 \lesssim \| \partial_t^\ell \phi \|^2 + \sum_{0 \leq j \leq \ell} \| \partial_t^j \phi \|^6 \lesssim (1 + \mathbf{E}_\ell^2 (t)) \mathbf{E}_\ell (t) \,.
		\end{aligned}
	\end{equation*}
	Obviously, $\| \partial_t^{\ell+1} \phi \|^2 \lesssim \mathbf{E}_{\ell + 1} (t)$. We summarily obtain
	\begin{equation*}
		\begin{aligned}
			\| \Phi_\ell (u, \phi) \|^2 \lesssim (1 + \mathbf{E}_{\ell+1}^2 (t)) \mathbf{E}_{\ell+1} (t) \,,
		\end{aligned}
	\end{equation*}
	hence, the second inequality in \eqref{UPhi-l-0} holds.
	
	{\bf Step 2. $s=1$: To control $\| U_\ell (u, \phi) \|^2_1$ and $\| \Phi_\ell (u, \phi) \|^2_1$ for $\ell \geq 0$.}
	
	We first estimate the quantity $\| U_\ell (u, \phi) \|^2_1$. Note that, by \eqref{U-Phi-l},
	\begin{equation*}
		\begin{aligned}
			\| \nabla U_\ell (u, \phi) \|^2 \lesssim \| \nabla \partial_t^\ell [\rho (\phi) u_t] \|^2 + \| \nabla \partial_t^\ell [\rho (\phi) u \cdot \nabla u] \|^2 + \| \nabla \cdot \nabla \partial_t^\ell ( \nabla \phi \otimes \nabla \phi ) \|^2 \,.
		\end{aligned}
	\end{equation*}
	By the last two inequality in \eqref{Calculus-Inq},
	\begin{equation*}
		\begin{aligned}
			\| \nabla \partial_t^\ell [\rho (\phi) u_t] \|^2 \lesssim & \| \partial_t^\ell [ \nabla \rho (\phi) u_t ] \|^2 + \| \partial_t^\ell [\rho (\phi) \nabla u_t] \|^2 \\
			\lesssim & \sum_{a+b=\ell} \| \partial_t^a \nabla \rho (\phi) \|^2_{L^4} \| \partial_t^{b+1} u \|^2_{L^4} + \sum_{a+b=\ell} \| \partial_t^a \rho (\phi) \|^2_{L^\infty} \| \nabla \partial_t^{b+1} u \|^2 \\
			\lesssim & \sum_{0 \leq j \leq \ell} (1 + \| \partial_t^j \phi \|^4_2) \| \partial_t^{j+1} u \|^2_1 \lesssim (1 + \mathbf{E}_\ell^2 (t)) \mathbf{E}_{\ell+1} (t) \,.
		\end{aligned}
	\end{equation*}
	Moreover, by the last two inequality in \eqref{Calculus-Inq} and the estimate \eqref{u-U-bnd},
	\begin{equation*}
		\begin{aligned}
			\| \nabla \partial_t^\ell [\rho (\phi) u \cdot \nabla u] \|^2 \lesssim & \sum_{a+b+c=\ell} \Big( \| \partial_t^a \nabla \rho (\phi) \partial_t^b u \cdot \nabla \partial_t^c u \|^2 + \| \partial_t^a \rho (\phi) \nabla \partial_t^b u \cdot \nabla \partial_t^c u \|^2 \Big) \\
			& \qquad \qquad + \sum_{a+b+c=\ell} \| \partial_t^a \rho (\phi) \partial_t^b u \cdot \nabla^2 \partial_t^c u \|^2 \\
			\lesssim & \sum_{0 \leq j \leq \ell} ( 1 + \| \partial_t^j \phi \|^4_2 ) \| \partial_t^j u \|^4_2 \lesssim \sum_{0 \leq j \leq \ell} (1 + E_j^2 (t)) \| U_j (u, \phi) \|^4 \\
			\lesssim & (1 +\mathbf{E}_\ell^2 (t)) \mathbf{U}_{\ell, 0}^2 \,.
		\end{aligned}
	\end{equation*}
	Similarly, one has
	\begin{equation*}
		\begin{aligned}
			\| \nabla \cdot \nabla \partial_t^\ell ( \nabla \phi \otimes \nabla \phi ) \|^2 \lesssim f& \sum_{a+b=\ell} \| \nabla^3 \partial_t^a \phi \|^2 \| \nabla \partial_t^b \phi \|^2_{L^\infty} + \sum_{a+b=\ell} \| \nabla^2 \partial_t^a \phi \|^2_{L^4} \| \nabla^2 \partial_t^b \phi \|^2_{L^4} \\
			\lesssim & \sum_{0 \leq j \leq \ell} \| \partial_t^j \phi \|^4_3 \lesssim \sum_{0 \leq j \leq \ell} \Big( \| \Phi_j (u, \phi) \|^4_1 + \| \partial_t^j \phi \|^4_1 \Big) \\
			\lesssim & \bm{\Phi}^2_{\ell, 1} + \mathbf{E}_\ell^2 (t) \,,
		\end{aligned}
	\end{equation*}
	where the last second inequality is implies by the bound \eqref{phi-Phi-bnd}. Consequently, we obtain
	\begin{equation*}
		\begin{aligned}
			\| \nabla U_\ell (u, \phi) \|^2 \lesssim (1 +\mathbf{E}_\ell^2 (t)) \mathbf{U}_{\ell, 0}^2 + (1 + \mathbf{E}_\ell^2 (t)) \mathbf{E}_{\ell+1} (t) \,.
		\end{aligned}
	\end{equation*}
	Together with the first bound in \eqref{UPhi-l-0}, it follows the validity of the first inequality about $\| U_\ell (u, \phi) \|^2$ in \eqref{UPhi-l-1}.
	
	We then estimate the quantity $\| \Phi_\ell (u, \phi) \|^2_\ell$. By \eqref{H2tk-5} and \eqref{u-U-bnd}-\eqref{phi-Phi-bnd} we know that
	\begin{equation*}
		\begin{aligned}
			\| \nabla \Phi_k (u, \phi) \|^2 \lesssim & \| \nabla \partial_t^{\ell+1} \phi \|^2 + \sum_{0 \leq j \leq \ell} (1 + E_j^2 (t) ) E_j (t) \\
			& + \sum_{0 \leq j \leq \ell} (1 + E_j (t)) E_j^\frac{5}{4} (t) \big( \| \Delta \partial_t^j u \|^\frac{3}{2} + \| \nabla \Delta \partial_t^j \phi \|^\frac{3}{2} \big) \\
			\lesssim & \eps_0 ( \mathbf{U}_{\ell, 0} + \bm{\Phi}_{\ell, 1} ) + (\mathbf{E}_{\ell+1}^2 (t)) \mathbf{E}_{\ell+1} (t) \,.
		\end{aligned}
	\end{equation*}
	Together with the second inequality in \eqref{UPhi-l-0}, we conclude the bound of $\| \Phi_\ell (u, \phi) \|^2_1$ in \eqref{UPhi-l-1}.
	
	{\bf Step 3. $s \geq 2$: To control $\| U_\ell (u, \phi) \|^2_s$ and $\| \Phi_\ell (u, \phi) \|^2_s$ for $\ell \geq 0$.}
	
	We first dominate the quantity $\| U_\ell (u, \phi) \|^2_s$. By \eqref{U-Phi-l},
	\begin{equation*}
		\begin{aligned}
			\| U_\ell (u, \phi) \|^2_s \lesssim \| \partial_t^\ell [\rho (\phi) u_t] \|^2_s + \| \partial_t^\ell [\rho (\phi) u \cdot \nabla u] \|^2_s + \| \nabla \cdot \partial_t^\ell (\nabla \phi \otimes \nabla \phi ) \|^2_s \,.
		\end{aligned}
	\end{equation*}
	Note that, by \eqref{rho-def} and the first inequality in \eqref{Calculus-Inq},
	\begin{equation}\label{rho-s}
		\begin{aligned}
			\| \partial_t^a \rho (\phi) \|_s \lesssim 1 + \sum_{0 \leq j \leq a} \| \partial_t^j \phi \|^2_s \,.
		\end{aligned}
	\end{equation}
	Then, it follows from \eqref{Calculus-Inq} and \eqref{rho-s} that
	\begin{equation*}
		\begin{aligned}
			\| \partial_t^\ell [\rho (\phi) u_t] \|^2_s \lesssim \sum_{a+b=\ell} \| \partial_t^a \rho (\phi) \|^2_s \| \partial_t^{b+1} u \|^2_s \lesssim \sum_{0 \leq j \leq \ell} (1 + \| \partial_t^j \phi \|^4_s) \| \partial_t^{j+1} u \|^2_s \,.
		\end{aligned}
	\end{equation*}
	Similarly, one has
	\begin{equation*}
		\begin{aligned}
			\| \partial_t^\ell [\rho (\phi) u \cdot \nabla u] \|^2_s \lesssim \sum_{0 \leq j \leq \ell} ( 1 + \| \partial_t^j \phi \|^4_s ) \| \partial_t^j u \|^4_{s+1} \,,
		\end{aligned}
	\end{equation*}
	and
	\begin{equation*}
		\begin{aligned}
			\| \nabla \cdot \partial_t^\ell (\nabla \phi \otimes \nabla \phi ) \|^2_s \lesssim \sum_{0 \leq j \leq \ell} \| \nabla \partial_t^{j+1} \phi \|^4_{s+1} \,.
		\end{aligned}
	\end{equation*}
	As a result, there holds
	\begin{equation}\label{Uls-1}
		\begin{aligned}
			\| U_\ell (u, \phi) \|^2_s \lesssim \sum_{0 \leq j \leq \ell} (1 + \| \partial_t^j \phi \|^4_s) \big( \| \partial_t^{j+1} u \|^2_s + \| \partial_t^j u \|^4_{s+1} \big) + \sum_{0 \leq j \leq \ell} \| \nabla \partial_t^{j+1} \phi \|^4_{s+1} \,.
		\end{aligned}
	\end{equation}
	From the estimates \eqref{u-U-bnd} and \eqref{phi-Phi-bnd}, it is deduced that
	\begin{equation}\label{Uls-2}
		\begin{aligned}
			& \| \partial_t^{j+1} u \|^2_s \lesssim \| U_{j+1} (u, \phi) \|^2_{s-2} \,, \quad \quad \ \, \| \partial_t^j u \|^2_{s+1} \lesssim \|U_j (u, \phi) \|^2_{s-1} \,, \\
			& \| \partial_t^j \phi \|^2_s \lesssim \| \Phi_j (u, \phi) \|^2_{s-2} + E_j (t) \,, \quad \| \nabla \partial_t^j \phi \|^2_{s+1} \lesssim \| \Phi_j (u, \phi) \|^2_s + E_j (t) \,.
		\end{aligned}
	\end{equation}
	Then, by \eqref{Uls-1}-\eqref{Uls-2} and \eqref{EUPhi-ls}, we know that the bound of $\| U_\ell (u, \phi) \|^2_s$ in \eqref{UPhi-l-s} holds.
	
	To end the proof, we estimate the quantity $\| \Phi_\ell (u, \phi) \|^2_s$ for $s \geq 2$. Recalling the definition of $\Phi_\ell (u, \phi)$ in \eqref{U-Phi-l}, we know that
	\begin{equation*}
		\begin{aligned}
			\| \Phi_\ell (u, \phi) \|^2_s \lesssim \| \partial_t^{\ell+1} \phi \|^2_s + \| \partial_t^\ell (u \cdot \nabla \phi) \|^2_s + \| \partial_t^\ell f' (\phi) \|^2_s + \| \partial_t^\ell [ \rho' (\phi) |u|^2 ] \|^2_s \,.
		\end{aligned}
	\end{equation*}
	By \eqref{phi-Phi-bnd},
	\begin{equation*}
		\begin{aligned}
			\| \partial_t^{\ell+1} \phi \|^2_s \lesssim \| \Phi_{\ell+1} (u, \phi) \|^2_{s-2} + E_{\ell+1} (t) \,.
		\end{aligned}
	\end{equation*}
	Moreover, it follows from \eqref{Calculus-Inq} and \eqref{u-U-bnd}-\eqref{phi-Phi-bnd} that
	\begin{equation*}
		\begin{aligned}
			\| \partial_t^\ell (u \cdot \nabla \phi) \|^2_s \lesssim & \sum_{a+b=\ell} \| \partial_t^a u \|^2_s \| \nabla \partial_t^b \phi \|^2_s \lesssim \sum_{0 \leq j \leq \ell} \big( \| \partial_t^j u \|^4_s + \| \nabla \partial_t^j \phi \|^4_s \big) \\
			\lesssim & \sum_{0 \leq j \leq \ell} \big( \| U_j (u, \phi) \|^4_{s-2} + \| \Phi_j (u, \phi) \|^4_{s-2} + E_j^2 (t) \big) \,.
		\end{aligned}
	\end{equation*}
	Similarly, it infers that
	\begin{equation*}
		\begin{aligned}
			\| \partial_t^\ell [ \rho' (\phi) |u|^2 ] \|^2_s \lesssim & \sum_{0 \leq j \leq \ell} ( 1 + \| \partial_t^j \phi \|^2_s ) \| \partial_t^j u \|^4_s \\
			\lesssim & \sum_{0 \leq j \leq \ell} \big( 1 + \| \Phi_j (u, \phi) \|^2_{s-2} + E_j (t) \big) \| U_j (u, \phi) \|^4_{s-2} \,.
		\end{aligned}
	\end{equation*}
	Note that $f' (\phi) = \tfrac{1}{\eps^2} (\phi^2 - 1) \phi$. Then \eqref{Calculus-Inq} and \eqref{phi-Phi-bnd} imply that
		\begin{align*}
			\| \partial_t^\ell f' (\phi) \|^2_s \lesssim & \| \partial_t^\ell \phi \|^2_s + \sum_{0 \leq j \leq \ell} \| \partial_t^j \phi \|^6_s \\
			\lesssim & \| \Phi_\ell (u, \phi) \|^2_{s-2} + E_\ell (t) + \sum_{0 \leq j \leq \ell} \big( \| \Phi_j (u, \phi) \|^6_{s-2} + E_j^3 (t) \big) \,.
		\end{align*}
	Consequently, we have
	\begin{equation*}
		\begin{aligned}
			\| \Phi_\ell (u, \phi) \|^2_s \lesssim & \| \Phi_{\ell+1} (u, \phi) \|^2_{s-2} + \sum_{0 \leq j \leq \ell} \big( 1 + \| U_j (u, \phi) \|^4_{s-2} + \| \Phi_j (u, \phi) \|^4_{s-2} \big) \| U_j (u, \phi) \|^2_{s-2} \\
			\lesssim & \sum_{0 \leq j \leq \ell} ( 1 + E_j (t) ) \| U_j (u, \phi) \|^4_{s-2} + \sum_{0 \leq j \leq \ell+1} (1 + E_j^2 (t) ) E_j (t) \,,
		\end{aligned}
	\end{equation*}
	which, combining with \eqref{EUPhi-ls}, concludes the bound of $\| \Phi_\ell (u, \phi) \|^2_s$ in \eqref{UPhi-l-s}. Then the proof of Lemma \ref{Lmm-U-Phi} is finished.
\end{proof}

	
	\section*{Acknowledgments}
	
	The first author N. J. was supported by grants from the National Natural Science Foundation of China under contract No. 11471181 and No. 11731008. The second author Y.-L. L. was supported by grants from the National Natural Science Foundation of China under contract No. 12201220, the Guang Dong Basic and Applied Basic Research Foundation under contract No. 2021A1515110210, and the Science and Technology Program of Guangzhou, China under the contract No. 202201010497.
	
	\bigskip
	
	\bibliography{reference}

\end{document}